\documentclass{article}

\usepackage{amsmath,amsfonts,amssymb,amsthm}
\usepackage{graphicx,subfig}
\usepackage{tabularx}
\usepackage[utf8]{inputenc}
\usepackage{authblk}
\usepackage{xcolor}
\usepackage{verbatim}
\usepackage[title]{appendix}
\usepackage{paralist}
\usepackage{url}
\usepackage{mathtools}
\newtheorem{theorem}{Theorem}
\newtheorem{lemma}{Lemma}

\newtheorem{corollary}{Corollary}
\theoremstyle{definition}
\newtheorem{remark}{Remark}
\newtheorem{definition}{Definition}
\usepackage{enumerate}

\makeatletter
\def\blfootnote{\gdef\@thefnmark{}\@footnotetext}
\makeatother

\renewcommand{\le}{\leqslant}
\renewcommand{\ge}{\geqslant}
\renewcommand{\leq}{\leqslant}
\renewcommand{\geq}{\geqslant}
\renewcommand{\emptyset}{\varnothing}

\newcommand{\cj}{\mathcal{J}}
\newcommand{\tmod}{\ \mathsf{mod}\ }

\newcommand{\real}{\mathbb{R}}
\newcommand{\ints}{\mathbb{Z}}
\newcommand{\natu}{\mathbb{N}}

\newcommand{\bsk}{\boldsymbol{k}}

\newcommand{\bsx}{\boldsymbol{x}}
\newcommand{\bsy}{\boldsymbol{y}}

\newcommand{\bszero}{\boldsymbol{0}}
\newcommand{\bsone}{\boldsymbol{1}}
\newcommand{\bsell}{\boldsymbol{\ell}}

\newcommand{\bskappa}{\boldsymbol{\kappa}}

\newcommand{\rd}{\,\mathrm{d}}

\newcommand{\dunif}{\mathbb{U}}

\newcommand{\e}{\mathbb{E}}
\newcommand{\var}{\mathrm{Var}}
\newcommand{\cov}{\mathrm{Cov}}

\newcommand{\rank}{\mathrm{rank}}

\newcommand{\tran}{\mathsf{T}}

\newcommand{\walk}{\mathrm{wal}_k}
\newcommand{\walbsk}{\mathrm{wal}_{\bsk}}

\newcommand{\appA}{A}
\newcommand{\appB}{B}
\newcommand{\appC}{C}
\newcommand{\appD}{D}
\newcommand{\appE}{E}
\newcommand{\appF}{F}
\newcommand{\appG}{G}
\newcommand{\appH}{H}

\begin{document}

\title{Quasi-Monte Carlo confidence intervals using quantiles of randomized nets}
\author[1,2]{Zexin Pan}
  \date{}
\affil[1]{Johann Radon Institute for Computational and Applied Mathematics,
  \"OAW, Altenbergerstrasse~69, 4040~Linz, Austria.}
\affil[2]{Institute of Fundamental and Transdisciplinary Research, Zhejiang University, 866 Yuhangtang Road, Xihu District, Hangzhou, Zhejiang Province, 310058, China.}
\renewcommand\Affilfont{\footnotesize}
\blfootnote{\noindent Email addresses: \url{zep002@zju.edu.cn}}
\maketitle
\begin{abstract}

Recent advances in quasi-Monte Carlo integration have shown that for linearly scrambled digital net estimators, the convergence rate can be dramatically improved by taking the median rather than the mean of multiple independent replicates. In this work, we demonstrate that the quantiles of such estimators can be used to construct confidence intervals with asymptotically valid coverage for high-dimensional integrals. By analyzing the error distribution for a class of infinitely differentiable integrands, we prove that as the sample size increases, the integration error decomposes into an asymptotically symmetric component and a vanishing remainder. Consequently, the asymptotic error distribution is symmetric about zero, ensuring that a quantile-based interval constructed from independent replicates captures the true integral with probability converging to a nominal level determined by the binomial distribution.
\end{abstract}

\section{Introduction}

Quasi-Monte Carlo (QMC) methods have emerged as a powerful alternative to conventional Monte Carlo (MC) integration \cite{dick:kuo:sloa:2013}.
Like MC, QMC approximates a high-dimensional integral $\mu=\int_{[0,1]^s}f(\bsx)\rd\bsx$ by averaging $n$ function evaluations $\hat\mu=(1/n)\sum_{i=0}^{n-1}f(\bsx_i)$.
Unlike MC, however, QMC replaces independent sampling with carefully constructed point sets $\{\bsx_0, \bsx_1,\dots,\bsx_{n-1}\}$ designed to efficiently explore the integration domain $[0,1]^s$, which corresponds to $s$ uniform random seeds. This paper focuses on digital nets \cite{dick:pill:2010}, a class of point set constructions that enable QMC to mitigate the curse of dimensionality more effectively than classical quadrature rules while achieving convergence rates faster than MC under smoothness assumptions. 

Despite their success, QMC estimators based on digital nets face challenges in error quantification because the function evaluations $f(\bsx_i)$ within a QMC mean $\hat\mu$ are dependent \cite{owen2025errorestimationquasimontecarlo}. Conventional solutions employ randomization techniques to generate independent replicates of QMC means, from which $t$-confidence intervals are constructed. Common choices of randomization are Owen’s nested uniform scrambling \cite{rtms} and Matou\v{s}ek’s random linear scrambling \cite{mato:1998:2}. Loh \cite{loh:2003} establishes the asymptotic normality of Owen-scrambled QMC means under certain restrictions, thereby justifying the use of $t$-intervals. However, even for infinitely differentiable integrands, the variance of these estimators cannot decay faster than $O(n^{-3})$, compromising their computational efficiency.

Random linear scrambling, on the other hand, produces estimators with the same variance as Owen's scrambling \cite{wiart2021dependence}, but exhibits markedly different error behavior. These estimators lack asymptotic normality and instead concentrate within a narrow interval whose length adapts to the smoothness of integrands. Notably, for one-dimensional analytic integrands, the median of linearly scrambled QMC means converges to the target integral at super-polynomial rates $O(n^{-c\log n})$ for a constant $c>0$ \cite{superpolyone}. This result was later generalized to multidimensional analytic integrands \cite{superpolymulti}. For $\alpha$-times differentiable integrands, Pan \cite{pan2024automatic} demonstrates that the median attains nearly optimal convergence rates $O(n^{-\alpha-1/2+\varepsilon})$ for every $\varepsilon>0$ even without prior knowledge of $\alpha$. 

The apparent paradox between fast median convergence and slow variance decay is explained by the presence of outliers \cite[Section 3]{superpolyone}. Owing to their outlier sensitivity, $t$-intervals constructed under linear scrambling often overestimate uncertainty and achieve coverage exceeding the nominal level, as observed empirically in \cite{ci4rqmc}. Quantile-based intervals, while more robust and empirically accurate, currently lack theoretical guarantees on coverage—a critical gap that this work addresses. 
 
Before presenting our results, we  situate our contributions within the context of existing methods for error quantification with QMC. In \cite{10.1145/3643847}, asymptotically valid $t$-intervals are established by letting the number of independent QMC replicates $r$ grow polynomially with the per-replicate sample size $n$. While theoretically sound, this approach requires a total sample size $N=nr= O(n^{1+c}) $ for $ r = O(n^c)$, resulting in suboptimal convergence rates.  An alternative approach by \cite{gobet2022mean} introduces robust estimation techniques to handle non-normal errors, but still requires reliable variance estimation and is ultimately constrained by the slow decay of the variance. Higher-order scrambled digital nets \cite{dick:2011} attain optimal rates under explicit smoothness priors and enable empirically valid $t$-intervals, though rigorous coverage guarantees remain unproven. A recent method by \cite{jain2025empirical} establishes valid finite-sample coverage based on empirical Bernstein bounds, but it requires explicit upper and lower bounds on the integrand and its variance decays no faster than $O(N^{-2})$ given a total sample size $N$. For completely monotone integrands, point sets with non-negative (or non-positive) local discrepancy yield computable upper (or lower) error bounds \cite{10.1093/imaiai/iaae021}, but their convergence rates degrade with the dimension $s$ and become unattractive for $s > 4$. We refer to \cite{owen2025errorestimationquasimontecarlo} for a comprehensive survey. 

Together, these gaps motivate our focus on quantile-based intervals, which achieve both convergence rates that adapt to the smoothness of integrands and, as we will establish, asymptotically valid coverage for a class of infinitely differentiable integrands including $f(\bsx)=\exp(\sum_{j=1}^s x_j)$ and $f(\bsx)=\prod_{j=1}^s (1-x_j/2)^{-1}$. Furthermore, the number of replicates $r$ need not scale with the per-replicate sample size $n$. Our main result, informally stated, is the following theorem. We refer to Theorem~\ref{thm:convergetomedian} and Theorem~\ref{thm:generalver} for precise statements.
\begin{theorem}\label{thm:shortver}
    If the integrand $f$ satisfies the assumptions of Theorem~\ref{thm:step3}, then the QMC estimator $\hat{\mu}_{m,\infty}$ with sample size $n=2^m$ defined in Subsection~\ref{subsec:nets} satisfies
    $$\lim_{m\to\infty}\left(\Pr(\hat{\mu}_{m,\infty}<\mu)+\frac{1}{2}\Pr(\hat{\mu}_{m,\infty}=\mu)\right)=\frac{1}{2}.$$
\end{theorem}
As a consequence, given $r$ independent replicates $\hat{\mu}^{(1)}_{m,\infty},\dots, \hat{\mu}^{(r)}_{m,\infty}$ ordered from smallest to largest,
we may choose indices $1\leq \ell< u\leq r$ and construct the confidence interval $[\hat{\mu}^{(\ell)}_{m,\infty},\hat{\mu}^{(u)}_{m,\infty}]$ for $\mu$. Its asymptotic coverage is guaranteed by Theorem~\ref{thm:shortver} and follows directly from the binomial distribution. Precise expressions are given in Corollary~\ref{cor:confint}.
  
The remainder of this paper is structured as follows. Section~\ref{sec:back} introduces foundational concepts and notation, including the Walsh decomposition framework and properties of Walsh coefficients that are critical to our analysis. Section~\ref{sec:main} proves Theorem~\ref{thm:shortver} for complete random designs, a simplified yet illustrative randomization scheme. After outlining the proof strategy, we systematically address each key component of the argument in Subsections~\ref{subsec:step1}–\ref{subsec:anticoncentration}, and derive crucial corollaries in Subsection~\ref{subsec:main}. Section~\ref{sec:general} extends these results to more general randomization schemes, with random linear scrambling as a central special case. Section~\ref{sec:experiment} empirically validates our theory on two highly skewed integrands. Finally, Section~\ref{sec:discussion} identifies challenges in extending these results to finitely differentiable integrands and concludes the paper with a discussion of open research questions.

\section{Background and notation}\label{sec:back}

Let $\natu = \{1, 2, 3, \dots\}$ denote the natural numbers, $\natu_0 = \natu \cup \{0\}$, and $\natu_*^s = \natu_0^s \setminus \{\boldsymbol{0}\}$ (excluding the zero vector). For $\ell \in \natu$, we define $\ints_\ell = \{0, 1, \dots, \ell-1\}$. When we wish to emphasize its field structure, we write $\ints_2$ as $\mathbb{F}_2$.
 The dimension of the integration domain is $s \in \mathbb{N}$, with $1{:}s = \{1, 2, \dots, s\}$. For a matrix $C$, $C(\ell, :)$ denotes its $\ell$-th row. The indicator function $\bsone\{\mathcal{A}\}$ equals 1 if event $\mathcal{A}$ occurs and 0 otherwise. For a finite set $K$, $|K|$ is its cardinality, and $\dunif(K)$ represents the uniform distribution over $K$. Equality in distribution is written as $X\overset{d}{=} Y$. For asymptotics, $a_m\sim b_m$ denotes $\lim_{m\to\infty}a_m/b_m=1$ and $a_m\sim \sum_{\ell=1}^L b_{m,\ell}$ recursively means $a_m-\sum_{\ell=1}^{L'-1}b_{m,\ell}\sim b_{m,L'}$ for $2\leq L' \leq L$. 

The integrand $f: [0,1]^s \to \mathbb{R}$ has $L^1$-norm $\|f\|_1 = \int_{[0,1]^s} |f(\bsx)|  \rd\bsx$ and $L^\infty$-norm $\|f\|_\infty = \sup_{\bsx \in [0,1]^s} |f(\bsx)|$. Let $C([0,1]^s)$ and $C^\infty([0,1]^s)$ denote the spaces of continuous and infinitely differentiable functions, respectively.  

\subsection{Digital nets and randomization}\label{subsec:nets}

For $m\in \natu$ and $i\in \ints_{2^m}$, let the binary expansion $i=\sum_{\ell=1}^{m}i_\ell 2^{\ell-1}$ be represented by the vector $\vec{i}=\vec{i}[m] = (i_1,\dots,i_m)^\tran\in\{0,1\}^m$. Similarly, for $a\in[0,1)$ and precision $E\in \natu$, we truncate the binary expansion $a=\sum_{\ell=1}^\infty a_\ell 2^{-\ell}$ to $E$ digits, denoted $\vec{a}=\vec{a}[E]=(a_1,\dots,a_E)^\tran\in\{0,1\}^E$. For dyadic rationals (numbers with dual binary expansions), we select the representation terminating in zeros.

Let $s$ matrices $C_j\in\{0,1\}^{E\times m}$ define a base-2 digital net over $[0,1]^s$. The unrandomized points $\bsx_i=(x_{i1},\dots,x_{is})$ are generated by
\begin{align}\label{eq:plainqmc}
\vec{x}_{ij} = C_j\vec{i} \ \tmod 2 \text{ for } i\in \ints_{2^m}, j\in1{:}s,
\end{align}
where $\vec{x}_{ij}\in \{0,1\}^E$ represents $x_{ij}\in [0,1)$ truncated to $E$ digits (trailing digits set to 0). Typically, $E\leq m$ for unrandomized digital nets.

We introduce randomization via
\begin{align}\label{eqn:xequalMCiplusD}
\vec{x}_{ij} = C_j\vec{i} + \vec{D}_j\ \tmod 2  \text{ for } i\in \ints_{2^m}, j\in1{:}s,
\end{align}
where $C_j\in \{0,1\}^{E\times m}$ and $\vec{D}_j\in\{0,1\}^E$ are random with precision $E\ge m$. The vector $\vec{D}_j$ is called the digital shift and consists of independent $\dunif(\{0,1\})$ entries. A widely used method to randomize $C_j$ is the \textbf{random linear scrambling} \cite{mato:1998:2}: 
$$C_j=M_j\mathcal{C}_j\ \tmod 2,$$
where $M_j\in \{0,1\}^{E\times m}$ is a random lower-triangular matrix with ones on the diagonal and $\dunif(\{0,1\})$ entries below, and $\mathcal{C}_j \in \{0,1\}^{m\times m}$ is a fixed generating matrix designed to avoid linear dependencies (see \cite[Chapter 4.4]{dick:pill:2010} for details).

Despite the popularity of random linear scrambling, its dependence on $\mathcal{C}_j$ causes technical difficulties, so we postpone its analysis until Section~\ref{sec:general}. In Section~\ref{sec:main}, we instead use \textbf{complete random designs} \cite{pan2024automatic}, where all entries of $C_j$ are independently drawn from $\dunif(\{0,1\})$. This retains the asymptotic convergence rate of random linear scrambling without requiring pre-designed $\mathcal{C}_j$. Numerically, errors under complete random designs are typically larger than those under random linear scrambling, but the difference diminishes as $m$ increases. 

Let $\{\bsx_{i}[E],i\in \ints_{2^m}\}$ denote points from equation~\eqref{eqn:xequalMCiplusD} with precision $E$. Our QMC estimator for $\mu=\int_{[0,1]^s}f(\bsx)\rd\bsx$ is
\begin{equation}\label{eqn:muEdef}
   \hat\mu_{m,E} = \frac1{2^m}\sum_{i=0}^{2^m-1}f\left(\bsx_i[E]\right). 
\end{equation}
Notice that for random linear scrambling, $ \hat\mu_{m,E}$ implicitly depends on the generating matrices $\{\mathcal{C}_j,j \in 1{:}s\}$. For most of our paper, we conveniently assume $E=\infty$ and focus our analysis on $\hat{\mu}_{m,\infty}$. Practical implementation uses finite $E$, often constrained by the floating point representation in use.  Corollary~\ref{cor:confint} quantifies the required $E$ to ensure the truncation error $|\hat{\mu}_{m,E}-\hat{\mu}_{m,\infty}|$ is negligible.

\subsection{Fourier-Walsh decomposition}

Walsh functions serve as the natural orthonormal basis for analyzing base-2 digital nets. For $k\in\natu_0$ and $x\in[0,1)$, the univariate
Walsh function $\walk(x)$ is defined by
$$
\walk(x) = (-1)^{\vec{k}^\tran\vec{x}},
$$
where $\vec{k}\in \{0,1\}^\infty$ and $\vec{x}\in \{0,1\}^\infty$ are the binary expansions of $k$ and $x$, respectively. Since $\vec{k}$ contains a finite number of nonzero entries, a finite-precision truncation suffices for computation.

For multivariate functions, the $s$-dimensional Walsh function $\walbsk:[0,1)^s\to\{-1,1\}$ is given by the tensor product
$$
\walbsk(\bsx) =\prod_{j=1}^s\mathrm{wal}_{k_j}(x_j)
=(-1)^{\sum_{j=1}^s\vec{k}_j^\tran \vec{x}_j},
$$
where $\bsk=(k_1,\dots,k_s)\in \natu^s_0$. These functions form a complete orthonormal basis for $L^2([0,1]^s)$~\cite{dick:pill:2010}, enabling the Walsh decomposition:
\begin{align}\label{eqn:Walshdecomposition}
f(\bsx) &= \sum_{\bsk\in\natu_0^s}\hat f(\bsk) \walbsk(\bsx),\quad\text{where}\\
\hat f(\bsk) &= \int_{[0,1]^s}f(\bsx)\walbsk(\bsx)\rd\bsx. \nonumber
\end{align}
Equality~\eqref{eqn:Walshdecomposition} holds in the $L^2$ sense. A straightforward calculation using \eqref{eqn:Walshdecomposition}, such as that in \cite[Theorem 1]{superpolymulti}, yields the following error decomposition for QMC estimators:
\begin{lemma}\label{lem:decomp}
For $f\in C^\infty([0,1]^s)$, the error of $\hat{\mu}_{m,\infty}$ defined by equation~\eqref{eqn:muEdef} satisfies
\begin{equation}\label{eqn:errordecomposition}
    \hat{\mu}_{m,\infty}-\mu=\sum_{\bsk\in \natu_*^s}Z(\bsk)S(\bsk)\hat{f}(\bsk),
\end{equation}
where
$$Z(\bsk)=\bsone\Big\{\sum_{j=1}^s \vec{k}_j^\tran  C_j=\bszero \tmod 2\Big\} \quad\text{and}\quad S(\bsk) = (-1)^{\sum_{j=1}^s\vec{k}_j^\tran \vec{D}_j}.$$
\end{lemma}

We note that each $S(\bsk)$ with nonzero $\bsk$ follows a $\dunif(\{-1,1\})$ distribution given $\vec{D}_j$ consists of independent $\dunif(\{0,1\})$ entries. The distribution of $Z(\bsk)$ depends on $m,\bsk$ and the choice of randomization for $C_j$. Under the complete random design assumption, each $Z(\bsk)$ with nonzero $\bsk$ follows a Bernoulli distribution with success probability $2^{-m}$, and the collection $\{Z(\bsk),\bsk\in \natu^s_*\}$ is pairwise independent. Their distribution under more general randomization schemes is analyzed in Section~\ref{sec:general}.

\subsection{Notations involving $\bsk$ and $\bskappa$}\label{subset:kandkappa}
For $k = \sum_{\ell=1}^\infty a_\ell 2^{\ell-1}\in \natu_0$, we define the set of nonzero bits $\kappa = \{\ell\in \natu\mid a_\ell=1\}\subseteq\natu$. The bijection between $k$ and $\kappa$ allows interchangeable use of integer and set notation. In this framework, we can rewrite $Z(\bsk)$ as
\begin{equation}\label{eqn:Zk}
Z(\bsk)=\bsone\Big\{\sum_{j=1}^s \sum_{\ell\in \kappa_j}  C_j(\ell,:)=\bszero \tmod 2\Big\},    
\end{equation}
where $\bsk=(k_1,\dots,k_s)$ and $\kappa_j$ is the nonzero bits of $k_j$.

Next, we define some useful norms on $\bsk$ and $\bskappa$. For a finite subset $\kappa\subseteq \natu$, we denote the cardinality of $\kappa$ as $|\kappa|$, the sum of elements in $\kappa$ as $\Vert\kappa\Vert_1$ and the largest element of $\kappa$ as $\lceil\kappa\rceil$. When $\kappa=\emptyset$, we conventionally define $|\kappa|=\Vert\kappa\Vert_1=\lceil\kappa\rceil=0$. For $\bsk=(k_1,\dots,k_s)$ and the corresponding $\bskappa=(\kappa_1,\dots,\kappa_s)$, we define
$$
\Vert\bsk\Vert_0=\Vert\bskappa\Vert_0=\sum_{j=1}^s|\kappa_j|,
\
\Vert\bsk\Vert_1=\Vert\bskappa\Vert_1 = \sum_{j=1}^s\Vert\kappa_j\Vert_1
\ \text{and}\
\lceil\bsk\rceil=\lceil\bskappa\rceil=\max_{j\in 1{:}s}\lceil\kappa_j\rceil.
$$

In our later analysis, it is helpful to view $\natu^s_0$ as a $\mathbb{F}_2$-vector space. For $\bsk_1=(k_{1,1},\dots,k_{s,1})$ and $\bsk_2=(k_{1,2},\dots,k_{s,2})$, we define the sum of $\bsk_1$ and $\bsk_2$ to be $\bsk_1\oplus\bsk_2=(k^\oplus_1,\dots,k^\oplus_s)$ with $\vec{k}^\oplus_j=\vec{k}_{j,1}+\vec{k}_{j,2} \tmod 2$ for each $j\in 1{:}s$. In other words, each $\kappa^\oplus_j$ is the symmetric difference of $\kappa_{j,1}$ and $\kappa_{j,2}$. We also write $\oplus_{i=1}^r\bsk_i$ for the sum of $\bsk_1, \dots, \bsk_r$. For a finite subset $V\subseteq \natu^s_0$, we define the rank of $V$ as the size of its largest linearly independent subset. We say $V$ has full rank if $\rank(V)=|V|$. One can verify that
$$S(\oplus_{i=1}^r\bsk_i)=\prod_{i=1}^rS(\bsk_i) \text{ for } \bsk_1,\dots,\bsk_r\in \natu^s_0,$$
and $\{S(\bsk),\bsk\in V\}$ are jointly independent random variables if $V$ has full rank.

\subsection{Bounds on Walsh coefficients}
The following lemma relates the Walsh coefficients $\hat{f}(\bsk)$ to the partial derivatives of $f$. For $|\bskappa|=(|\kappa_1|,\dots,|\kappa_s|)\in \natu^s_0$, let
\begin{align*}
f^{|\bskappa|}&=\frac{\partial^{\Vert\bskappa\Vert_0}f}{\partial x_1^{|\kappa_1|}\cdots\partial x_s^{|\kappa_s|}}.
\end{align*}
    \begin{lemma}\label{lem:exactfk}
For $f\in C^\infty([0,1]^s)$, 
\begin{equation}\label{eqn:exactfk}
\hat{f}(\bsk)=(-1)^{\Vert\bskappa\Vert_0}\int_{[0,1]^s}f^{|\bskappa|}(\bsx)\prod_{j=1}^s W_{\kappa_j}(x_j)\rd \bsx,
\end{equation} 
where $W_\kappa :[0,1]\to\mathbb{R}$ for $\kappa\subseteq \natu$ is defined recursively by $W_\emptyset(x)=1$ and  
\begin{equation*}
  W_{\kappa}(x)=\int_{[0,1]}(-1)^{\vec{x}(\lfloor\kappa\rfloor)}W_{\kappa\setminus \lfloor\kappa\rfloor}(x)\rd x  
\end{equation*}
with $\vec{x}(\ell)$ denoting the $\ell$-th bit of $x$ and $\lfloor\kappa\rfloor$ denoting the smallest element of $\kappa$. In particular, $W_\kappa(x)$ for nonempty $\kappa$  is continuous, nonnegative, periodic with period $2^{-\lfloor\kappa\rfloor+1}$ and satisfies
\begin{equation}\label{eqn:Wint}
\int_{[0,1]}W_\kappa(x)\rd x=\prod_{\ell\in\kappa}2^{-\ell-1}\quad \text{and} \quad \max_{x\in [0,1]} W_\kappa(x)= 2\prod_{\ell\in\kappa} 2^{-\ell-1}.
\end{equation}
\end{lemma}
\begin{proof}
Theorem 2.5 of \cite{SUZUKI20161} with $n_j=|\kappa_j|$ implies equation~\eqref{eqn:exactfk}. Properties of $W_\kappa(x)$ are proven in Section 3 of \cite{SUZUKI20161}.
\end{proof}

\begin{corollary}\label{cor:fkbound}
    For $f\in C^\infty([0,1]^s)$, 
    $$|\hat{f}(\bsk)|\leq 2^{-\Vert \bskappa \Vert_1} \Vert f^{|\bskappa|}\Vert_1 .$$
\end{corollary}
\begin{proof}
By equation~\eqref{eqn:Wint},
\begin{align*}
        \Big\Vert\prod_{j=1}^s W_{\kappa_j}(x_j)\Big\Vert_\infty \leq \prod_{j\in 1{:}s,\kappa_j\neq\emptyset} 2\prod_{\ell\in\kappa_j} 2^{-\ell-1}\leq \prod_{j\in 1{:}s} \prod_{\ell\in\kappa_j} 2^{-\ell}=2^{-\Vert \bskappa \Vert_1}.
    \end{align*}
    The result follows after applying  Hölder's inequality to equation~\eqref{eqn:exactfk}.
\end{proof}

\section{Proof of main results}\label{sec:main}

In this section, we prove Theorem~\ref{thm:shortver} for $\hat{\mu}_{m,\infty}$ under complete random designs. The proof strategy is as follows. Given a sequence of subsets $K_m\subseteq \natu^s_*$, we use Lemma~\ref{lem:decomp} to decompose the error $\hat{\mu}_{m,\infty}-\mu$ into two components
\begin{equation}\label{eqn:SUM1}
    \mathrm{SUM}_{1,m}=\sum_{\bsk\in K_m }Z(\bsk)S(\bsk)\hat{f}(\bsk) \quad \text{and} \quad \mathrm{SUM}_{2,m}=\sum_{\bsk\in  \natu_*^s\setminus K_m}Z(\bsk)S(\bsk)\hat{f}(\bsk).
\end{equation}
We further define
\begin{equation}\label{eqn:SUM'1}
    \mathrm{SUM}'_{1,m}=\sum_{\bsk\in K_m}Z(\bsk)S'(\bsk)\hat{f}(\bsk),
\end{equation}
where each $S'(\bsk)$ is independently drawn from $\dunif(\{-1,1\})$. We seek $K_m$ small enough that $\mathrm{SUM}_{1,m}\approx\mathrm{SUM}'_{1,m}$ in distribution, yet large enough that  $|\mathrm{SUM}_{2,m}/\mathrm{SUM}_{1,m}|< 1$ holds with high probability, as made precise in the following lemma:

\begin{lemma}\label{lem:proofoutline}
    For a subset $K_m\subseteq \natu^s_*$, let the sums $\mathrm{SUM}_{1,m}$,$\mathrm{SUM}_{2,m}$ and $\mathrm{SUM}'_{1,m}$ be defined as above. We have the bound
    \begin{multline*}
        \Big|\Pr(\hat{\mu}_{m,\infty}<\mu)+\frac{1}{2}\Pr(\hat{\mu}_{m,\infty}=\mu)-\frac{1}{2}\Big|\\\leq d_{TV}(\mathrm{SUM}_{1,m},\mathrm{SUM}'_{1,m})+ \Pr(|\mathrm{SUM}_{1,m}|\leq |\mathrm{SUM}_{2,m}| ),
    \end{multline*}
        where $d_{TV}(X,Y)=\sup_{A\in \mathcal{B}(\real)}|\Pr(X\in A)-\Pr(Y\in A)|$ is the total variation distance between the distribution of random variables $X$ and $Y$ ($\mathcal{B}(\real)$ denotes the Borel sets of $\real$).
\end{lemma}

The proof is given in Appendix~\appA.

To apply Lemma~\ref{lem:proofoutline}, we define 
\begin{equation*}
   Q_N=\{\bsk\in \natu_*^s\mid \Vert\bskappa\Vert_1 \leq N\}  
\end{equation*}
and set $K_m=Q_{N_m}$ with
\begin{equation}\label{eqn:Nmdef}
    N_m=\sup \{N\in \natu_0\mid |Q_N|\le c_s m 2^m \},
\end{equation}
where $c_s$ is a constant to be specified in equation~\eqref{eqn:cs}.
Notice that $\bszero\notin Q_N$ and $Q_0=\emptyset$. 

Our proof of Theorem~\ref{thm:shortver} contains the following three steps:
\begin{itemize}
    \item \textbf{Step 1}: prove $\lim_{m\to\infty}d_{TV}(\mathrm{SUM}_{1,m},\mathrm{SUM}'_{1,m})=0.$
    \item \textbf{Step 2}: prove $\lim_{m\to\infty}\Pr(|\mathrm{SUM}_{2,m}|\ge T_m)=0$ for a sequence $T_m$ to be specified in Corollary~\ref{cor:SUM2bound}.
    \item \textbf{Step 3}: prove $\lim_{m\to\infty}\Pr(|\mathrm{SUM}'_{1,m}|>T_m)=1$.
\end{itemize}
From Steps 1 and 3 we obtain
$$\lim_{m\to\infty}\Pr(|\mathrm{SUM}_{1,m}|>T_m)=\lim_{m\to\infty}\Pr(|\mathrm{SUM}'_{1,m}|>T_m)=1,$$
and together with Step 2 this yields
$$\lim_{m\to\infty}\Pr(|\mathrm{SUM}_{1,m}|>T_m>|\mathrm{SUM}_{2,m}|)=1.$$
Consequently, both $ d_{TV}(\mathrm{SUM}_{1,m},\mathrm{SUM}'_{1,m})$ and $\Pr(|\mathrm{SUM}_{1,m}|\leq |\mathrm{SUM}_{2,m}| )$ converge to $0$ as $m\to\infty$, and Theorem~\ref{thm:shortver} follows directly from Lemma~\ref{lem:proofoutline}. The next three subsections are devoted to the proofs of Step 1-3.

Before presenting the proofs, we gather some useful facts about $Q_N$ and $N_m$.  Corollary 4 of \cite{superpolymulti} establishes the asymptotic relation
    \begin{equation}\label{eqn:QNasymp}
        |Q_N|\sim \frac{D_s}{N^{1/4}}\exp\Bigl(\pi\sqrt{\frac{sN}{3}}\Bigr),
    \end{equation}
    where $D_s$ is a constant depending on $s$. From equation~\eqref{eqn:QNasymp}, it follows immediately that $\lim_{N\to\infty} |Q_{N+1}|/|Q_{N}|=1$. Together with the defining inequalities $|Q_{N_m}|\leq c_s m 2^m < |Q_{N_m+1}|$, this yields
    \begin{equation}\label{eqn:QNm}
        |Q_{N_m}|\sim c_s m 2^m.
    \end{equation}
Equating the asymptotic expression~\eqref{eqn:QNasymp} with $c_sm 2^m$ and solving for $N_m$ gives
    \begin{equation}\label{eqn:Nm}
    N_m\sim \lambda m^2/s+ 3\lambda m\log_2(m)/s \quad\text{for} \quad\lambda=3(\log 2)^2/\pi^2.
\end{equation}

\subsection{Proof of Step 1}\label{subsec:step1}

We first show the number of summands in $\mathrm{SUM}_{1,m}$ is bounded by $2c_sm$ with high probability.

\begin{lemma}\label{lem:SUM1bound}
Under the complete random design assumption,
    $$\Pr\Big(\sum_{\bsk\in Q_{N_m}}Z(\bsk)\geq 2c_sm\Big)\leq \frac{1}{c_s m}.$$
\end{lemma}
\begin{proof}
First recall that $\Pr(Z(\bsk)=1)=2^{-m}$ and $\{Z(\bsk),\bsk\in Q_{N_m}\}$ are pairwise independent. By Chebyshev's inequality,
    \begin{align*}
        & \Pr\left(\left|\sum_{\bsk\in Q_{N_m}}Z(\bsk)-\frac{|Q_{N_m}|}{2^m}\right|\geq c_s m\right) \\
        \leq  &\frac{1}{c^2_s m^2}\var\left(\sum_{\bsk\in Q_{N_m}}Z(\bsk)\right)
        =  \frac{1}{c^2_s m^2} 2^{-m}(1-2^{-m})|Q_{N_m}|
        \leq  \frac{1}{c^2_s m^2} 2^{-m}|Q_{N_m}|.
    \end{align*}
    Our conclusion then follows from $|Q_{N_m}|\leq c_s m 2^m$.
\end{proof}

Next, we show $Q_N$ contains few additive relations under the operation $\oplus$ defined in Subsection~\ref{subset:kandkappa}. The proof is given in Appendix~\appB.

\begin{lemma}\label{lem:sumk}
Let $N\geq 1$ and $\bsk_1,\dots,\bsk_r$ be sampled independently from $\dunif (Q_N)$. Then there exist positive constants $A_s, B_s$ depending on $s$ such that for all $r\geq 2$
\begin{equation*}
    \Pr\left(\oplus_{i=1}^r\bsk_i\in Q_N\right)\leq A^r_s N^{r/4} r^{-B_s\sqrt{N}}.
\end{equation*}
\end{lemma}

As a consequence, we have the following bound on the cardinality of minimally rank-deficient subsets of $Q_N$.

\begin{lemma}\label{lem:linearindep}
    Let
    \begin{align*}
        &I=\{V\subseteq Q_{N}\mid \rank(V)<|V| \},\\
        &I^*=\{V\in I\mid \text{every proper } W  \subset V \text{ has full rank}\},\\
        &I^*_{r}=I^*\cap \{V\subseteq Q_{N}\mid |V|=r\}.
    \end{align*}
Then with $A_s,B_s$ from Lemma~\ref{lem:sumk}, we have for $r\geq 2$
    $$|I^*_{r+1}|\leq \frac{|Q_{N}|^{r}}{(r+1)!}A^r_s N^{r/4} r^{-B_s\sqrt{N}}.$$
\end{lemma}

\begin{proof}
    Notice that $(r+1)!|I^*_{r+1}|/|Q_{N}|^{r+1}$ is the probability that independent $\bsk_1,\dots,\bsk_{r+1}$ sampled from $\dunif( Q_{N})$ constitute a set $V\in I^*_{r+1}$, which is further bounded by the probability that $\oplus_{i=1}^{r+1}\bsk_i=\bszero$ since all proper subsets $W$ of $V$ have full rank. Because for any given $\bsk_1,...,\bsk_{r}$, there is at most one $\bsk_{r+1}\in Q_N$ for $\oplus_{i=1}^{r+1}\bsk_i=\bszero$, we therefore have
    $$\frac{(r+1)!|I^*_{r+1}|}{|Q_{N}|^{r+1}} \leq \frac{1}{|Q_N|}\Pr\left(\oplus_{i=1}^{r}\bsk_i\in Q_N\right)\leq \frac{1}{|Q_N|}A^r_s N^{r/4} r^{-B_s\sqrt{N}}.$$
    Rearranging yields the claimed bound.
\end{proof}

\begin{theorem}\label{thm:step1}
Set $c_s$ in equation~\eqref{eqn:Nmdef} to be
\begin{equation}\label{eqn:cs}
    c_s=\frac{1}{4}B_s\sqrt{\frac{\lambda}{s}},
\end{equation}
with $\lambda=3(\log 2)^2/\pi^2$ and $B_s$ from Lemma~\ref{lem:sumk}.  Then under the complete random design assumption, there exist constants $d_s,\underline{m}_s$ depending on $s$ such that for $m\geq \underline{m}_s$,
$$d_{TV}(\mathrm{SUM}_{1,m},\mathrm{SUM}'_{1,m})\leq \frac{1}{c_sm}+m^{d_s} 2^{-4c_sm}.$$
\end{theorem}

\begin{proof}
Let $\mathcal{V}=\{\bsk\in Q_{N_m}\mid Z(\bsk)=1\}$. We can rewrite $\mathrm{SUM}_{1,m}$ as
    \begin{align*}
        \mathrm{SUM}_{1,m}=\sum_{V\subseteq Q_{N_m}}\bsone\{\mathcal{V}=V\}\sum_{\bsk\in V}S(\bsk)\hat{f}(\bsk).
    \end{align*}
    When $V=\emptyset$, we conventionally define the sum over $V$ as $0$. Because $\{Z(\bsk),\bsk\in Q_{N_m}\}$ are independent of $\{S(\bsk),\bsk\in Q_{N_m}\}$, the distribution of $\mathrm{SUM}_{1,m}$ is a mixture of $\sum_{\bsk\in V}S(\bsk)\hat{f}(\bsk)$ weighted by $\Pr(\mathcal{V}=V)$.
    A similar argument shows $\mathrm{SUM}'_{1,m}$ is a mixture of $\sum_{\bsk\in V}S'(\bsk)\hat{f}(\bsk)$ weighted by $\Pr(\mathcal{V}=V)$. When $V$ has full rank, $\{S(\bsk),\bsk\in V\}$ are jointly independent and 
    $$\sum_{\bsk\in V}S(\bsk)\hat{f}(\bsk)\overset{d}{=}\sum_{\bsk\in V}S'(\bsk)\hat{f}(\bsk).$$ 
    Let $I_m$ be $I$ from Lemma~\ref{lem:linearindep} with $N=N_m$. We have the bound 
    \begin{align}\label{eqn:dTV}
        d_{TV}(\mathrm{SUM}_{1,m},\mathrm{SUM}'_{1,m})&\leq \sum_{V\in I_m} \Pr(\mathcal{V}=V)=\Pr(\mathcal{V}\in I_m),
    \end{align}
    where we have used the fact that the total variation distance  satisfies the triangular inequality and is bounded by $1$ between any two distributions. By Lemma~\ref{lem:SUM1bound}, we further have
    $$\Pr(\mathcal{V}\in I_m)\leq \Pr(\mathcal{V}\in I_m, |\mathcal{V}|\leq 2c_sm ) + \frac{1}{c_sm}.$$
It remains to bound  $ \Pr(\mathcal{V}\in I_m, |\mathcal{V}|\leq 2c_sm ) $. Let $I^*_m$, $I^*_{m,r}$ be $I^*, I^*_r$ from Lemma~\ref{lem:linearindep} with $N=N_m$.
    When $\mathcal{V}\in I_m$, we can find a subset $\mathcal{W}\subseteq \mathcal{V}$ such that $\mathcal{W} \in I^*_m$. The cardinality $|\mathcal{W}|$ is at least $3$ because a pair of distinct $\bsk_1,\bsk_2 \in Q_N$ must have rank 2. Hence, a union bound argument shows for large enough $m$
    \begin{equation}\label{eqn:VIm}
    \Pr(\mathcal{V}\in I_m, |\mathcal{V}|\leq 2c_sm )  \leq \sum_{r=2}^{\lfloor2c_sm\rfloor} \sum_{W\in I^*_{m,r+1}}\Pr(W\subseteq \mathcal{V}).
    \end{equation}
    Because $W\in I^*_{m,r+1}$ has rank $r$, 
    $$\Pr(W\subseteq \mathcal{V})= \Pr(Z(\bsk)=1 \text{ for all } \bsk\in W)=2^{-mr}.$$
    Then by Lemma~\ref{lem:linearindep}
    \begin{align*}
        \Pr(\mathcal{V}\in I_m, |\mathcal{V}|\leq 2c_sm )  &\leq \sum_{r=2}^{\lfloor2c_sm\rfloor}2^{-mr}|I^*_{m,r+1}|\\\
        &\le\sum_{r=2}^{\lfloor2c_sm\rfloor}2^{-mr}\frac{|Q_{N_m}|^{r}}{(r+1)!}A^r_s N_m^{r/4} r^{-B_s\sqrt{N_m}}\\
        &\leq \sum_{r=2}^{\lfloor2c_sm\rfloor} \frac{(c_s m A_s N_m^{1/4})^r}{(r+1)!}r^{-B_s\sqrt{N_m}},
    \end{align*}
    where we have used $|Q_{N_m}|\le c_sm 2^m$. Finally, equation~\eqref{eqn:Nm} implies $\lambda m^2/s\leq N_m \leq 2\lambda m^2/s$ for large enough $m$ and
    \begin{align*}
        \Pr(\mathcal{V}\in I_m, |\mathcal{V}|\leq 2c_sm ) \leq &\sum_{r=2}^{\lfloor2c_sm\rfloor} \frac{( c_sA_s (2\lambda/s)^{1/4})^r}{(r+1)!} m^{(3/2)r}r^{-mB_s\sqrt{\lambda/s}}\\
        & \leq \exp(c_sA_s (2\lambda/s)^{1/4}) \max_{2\leq r\leq 2c_sm}m^{(3/2)r}r^{-mB_s\sqrt{\lambda/s}}.
    \end{align*}
    Because $m^{(3/2)r}r^{-mB_s\sqrt{\lambda/s}}$ is log-convex in $r$, the maximum is attained at either $r=2$ or $r=2c_sm$. Plugging in equation~\eqref{eqn:cs} yields
    $$\max_{2\leq r\leq 2c_sm}m^{(3/2)r}r^{-mB_s\sqrt{\lambda/s}}=\max(m^3 2^{-4c_sm}, m^{-c_sm} (2c_s)^{-4c_sm}).$$
    The conclusion follows after we choose $d_s>3$ and $\underline{m}_s$ large enough.
\end{proof}

\subsection{Proof of Step 2}\label{subsec:step2}
We assume $f\in C^{\infty}([0,1]^s)$ throughout this subsection. Recall that
$$\mathrm{SUM}_{2,m}=\sum_{\bsk\in  \natu_*^s\setminus Q_{N_m}}Z(\bsk)S(\bsk)\hat{f}(\bsk).$$ 
    Corollary~\ref{cor:fkbound} indicates that the magnitude of $\mathrm{SUM}_{2,m}$ is governed by the growth rate of $\Vert f^{|\bskappa|}\Vert_1$ as $|\bskappa|$ increases. We now present two results corresponding to different growth regimes. The first, simpler case occurs when $\Vert f^{|\bskappa|}\Vert_1$ grows exponentially in $|\bskappa|$. An example is $f(\bsx)=\exp(\sum_{j=1}^s x_j)$.

\begin{theorem}\label{thm:sum2case1}
Assume 
$$\Vert f^{|\bskappa|}\Vert_1\leq K_1\alpha^{\Vert\bskappa\Vert_0}$$
for some positive constants $K_1$ and $\alpha$. Then there exist a constant $D_1$ and a threshold $\underline{m}_1$ depending on $s$ and $\alpha$ such that for all $m\geq \underline{m}_1$
    $$|\mathrm{SUM}_{2,m}|\leq \sum_{\bsk\in  \natu_*^s\setminus Q_{N_m}}|\hat{f}(\bsk)|< K_1 2^{-N_m+D_1 \sqrt{N_m}}.$$
\end{theorem}
\begin{proof}
We follow the strategy used in the proof of \cite[Theorem 2]{superpolymulti}. Corollary~\ref{cor:fkbound} together with our assumption on $f^{|\bskappa|}$ yields
    $$|\hat{f}(\bsk)|\leq K_1 2^{-\Vert\bskappa\Vert_1}\alpha^{\Vert\bskappa\Vert_0}.$$
    The constraint $\bsk\in  \natu_*^s\setminus Q_{N_m}$ implies $\Vert\bskappa\Vert_1>N_m$. By \cite[Theorem 7]{superpolymulti},
    $$|\{\bsk\in \natu_*^s\mid \Vert\bskappa\Vert_1=N\}|\leq \frac{\pi\sqrt{s}}{2\sqrt{3N}}\exp\Bigl(\pi\sqrt{\frac{sN}{3}}\Bigr).$$
    Furthermore, 
     \begin{align}\label{eqn:kappa1vs0}
 \Vert\bskappa\Vert_1&=\sum_{j=1}^s \Vert\kappa_j\Vert_1
    \ge
  \sum_{j=1}^s\frac{|\kappa_j|^2}{2}\geq
  \frac{1}{2s}\biggl(\sum_{j=1}^s|\kappa_j|\biggr)^2=\frac{1}{2s}\Vert\bskappa\Vert_0^2.
\end{align}
Therefore, $\Vert\bskappa\Vert_0\leq \sqrt{2s \Vert\bskappa\Vert_1}$ and
\begin{align*}
    \sum_{\bsk\in  \natu_*^s\setminus Q_{N_m}}|\hat{f}(\bsk)|&\leq \sum_{N=N_m+1}^\infty K_1 2^{-N}\max\Bigl(\alpha^{\sqrt{2s N}},1\Bigr)\frac{\pi\sqrt{s}}{2\sqrt{3N}}\exp\Bigl(\pi\sqrt{\frac{sN}{3}}\Bigr)\\
    &\leq K_1\frac{\pi\sqrt{s}}{2\sqrt{3}}\sum_{N=N_m+1}^\infty 2^{-N+D_\alpha\sqrt{sN}} 
\end{align*}
with $D_\alpha=\sqrt{2}\max(\log_2(\alpha),0)+\pi/(\sqrt{3}\log(2))$. For any $\rho\in (0,1)$, we can find $N_{\rho,s,\alpha}$ such that $D_\alpha\sqrt{s(N+1)}-D_\alpha\sqrt{sN}<\rho$ for $N>N_{\rho,s,\alpha}$. When $m$ is large enough so that $N_m>N_{\rho,s,\alpha}$,
$$\sum_{N=N_m+1}^\infty 2^{-N+D_\alpha\sqrt{sN}}\leq 2^{-N_m+D_\alpha\sqrt{sN_m}}\sum_{N=1}^\infty2^{(\rho-1)N}=2^{-N_m+D_\alpha\sqrt{sN_m}}\frac{2^{\rho-1}}{1-2^{\rho-1}}.$$
By choosing $\rho=1/2$, we obtain for large enough $m$
$$\sum_{\bsk\in  \natu_*^s\setminus Q_{N_m}}|\hat{f}(\bsk)|\leq K_1\frac{ \pi \sqrt{s}}{2\sqrt{3}(\sqrt{2}-1)} 2^{-N_m+D_\alpha\sqrt{sN_m}}.$$
The conclusion follows once we choose a large enough $D_1$.
\end{proof}

A more careful analysis is required when $\Vert f^{|\bskappa|}\Vert_1$ grows factorially in $|\bskappa|$, such as when $f(\bsx)=\prod_{j\in J} (1-x_j/2)^{-1}$ for some $J\subseteq 1{:}s$. The key is to observe that for most $\bsk\in Q_N$, $|\kappa_j|$ is approximately $ 2\sqrt{\lambda N/s}$ in the following sense:

\begin{lemma}\label{lem:meankappa}
    Let $N\geq 1$ and $\bsk$ be sampled from $\dunif (Q_N)$. Then there exist positive constants $A'_s, B'_s$ depending on $s$ such that for any $j\in 1{:}s$ and $\epsilon\in (0,1)$,
    \begin{equation*}
        \Pr\Bigl(\Big|\frac{|\kappa_j|}{\sqrt{\lambda N/s}}-2\Big|>\epsilon\Bigr)\leq A'_sN^{1/4}\exp(-B'_s \epsilon^2 \sqrt{N}) \quad\text{for} \quad\lambda=3(\log 2)^2/\pi^2.
    \end{equation*}
\end{lemma}
The proof is given in Appendix~\appB.

\begin{theorem}\label{thm:sum2case2}
 Assume
$$\Vert f^{|\bskappa|}\Vert_1\leq K_2\alpha^{\Vert\bskappa\Vert_0}\prod_{j\in J} (|\kappa_j|)!$$
for some positive constants $K_2,\alpha$ and some nonempty $J\subseteq 1{:}s$. Then under the complete random design assumption, there exist a constant $d_{s,\alpha}$ depending on $s,\alpha$, a constant $D_2$  depending on $s,\alpha,|J|$ and a threshold $\underline{m}_2$ depending on $s,\alpha,|J|$ such that for $m\geq \underline{m}_2$, 
$$\Pr\Big(|\mathrm{SUM}_{2,m}|\ge K_2 2^{-N_m+(2|J|\log_2(m)+D_2)\sqrt{\lambda N_m/s }} \Big)\leq m^{d_{s,\alpha}}\exp(-c'_s\frac{m}{\log_2(m)^2})$$
with $c'_s=B'_s \sqrt{\lambda/(2s)}$ for $B'_s$ from Lemma~\ref{lem:meankappa}.

\end{theorem}

\begin{proof}
    Corollary~\ref{cor:fkbound} together with our assumption on $f$ yields
    \begin{equation}\label{eqn:fhatinsum2case2}
        |\hat{f}(\bsk)|\leq K_2 2^{-\Vert\bskappa\Vert_1}\alpha^{\Vert\bskappa\Vert_0}\prod_{j\in J} (|\kappa_j|)!.
    \end{equation}
    By equation~\eqref{eqn:kappa1vs0}, $\Vert\bskappa\Vert_0\leq \sqrt{2s \Vert\bskappa\Vert_1}$, so
    $$\prod_{j\in J} (|\kappa_j|)!\leq \Bigl(\sum_{j\in J}|\kappa_j|\Bigr)!\leq (\Vert\bskappa\Vert_0)!\leq (\sqrt{2s \Vert\bskappa\Vert_1})^{\sqrt{2s \Vert\bskappa\Vert_1}}.$$
    Let $N^*_m\geq N_m$ be a new threshold to be specified later. A proof similar to that of Theorem~\ref{thm:sum2case1} shows that  for large $m$
    \begin{align*}
    &\sum_{\bsk\in  \natu_*^s\setminus Q_{N^*_m}}|\hat{f}(\bsk)|\\
    \leq &\sum_{N=N^*_m+1}^\infty K_2 2^{-N}\max\Bigl(\alpha^{\sqrt{2s N}},1\Bigr)\frac{\pi\sqrt{s}}{2\sqrt{3N}}\exp\Bigl(\pi\sqrt{\frac{sN}{3}}\Bigr)(\sqrt{2s N})^{\sqrt{2s N}}\\
    \leq &K_2\frac{\pi\sqrt{s}}{2\sqrt{3}}\sum_{N=N^*_m+1}^\infty  2^{-N+D_\alpha\sqrt{N}} (\sqrt{2sN})^{\sqrt{2s N}}\\
    \leq & K_2\frac{ \pi\sqrt{s}}{2\sqrt{3}(\sqrt{2}-1)} 2^{-N^*_m+D_\alpha\sqrt{N^*_m}}(\sqrt{2sN^*_m})^{\sqrt{2s N^*_m}}.
\end{align*}
Because $N_m\sim \lambda m^2/s+3\lambda m\log_2(m)/s$, we can choose $N^*_m=\lceil N_m+K_3m\log_2(m)\rceil$ with $K_3$ large enough that
$$2^{-N^*_m+D_\alpha\sqrt{N^*_m}}(\sqrt{2sN^*_m})^{\sqrt{2s N^*_m}}\leq 2^{-N_m}$$
for sufficiently large $m$. Then
$$ \Bigl|\sum_{\bsk\in \natu_*^s\setminus Q_{N^*_m} } Z(\bsk)S(\bsk)\hat{f}(\bsk)\Bigr|\leq \sum_{\bsk\in  \natu_*^s\setminus Q_{N^*_m}}|\hat{f}(\bsk)|\leq \frac{K_2}{2}2^{-N_m+2|J|\log_2(m)\sqrt{\lambda N_m/s }}.$$

It remains to show that with high probability
\begin{equation}\label{eqn:K2over2}
    \Bigl|\sum_{\bsk\in Q_{N^*_m}\setminus Q_{N_m} } Z(\bsk)S(\bsk)\hat{f}(\bsk)\Bigr|\le \frac{K_2}{2}2^{-N_m+(2|J|\log_2(m)+D_2)\sqrt{\lambda N_m/s }}
\end{equation}
for some $D_2$. Set $\rho_m=2+\epsilon_m$ for $\epsilon_m\in (0,1)$ to be determined, and define
\begin{equation*}
    \Tilde{Q}=\Bigl\{\bsk\in Q_{N^*_m}\mathrel{\Big|} |\kappa_j|> \rho_m \sqrt{\lambda N^*_m/s} \text{ for some } j \in 1{:}s\Bigr\}.
\end{equation*}
Lemma~\ref{lem:meankappa} together with a union bound argument over $j \in 1{:}s$ gives 
$$\frac{|\Tilde{Q}|}{|Q_{N^*_m}|}\leq sA'_s(N^*_m)^{1/4}\exp(-B'_s \epsilon_m^2 \sqrt{N^*_m}).$$ 
Equation~\eqref{eqn:QNasymp} and $N^*_m=\lceil N_m+K_3m\log_2(m)\rceil$ imply  $|Q_{N^*_m}|\leq m^{K_4}2^m$ for some $K_4$ and large $m$. Hence,
\begin{align}\label{eqn:ZkQtildebound}
 &\Pr(Z(\bsk)=1 \text{ for some }\bsk\in \Tilde{Q})\leq  2^{-m}|\Tilde{Q}| \\ \le & m^{K_4} s A'_s(N^*_m)^{1/4}\exp(-B'_s \epsilon_m^2 \sqrt{N^*_m})
 \le  m^{d_{s,\alpha} }\exp(-c'_s \epsilon_m^2 m), \nonumber
\end{align}
where $c'_s=B's\sqrt{\lambda/(2s)}$ and $d_{s,\alpha} $ is chosen large enough, using $\lambda m^2/(2s)\leq N^*_m \leq 2\lambda m^2/s$ for large $m$.

Thus, with probability at least $1-m^{d_{s,\alpha}} \exp(-c'_s \epsilon_m^2 m)$, every $\bsk$ with $Z(\bsk)=1$ lies outside $\Tilde{Q}$, and
\begin{equation}\label{eqn:removebadk}
    \Bigl|\sum_{\bsk\in Q_{N^*_m}\setminus Q_{N_m} } Z(\bsk)S(\bsk)\hat{f}(\bsk)\Bigr|\leq \sum_{\bsk\in Q_{N^*_m}\setminus(Q_{N_m}\cup\Tilde{Q}) } |\hat{f}(\bsk)|.
\end{equation}
For $\bsk\in  Q_{N^*_m}\setminus\Tilde{Q}$, we have $|\kappa_j|\leq \rho_m \sqrt{\lambda N^*_m/s}$ for all $j\in 1{:}s$, so by equation~\eqref{eqn:fhatinsum2case2}
\begin{align*}
   |\hat{f}(\bsk)|&\le K_2 2^{-\Vert\bskappa\Vert_1}\max\Bigl(\alpha^{s\rho_m \sqrt{\lambda N^*_m/s}},1\Bigr)\Big(\rho_m \sqrt{\lambda N^*_m/s}\Big)^{|J|\rho_m \sqrt{\lambda N^*_m/s}}.
\end{align*}
Because $N_m\sim \lambda m^2/s$ and $N^*_m-N_m\sim K_3m\log_2(m)$, we have $\sqrt{\lambda N^*_m/s}-\sqrt{\lambda N_m/s}\sim K_3\log_2(m)/2  $. Hence, there exists a constant $D^*$ depending on $K_3, |J|,\alpha,s$ such that for large $m$
\begin{align*}
   |\hat{f}(\bsk)|\le K_2 2^{-\Vert\bskappa\Vert_1+ \rho_m|J|  \log_2(m)\sqrt{\lambda N_m/s}+D^* m}.
\end{align*}
Choosing $\epsilon_m=1/\log_2(m)$ for $m\geq 3$ gives
$$\rho_m |J| \log_2(m)\sqrt{\lambda N_m/s}=2 |J| \log_2(m)\sqrt{\lambda N_m/s}+|J|\sqrt{\lambda N_m/s},$$
so
\begin{align}\label{eqn:fkboundoneL}
   |\hat{f}(\bsk)|\le K_2 2^{-\Vert\bskappa\Vert_1+2 |J| \log_2(m)\sqrt{\lambda N_m/s}+D^{**} m}
\end{align}
for some $D^{**} > D^*$. Consequently, when $m$ is sufficiently large,
\begin{align}\label{eqn:sum2case2}
    &\sum_{\bsk\in Q_{N^*_m}\setminus(Q_{N_m}\cup\Tilde{Q}) } |\hat{f}(\bsk)|\\
    \leq&K_2 2^{2|J| \log_2(m)\sqrt{\lambda N_m/s}+D^{**} m}\sum_{N=N_m+1}^{N^*_m} 2^{-N}\frac{\pi \sqrt{s}}{2\sqrt{3N}}\exp\Bigl(\pi\sqrt{\frac{sN}{3}}\Bigr)\nonumber\\
    \leq &K_2 2^{2 |J| \log_2(m)\sqrt{\lambda N_m/s}+D^{**} m}\frac{\pi \sqrt{s}}{2\sqrt{3}(\sqrt{2}-1)}2^{-N_m}\exp\Bigl(\pi\sqrt{\frac{sN_m}{3}}\Bigr). \nonumber
\end{align}
The above bound is asymptotically smaller than the right hand side of equation~\eqref{eqn:K2over2} once we choose a large enough $D_2>D^{**}$, which completes the proof.
\end{proof}

\begin{corollary}\label{cor:SUM2bound}
 Assume for $K,\alpha>0$ and  $\cj=\{J_1,...,J_L\}$ with $J_1,...,J_L\subseteq 1{:}s$, 
 \begin{equation}\label{eqn:fdassumption}
     \Vert f^{|\bskappa|}\Vert_1\leq K\alpha^{\Vert\bskappa\Vert_0}\max_{J\in \cj} \prod_{j\in J} (|\kappa_j|)!
 \end{equation}
 where $\prod_{j\in J} (|\kappa_j|)!=1$ if $J=\emptyset$. Let $\cj_{\max}=\max_{J\in \cj}|J|$. Then under the complete random design assumption, there exist constants $d_{s,\alpha}$ depending on $s,\alpha$, $D_{s,\alpha,\cj}$ depending on $s,\alpha,\cj_{\max}$ and $\underline{m}_{s,\alpha,\cj}$ depending on $s,\alpha,\cj_{\max}$ such that for $m\geq \underline{m}_{s,\alpha,\cj}$, 
$$\Pr(|\mathrm{SUM}_{2,m}|\geq T_m)\leq m^{d_{s,\alpha}}\exp\left(-c'_s \frac{m}{\log_2(m)^2}\right),$$
where $c'_s=B'_s \sqrt{\lambda/(2s)}$ for $B'_s$ from Lemma~\ref{lem:meankappa} and
\begin{equation}\label{eqn:am}
    T_m=K2^{-N_m+2 \cj_{\max}\log_2(m)\sqrt{\lambda N_m/s}+D_{s,\alpha,\cj}m}.
\end{equation}
\end{corollary}
\begin{proof}
The $\cj_{\max}=0$ case follows directly from Theorem~\ref{thm:sum2case1}. When $\cj_{\max}>0$, we notice $N^*_m$ in the proof of Theorem~\ref{thm:sum2case2} does not depend on $\cj$ and equation~\eqref{eqn:removebadk} still holds with probability at least $1-m^{d_{s,\alpha}} \exp(-c'_s\epsilon^2_m m)$ for $\epsilon_m=1/\log_2(m)$. Then similar to equation~\eqref{eqn:fkboundoneL}, we can find $D^{**}_J$ depending on $s,\alpha,|J|$ such that
\begin{align}\label{eqn:fkboundmultiL}
    |\hat{f}(\bsk)|&\le K 2^{-\Vert\bskappa\Vert_1} \max_{J\in \cj} 2^{2 |J| \log_2(m)\sqrt{\lambda N_m/s}+D^{**}_J m}\\
    &\leq K 2^{-\Vert\bskappa\Vert_1} 2^{2 \cj_{\max} \log_2(m)\sqrt{\lambda N_m/s}+(\max_{J\in\cj}D^{**}_J) m}.\nonumber
\end{align}
A calculation similar to equation~\eqref{eqn:sum2case2} gives the desired result.
\end{proof}

\begin{remark}\label{rmk:generalNm}
Corollary~\ref{cor:SUM2bound} can be generalized to a broader choice of $N_m$, as its proof requires only the asymptotic condition $N_m\sim \lambda m^2/s$.
\end{remark}

\begin{remark}
When $f$ is analytic over an open neighborhood of $[0,1]^s$, Proposition 2.2.10 of \cite{krantz2002primer} guarantees that for each $\bsx\in [0,1]^s$, there exist constants $K_{\bsx},\alpha_{\bsx}>0$ depending on $\bsx$ and an open ball $V_{\bsx}$ containing $\bsx$ such that for all multi-indices  $|\bskappa|\in \natu^s_0$,
$$\sup_{\bsy\in V_{\bsx}} |f^{|\bskappa|}(\bsy)|\leq K_{\bsx}\alpha_{\bsx}^{\Vert\bskappa\Vert_0} {\prod_{j=1}^s (|\kappa_j|)!}.$$
By the compactness of $[0,1]^s$, we can choose  $K,\alpha>0$ independent of $\bsx$ that satisfy \eqref{eqn:fdassumption} with $\cj=\{1{:}s\}$. Hence Corollary~\ref{cor:SUM2bound} applies.
\end{remark}

\subsection{Proof of Step 3}\label{subsec:anticoncentration}
Recall that 
$$\mathrm{SUM}'_{1,m}=\sum_{\bsk\in Q_{N_m}}Z(\bsk)S'(\bsk)\hat{f}(\bsk),$$
where each $S'(\bsk)$ is sampled independently from $\dunif(\{-1,1\})$. Our last step is to show $|\mathrm{SUM}'_{1,m}|$ exceeds the threshold $T_m$ from Corollary~\ref{cor:SUM2bound} with high probability. For this we require the following lemma from \cite{erdos:1945}:
\begin{lemma}\label{lem:erdos}
    Let $\{c_i, i\in 1{:}n \}$ be a set of real numbers with $|c_i|\geq 1$ for all $ i\in 1{:}n$ and $\{S'_i, i\in 1{:}n\}$ be independent $\dunif(\{-1,1\})$ random variables. Then 
    $$\sup_{t\in \mathbb{R}}\Pr\Big(\Bigl|\sum_{i=1}^n  c_i S'_i-t\Bigr|\leq 1\Big)\leq \frac{1}{2^n} \binom{n}{\lfloor n/2\rfloor}.$$
\end{lemma}

\begin{theorem}\label{thm:SUM1am}
For $T\geq 0$, define
$$\Lambda(T)=\{\bsk\in \natu^s_*\mid |\hat{f}(\bsk)|\geq T\}.$$
Suppose that $f$ satisfies the assumptions of Corollary~\ref{cor:SUM2bound} and $T_m$ in equation~\eqref{eqn:am} satisfies
 \begin{equation}\label{eqn:Qmam}
 \liminf_{m\to\infty}\frac{|Q_{N_m}\cap \Lambda(T_m)|}{|Q_{N_m}|}>0.
 \end{equation}
 Then under the complete random design assumption, 
\begin{equation*}
    \limsup_{m\to\infty}\sqrt{m}\Pr(|\mathrm{SUM}'_{1,m}|\le T_m)<\infty.
\end{equation*}
\end{theorem}

\begin{proof}
Let $\mathcal{W}=\{\bsk\in Q_{N_m}\cap \Lambda(T_m)\mid Z(\bsk)=1\}$. For large $m$, equations~\eqref{eqn:QNm} and \eqref{eqn:Qmam} imply
\begin{equation}\label{eqn:ewlowerbound}
 \e|\mathcal{W}|=2^{-m}|Q_{N_m}\cap \Lambda(T_m)|\geq c_0 c_s m   
\end{equation}
for some constant $c_0>0$ and $c_s$ from equation~\eqref{eqn:cs}.

Following an argument similar to that in Lemma~\ref{lem:SUM1bound},  Chebyshev’s inequality gives
\begin{align*}
\Pr\left(|\mathcal{W}|\leq \e|\mathcal{W}|/2\right)
\leq  \frac{\var(|\mathcal{W}|)}{\left(\e|\mathcal{W}|-\e|\mathcal{W}|/2\right)^2}\leq \frac{\e|\mathcal{W}|}{(\e|\mathcal{W}|)^2/4}
\leq \frac{4}{\e|\mathcal{W}|}.
\end{align*}
On the event $|\mathcal{W}|\geq \e|\mathcal{W}|/2$, we decompose
$$\mathrm{SUM}'_{1,m}=\sum_{\bsk\in\mathcal{W}}S'(\bsk)\hat{f}(\bsk)+\sum_{\bsk\in \mathcal{V}\setminus \mathcal{W}}S'(\bsk)\hat{f}(\bsk).$$
Conditioned on the $\sigma$-algebra generated by $\{Z(\bsk),\bsk\in Q_{N_m}\}$, we apply Lemma~\ref{lem:erdos} to $\mathrm{SUM}'_{1,m}$ by treating the sum over $\bsk\in \mathcal{V}\setminus \mathcal{W}$ as a shift term, obtaining
\begin{align*}
&\Pr\left(|\mathrm{SUM}'_{1,m}|\le T_m \vert Z(\bsk),\bsk\in Q_{N_m}\right)\\
\leq &\sup_{t\in \mathbb{R} }\Pr\left(\left|\sum_{\bsk\in \mathcal{W} }S'(\bsk)\frac{\hat{f}(\bsk)}{T_m}-t\right|\le 1 \Bigg\vert Z(\bsk),\bsk\in Q_{N_m}\right)
\leq \frac{1}{2^{|\mathcal{W}|}} \binom{|\mathcal{W}|}{\lfloor |\mathcal{W}|/2\rfloor}\leq \frac{C}{\sqrt{|\mathcal{W}|}},
\end{align*}
where the constant $C$ is chosen large enough, as guaranteed by the asymptotic relation $\binom{n}{\lfloor n/2\rfloor}\sim 2^n(\pi n)^{-1/2}$. Therefore,
\begin{align}\label{eqn:step3}
    &\Pr\left(|\mathrm{SUM}'_{1,m}|\le T_m\right)\\
    \leq &\Pr\left(|\mathrm{SUM}'_{1,m}|\le T_m,|\mathcal{W}|> \e|\mathcal{W}|/2\right) +\Pr\left(|\mathcal{W}|\leq \e|\mathcal{W}|/2\right)
    \leq  \frac{C}{\sqrt{\e|\mathcal{W}|/2}}+\frac{4}{\e|\mathcal{W}|},\nonumber 
\end{align}
which, together with equation~\eqref{eqn:ewlowerbound}, yields our conclusion.
\end{proof}

The following theorem gives a sufficient condition for equation~\eqref{eqn:Qmam}. Informally, $f$ must be "nondegenerate" in the sense that sufficiently many $\bsk\in Q_{N_m}$ have $|\hat{f}(\bsk)|$ comparable to their upper bounds in equation~\eqref{eqn:fkboundmultiL} up to an exponential factor in $m$.

\begin{theorem}\label{thm:step3}
For $\beta>0$ and a collection $\cj=\{J_1,\dots,J_L\}$ of subsets of $1{:}s$, define
            \begin{equation*}
        Q_{N,c,\beta,\cj}(f)=\Bigl\{\bsk\in Q_N \Big\vert |\hat{f}(\bsk)|\geq  c 2^{-\Vert\bskappa\Vert_1}\beta^{\Vert\bskappa\Vert_0}\max_{J\in \cj} \prod_{j\in J} (|\kappa_j|)!\Bigr\}
            \end{equation*}
            and
\begin{align*}
  \mathcal{F}_{\cj}=\Big\{f\in C^{\infty}([0,1]^s) \Big\vert \sup_{c>0}\sup_{\beta>0}\liminf_{N\to\infty}\frac{|Q_{N,c,\beta,\cj}(f)|}{|Q_{N}|} >0  \Big\}.  
\end{align*}
Then equation~\eqref{eqn:Qmam} holds if $f\in C^{\infty}([0,1]^s)$ satisfies equation~\eqref{eqn:fdassumption} for some $K,\alpha>0$ and some collection $\cj=\{J_1,\dots,J_L\}$, such that $f\in \mathcal{F}_{\cj}$.
\end{theorem}

The proof is given in Appendix~\appC.

\begin{remark}
    The definition of $\mathcal{F}_{\cj}$ may appear nonstandard. Notably, $\mathcal{F}_{\cj}$ is not convex and excludes the zero function. However, Clancy et al. \cite{CLANCY201421} argue that conical input function spaces are preferred over convex ones in the context of adaptive confidence intervals, and $\mathcal{F}_{\cj}$ is by design conical. More generally, some nondegenerate assumption is needed to exclude constant integrands, for which $\hat{\mu}_{m,\infty}-\mu=0$.
    See also \cite[Theorem 2]{loh:2003} and \cite[Theorem 2]{basu2017asymptotic} for analogous nondegenerate conditions used to establish asymptotic normality of Owen-scrambled QMC means.
\end{remark}

\begin{remark}\label{remark:assump}
As an example of a function not belonging $\mathcal{F}_{\cj}$, consider $f$ satisfying $f^{|\bskappa|}=0$ whenever $|\kappa_1|$ (the number of differentiation with respect to $x_1$) exceeds a threshold $\underline{\kappa}\in\natu_0$. Lemma~\ref{lem:exactfk} then implies $\hat{f}(\bsk)=0$ whenever $|\kappa_1|>\underline{\kappa}$ and Lemma~\ref{lem:meankappa} shows only an exponentially small fraction of $\bsk\in Q_N$ have nonzero $\hat{f}(\bsk)$, so $|Q_{N,c,\beta,\cj}(f)|/|Q_{N}|$ converges to $0$ regardless of the choice of $c,\beta$ and $\cj$.

In this case, however, $f$ admits the representation
$$f=\sum_{p=0}^{\underline{\kappa}}g_p(\bsx_{-1})x_1^p$$
with $\bsx_{-1}=(x_2,\dots,x_s)$, so one can integrate $f$ along the $x_1$ direction and apply our algorithm to $\sum_{p=0}^{\underline{\kappa}}g_p(\bsx_{-1})/(p+1)$ instead. This technique, known as pre-integration in the QMC literature, serves to regularize the integrand. See \cite{grie:kuo:leov:sloa:2018, Sifan:2023} for additional examples.

Another solution is to localize our calculation to $Q'=\{\bsk\in \natu_*^s\mid  |\kappa_1| \leq  \underline{\kappa}\}$. Specifically, we set $K_m=Q_{N_m}\cap Q'$ with $N_m=\sup \{N\in \natu_0\mid |Q_N \cap Q'|\le c'_s m 2^m \}$ for a suitable  $c'_s>0$. Repeating the proof strategy of Steps 1–3, we can establish the corresponding results when $f$ satisfies equation~\eqref{eqn:fdassumption} with a collection $\cj$ of subsets of $2{:}s$, such that
\begin{equation*}
\sup_{c>0}\sup_{\beta>0}\liminf_{N\to\infty}\frac{|Q_{N,c,\beta,\cj}(f)|}{|Q_{N}\cap Q'|} >0.
    \end{equation*}

    These two arguments extend naturally to the case where, for a subset $u\subseteq 1{:}s$ and thresholds $\{\underline{\kappa}_j\in\natu_0, j\in u\}$,  we have $f^{|\bskappa|}=0$ whenever $|\kappa_j|> \underline{\kappa}_j$ for some $j\in u$. Whether our analysis can be extended to functions $f\notin\mathcal{F}_{\cj}$ not covered by the above cases is an open question that we leave for future investigation.
\end{remark}

\begin{remark}
    It is straightforward to verify the assumptions of Theorem~\ref{thm:step3} when each derivative $f^{|\bskappa|}(\bsx)$ does not change sign over $[0,1]^s$. In this case, equations~\eqref{eqn:exactfk} and \eqref{eqn:Wint} imply
    \begin{align*}
       |\hat{f}(\bsk)|\geq &\Big(\inf_{\bsx\in[0,1]^s}|f^{|\bskappa|}(\bsx)|\Big)\int_{[0,1]^s}\prod_{j=1}^s W_{\kappa_j}(x_j)\rd\bsx \\
       =& \Big(\inf_{\bsx\in[0,1]^s}|f^{|\bskappa|}(\bsx)|\Big) \prod_{j\in 1{:}s,\kappa_j\neq\emptyset} \prod_{\ell\in\kappa_j} 2^{-\ell-1} \\
       = & \Big(\inf_{\bsx\in[0,1]^s}|f^{|\bskappa|}(\bsx)|\Big) 2^{-\Vert \bskappa \Vert_1-\Vert \bskappa \Vert_0}.
    \end{align*}
    For $f$ to belong to $\mathcal{F}_{\cj}$, it therefore suffices for
   \begin{equation}\label{eqn:fdinf}
       \inf_{\bsx\in[0,1]^s}|f^{|\bskappa|}(\bsx)|\geq c_0 \beta^{\Vert\bskappa\Vert_0}\max_{J\in \cj}\prod_{j\in J} (|\kappa_j|)!
   \end{equation}
   to hold for some constants $c_0,\beta>0$. In particular, simple integrands such as  $f(\bsx)=\exp(\sum_{j=1}^s x_j)$ and $f(\bsx)=\prod_{j=1}^s (1-x_j/2)^{-1}$ satisfy Theorem~\ref{thm:step3} by this criterion.

The above argument also suggests a regularization strategy in which we add a function with sufficiently large positive derivatives to $f$ to make the sum meet the conditions of Theorem~\ref{thm:step3}. For instance, if  $f$ satisfies equation~\eqref{eqn:fdassumption} with $\cj=\{\emptyset\}$ and some $K,\alpha>0$, then for any $K'>K$ the sum $f(\bsx)+K' \exp(\alpha\sum_{j=1}^s x_j)$ satisfies equation~\eqref{eqn:fdinf} with $c_0=K'-K$ and $\beta=\alpha$. This regularization is not practically useful, however, because choosing suitable $K'$ and $\alpha$ requires information about the derivatives of $f$. Moreover, the error incurred by integrating $K' \exp(\alpha\sum_{j=1}^s x_j)$ may dominate that of $f$,  leading to unnecessarily wide confidence intervals.

Formulating easily verifiable conditions that allow $f^{|\bskappa|}(\bsx)$ to change sign over $[0,1]^s$ is another open question we leave for future research.
\end{remark}

\subsection{Main results}\label{subsec:main}
As outlined at the beginning of Section~\ref{sec:main}, Steps 1-3 provide all the ingredients needed for the proof of Theorem~\ref{thm:shortver}. Moreover, our analysis yields a quantitative guarantee on the rate at which the quantile of $\mu$ converges to $1/2$.
\begin{theorem}\label{thm:convergetomedian}
    If $f\in C^{\infty}([0,1]^s)$ satisfies the assumptions of Theorem~\ref{thm:step3}, then under the complete random design assumption,
    \begin{equation*}
    \limsup_{m\to\infty} \sqrt{m}\Big|\Pr(\hat{\mu}_{m,\infty}<\mu)+\frac{1}{2}\Pr(\hat{\mu}_{m,\infty}=\mu)-\frac{1}{2}\Big|<\infty.
    \end{equation*}
\end{theorem}
\begin{proof}
    By  $|\Pr(|\mathrm{SUM}_{1,m}|\leq T_m)-\Pr(|\mathrm{SUM}'_{1,m}|\leq T_m)|\leq d_{TV}(\mathrm{SUM}_{1,m},\mathrm{SUM}'_{1,m})$ and Lemma~\ref{lem:proofoutline},
    \begin{align*}
       & \Big|\Pr(\hat{\mu}_{m,\infty}<\mu)+\frac{1}{2}\Pr(\hat{\mu}_{m,\infty}=\mu)-\frac{1}{2}\Big|\\
        \leq & d_{TV}(\mathrm{SUM}_{1,m},\mathrm{SUM}'_{1,m})+ \Pr(|\mathrm{SUM}_{1,m}|\leq |\mathrm{SUM}_{2,m}| )\\
        \leq & d_{TV}(\mathrm{SUM}_{1,m},\mathrm{SUM}'_{1,m})+ \Pr(|\mathrm{SUM}_{2,m}|\ge T_m)+\Pr(|\mathrm{SUM}_{1,m}|\leq T_m)\\
        \leq & 2d_{TV}(\mathrm{SUM}_{1,m},\mathrm{SUM}'_{1,m})+ \Pr(|\mathrm{SUM}_{2,m}|\ge T_m)+\Pr(|\mathrm{SUM}'_{1,m}|\leq T_m).
    \end{align*}
    Theorem~\ref{thm:step1} proves 
    $$\lim_{m\to\infty}\sqrt{m}d_{TV}(\mathrm{SUM}_{1,m},\mathrm{SUM}'_{1,m})= 0.$$ 
    Corollary~\ref{cor:SUM2bound} proves 
    $$\lim_{m\to\infty}\sqrt{m}\Pr(|\mathrm{SUM}_{2,m}|\ge T_m)= 0.$$ Theorem~\ref{thm:SUM1am} together with Theorem~\ref{thm:step3} proves $$\limsup_{m\to\infty} \sqrt{m} \Pr(|\mathrm{SUM}'_{1,m}|\leq T_m)<\infty.$$ 
    Our conclusion follows after combining the above results.
\end{proof}

The next corollary establishes that sample quantiles of $\hat{\mu}_{m,E}$ yield asymptotically valid confidence intervals for $\mu$, provided the precision $E$ grows with $m$ at a suitable rate.

\begin{corollary}\label{cor:confint}
Fix an integer $r\geq 1$ and let $\hat{\mu}^{(1)}_{m,E_m}, \dots, \hat{\mu}^{(r)}_{m,E_m}$, ordered from smallest to largest,  denote $r$ independent estimates from equation~\eqref{eqn:muEdef} with precision $E_m$. If $f\in C^{\infty}([0,1]^s)$ satisfies the assumptions of Theorem~\ref{thm:step3} and $E_m\geq \lambda m^2/s$ with $\lambda=3(\log 2)^2/\pi^2$, then under the complete random design assumption,
\begin{equation}\label{eqn:muEmedian} 
\limsup_{m\to\infty} \sqrt{m}\Big|\Pr(\hat{\mu}_{m,E_m}<\mu)+\frac{1}{2}\Pr(\hat{\mu}_{m,E_m}=\mu)-\frac{1}{2}\Big|<\infty.
\end{equation}
    Moreover, for any indices $1\le \ell\le u\le r$,
    \begin{equation}\label{eqn:coverage}
        \liminf_{m\to\infty}\Pr(\mu\in [\hat{\mu}^{(\ell)}_{m,E_m},\hat{\mu}^{(u)}_{m,E_m}] )\geq F(u-1)-F(\ell-1),
    \end{equation}
    where $F$ denotes the cumulative distribution function of the binomial distribution $\operatorname{Bin}(r,1/2)$.
\end{corollary}

\begin{proof}
Let $\mathrm{SUM}_{2,m,E_m}=\mathrm{SUM}_{2,m}+\hat{\mu}_{m,E_m}-\hat{\mu}_{m,\infty}$. Then 
$$\hat{\mu}_{m,E_m}-\mu=\hat{\mu}_{m,\infty}-\hat{\mu}_{m,E_m}+\hat{\mu}_{m,E_m}-\mu=\mathrm{SUM}_{1,m}+\mathrm{SUM}_{2,m,E_m}.$$ Lemma 1 of \cite{superpolymulti} provides the bound
\begin{equation}\label{eqn:muE}
    |\hat{\mu}_{m,E_m}-\hat{\mu}_{m,\infty}|\leq \frac{\sqrt{s}}{2^{E_m}}\sup_{\bsx\in [0,1]^s}||\nabla f(\bsx)||_2.
\end{equation}
Since $f\in C^{\infty}([0,1]^s)$, its gradient $\nabla f(\bsx)$ is continuous over the compact domain $[0,1]^s$ and therefore has a bounded norm $||\nabla f(\bsx)||_2$. With $E_m\geq \lambda m^2/s$ and $N_m\sim \lambda m^2/s+ 3\lambda m\log_2(m)/s$, we obtain $|\hat{\mu}_{m,E_m}-\hat{\mu}_{m,\infty}|\leq T_m$ for the sequence $T_m$ defined in equation~\eqref{eqn:am}, given $m$ is sufficiently large. Hence, under the conditions of Corollary~\ref{cor:SUM2bound},
\begin{equation*}
\Pr\Big(|\mathrm{SUM}_{2,m,E_m}|\geq 2T_m\Big)\leq m^{d_{s,\alpha}}\exp\left(-c'_s \frac{m}{\log_2(m)^2}\right)
\end{equation*}
for large $m$.
Equation~\eqref{eqn:muEmedian} then follows from a straightforward modification of the proof of Theorem~\ref{thm:convergetomedian}.

Next, using the standard formula for order statistics,
$$\Pr(\hat{\mu}^{(\ell)}_{m,E_m}>\mu)=\sum_{j=0}^{\ell-1} \binom{r}{j} \Pr(\hat{\mu}_{m,E_m} \le \mu)^{j}\Pr(\hat{\mu}_{m,E_m} > \mu)^{r-j},$$
which is monotonically decreasing in $\Pr(\hat{\mu}_{m,E_m} \le \mu)$. From equation~\eqref{eqn:muEmedian}, we have
$$\liminf_{m\to\infty}\Pr(\hat{\mu}_{m,E_m} \le \mu)\geq 1/2,$$
and therefore
\begin{equation}\label{eqn:muEell}
    \limsup_{m\to\infty}\Pr(\hat{\mu}^{(\ell)}_{m,E_m}>\mu)\leq \sum_{j=0}^{\ell-1} \binom{r}{j}\frac{1}{2^r}=F(\ell-1).
\end{equation}
Similarly, 
\begin{equation}\label{eqn:muEu}
    \limsup_{m\to\infty}\Pr(\hat{\mu}^{(u)}_{m,E_m}<\mu)\leq \sum_{j=u}^{r} \binom{r}{j}\frac{1}{2^r}=1-F(u-1).
\end{equation}
Combining equations~\eqref{eqn:muEell} and \eqref{eqn:muEu} gives the desired coverage bound~\eqref{eqn:coverage}.
\end{proof}

In addition to providing asymptotically valid coverage, the interval $[\hat{\mu}^{(\ell)}_{m,E_m},\hat{\mu}^{(u)}_{m,E_m}]$ has length $\hat{\mu}^{(u)}_{m,E_m}-\hat{\mu}^{(\ell)}_{m,E_m}$ that converges in probability to $0$ at a super‑polynomial rate. To prove this, we first generalize \cite[Theorem 2]{superpolymulti} to complete random designs.

    \begin{theorem}\label{thm:superpoly}
        If $f\in C^{\infty}([0,1]^s)$ satisfies equation~\eqref{eqn:fdassumption} for some $K,\alpha>0$ and $\cj=\{J_1,\dots,J_L\}$, then for any $\gamma>0$, we can find a constant $\Gamma$ depending on $s,\alpha,\gamma,\cj_{\max}$ such that under the complete random design assumption,
        $$\limsup_{m\to\infty} m^{\gamma} \Pr\Big(|\hat{\mu}_{m,\infty}-\mu|> K2^{-\lambda m^2/s+\Gamma m\log_2(m)}\Big)\leq 1.$$
    \end{theorem}

    The proof is given in Appendix~\appD.

    \begin{corollary}\label{cor:superpoly}
        Under the assumptions of Corollary~\ref{cor:confint}, we can find for any $\gamma>0$ a constant $\Gamma'$ depending on $s,\alpha,\gamma,\cj_{\max}$ such that
        $$\limsup_{m\to\infty} m^{r^*\gamma }\Pr\Big(\hat{\mu}^{(u)}_{m,E_m}-\hat{\mu}^{(\ell)}_{m,E_m}> 4K2^{-\lambda m^2/s+\Gamma' m\log_2(m)}\Big)\leq \binom{r}{r^*},$$
        where $r^*=\min(\ell,r-u+1)$. 
    \end{corollary}
    \begin{proof}
    Given $\gamma>0$ and the corresponding $\Gamma$ from Theorem~\ref{thm:superpoly}, we can find a constant $\Gamma'\geq \Gamma$ such that $N_m-\lambda m^2/s+\Gamma' m\log_2(m) \to \infty$ as $m\to\infty$ because of the relation $N_m\sim \lambda m^2/s+3\lambda m\log_2(m)/s$. Since $E_m\geq \lambda m^2/s$, equation~\eqref{eqn:muE} implies $|\hat{\mu}_{m,E_m}-\hat{\mu}_{m,\infty}|\leq K2^{-\lambda m^2/s+\Gamma' m\log_2(m)}$ for large $m$. Together with Theorem~\ref{thm:superpoly}, we obtain
    $$\limsup_{m\to\infty} m^{\gamma} \Pr\Big(|\hat{\mu}_{m,E_m}-\mu|> 2K2^{-\lambda m^2/s+\Gamma' m\log_2(m)}\Big)\leq 1.$$
    In order for either $|\hat{\mu}^{(\ell)}_{m,E_m}-\mu|$ or $|\hat{\mu}^{(u)}_{m,E_m}-\mu|$ to exceed $2K2^{-\lambda m^2/s+\Gamma' m\log_2(m)}$, at least $r^*$ of the $r$ replicates  must have an error larger than this threshold.  A union bound over the $\binom{r}{r^*}$ subsets of size $r^*$ gives
    $$\limsup_{m\to\infty} m^{r^* \gamma } \Pr\Big(\max(|\hat{\mu}^{(\ell)}_{m,E_m}-\mu|,|\hat{\mu}^{(u)}_{m,E_m}-\mu|)> 2K2^{-\lambda m^2/s+\Gamma' m\log_2(m)}\Big)\leq \binom{r}{r^*}.$$
    If both $|\hat{\mu}^{(\ell)}_{m,E_m}-\mu|$ and $|\hat{\mu}^{(u)}_{m,E_m}-\mu|$ are bounded by $2K2^{-\lambda m^2/s+\Gamma' m\log_2(m)}$, then  $\hat{\mu}^{(u)}_{m,E_m}-\hat{\mu}^{(\ell)}_{m,E_m}\leq 4K2^{-\lambda m^2/s+\Gamma' m\log_2(m)} $ and our proof is complete.
    \end{proof}

    \begin{remark}
        If a point estimator for $\mu$ is desired, we can generate $r_m$ groups, each consisting of $r$ independent $\hat{\mu}_{m,E_m}$. For each group we compute the lower bound $\hat{\mu}^{(\ell)}_{m,E_m}$ and upper bound $\hat{\mu}^{(u)}_{m,E_m}$. Taking the median over the $r_m$ groups then yields $\mathrm{Med}(\hat{\mu}^{(\ell)}_{m,E_m})$ and $\mathrm{Med}(\hat{\mu}^{(\ell)}_{m,E_m})$, respectively. By an argument analogous to the proof of \cite[Corollary 3]{superpolymulti}, the mean squared errors of both median estimators converge to $0$ at a super‑polynomial rate provided $r_m\sim m^2$ as $m\to\infty$. Consequently, any value between  $\mathrm{Med}(\hat{\mu}^{(\ell)}_{m,E_m})$ and  $\mathrm{Med}(\hat{\mu}^{(u)}_{m,E_m})$ can serve as a point estimator for $\mu$.  Moreover, if $F(\ell-1)<1/2$ and $F(u-1)>1/2$ for $F$ defined in Corollary~\ref{cor:confint},  then equations~\eqref{eqn:muEell} and \eqref{eqn:muEu} imply that $\Pr(\mu\in [\mathrm{Med}(\hat{\mu}^{(\ell)}_{m,E_m}),\mathrm{Med}(\hat{\mu}^{(u)}_{m,E_m})] )$ converges to $1$ whenever $r_m\to\infty$ as $m\to\infty$.       
    \end{remark}

\section{Generalization to other randomization}\label{sec:general}

Up to this point, our discussion has focused on complete random designs. In applications, however, random linear scrambling is often preferred because the resulting digital nets typically have better low-dimensional projections. For example, the construction of \cite{joe:kuo:2008} explicitly targets the quality of all two-dimensional projections in its optimization. The key difficulty in extending our results in Section~\ref{sec:main} to this setting is that, under random linear scrambling, some rows of the matrices $\{C_j, j\in 1{:}s\}$ in equation~\eqref{eqn:xequalMCiplusD} are not distributed as $\dunif(\{0,1\}^m)$. In this section, we establish the additional assumptions required to generalize our main theorem to such broader randomization schemes.

Recall that in random linear scrambling, each matrix takes the form $C_j=M_j\mathcal{C}_j$, where $M_j\in \{0,1\}^{E\times m}$  is a  random lower-triangular matrix and $\mathcal{C}_j \in \{0,1\}^{m\times m}$ is a fixed generating  matrix. Typically every $\mathcal{C}_j$ is nonsingular, ensuring that no two points in the digital net overlap in their one-dimensional projections. A useful consequence of  full-rank generating matrices is that, except for the first $m$ rows of each $C_j$, the distribution of $C_j$ under random linear scrambling coincides with that under complete random designs. This observation motivates the following notion:

\begin{definition}
    The \textbf{marginal order} of a randomization scheme for the matrices $\{C_j\in \{0,1\}^{E\times m}, j\in 1{:}s\}$ is the smallest integer $d\in \natu_0$ such that, for every $j\in 1{:}s$ and $\ell>d m$, the row $C_j(\ell,:)$ is independently drawn from $\dunif(\{0,1\}^m)$.
\end{definition}

The marginal order is $0$ for complete random designs and $1$ for random linear scrambling provided every generating matrix has full rank. Randomization with higher marginal order becomes relevant when randomizing higher‑order digital nets, such as those introduced in \cite{dick:2011}.

Another key property of random linear scrambling is that for any nonzero $\bsk$, $\Pr(Z(\bsk)=1)\leq 2^{-m+R}$ for some constant $R$ depending on $s$ and the choice of generating matrices \cite[Corollary 1]{superpolymulti}. This motivates the following general definition:

\begin{definition}
For $r\in\natu$, let
$$\mathbb{V}_r=\{V\subseteq \natu^s_*\mid |V|=\rank(V)=r\}.$$
The \textbf{$r$-way rank deficiency} $R_{m,r}$ of a randomization scheme for matrices $\{C_j\in \{0,1\}^{E\times m}, j\in 1{:}s\}$ is defined by
\begin{equation}\label{eqn:Rmr}
   R_{m,r}=mr+\sup_{V\in \mathbb{V}_r}\log_2\Big(\Pr(Z(\bsk)=1 \text{ for all }\bsk\in V)\Big), 
\end{equation}
   where $Z(\bsk)$ is given by equation~\eqref{eqn:Zk}.
\end{definition}

In \cite{pan2024automatic}, a randomization scheme is called asymptotically full-rank if $R_{m,1}$ remains bounded as $m\to\infty$. This property holds for random linear scrambling based on common generating matrices, such as those from Sobol' \cite{sobol67} and Niederreiter \cite{NIEDERREITER198851}. Much less is known about $R_{m,r}$ for $r\geq 2$. One might conjecture that $R_{m,r}\leq r R_{m,1}$, but this fails in general. For instance, \cite[Section 5]{pan2024skewnessrandomizedquasimontecarlo} exhibits a three-dimensional example where $R_{m,1}\leq 5$ while $R_{m,2}\geq m/2 +3$, and $m$ can be taken as an arbitrarily large even number. Fortunately, for most generating matrices the corresponding $R_{m,r}$ grows logarithmically in $m$, as made precise in the following theorem:

\begin{theorem}\label{thm:Rmr}
Let $\mathcal{I}_m$ be the set of nonsingular $m\times m$ matrices over $\mathbb{F}_2$, and let the generating matrices $\{\mathcal{C}_j, j\in 1{:}s\}$ be independently drawn from $\dunif (\mathcal{I}_m)$.
Then for the random linear scrambling based on these generating matrices,
    \begin{equation*}
        \Pr\Big(R_{m,1}\geq 3s\log_2(m+1) \Big) \leq \frac{\exp(2s)}{(m+1)^{2s}},
    \end{equation*}
    and for $r\geq 2$,
    \begin{equation*}
        \Pr\Big(R_{m,r}\geq \max\big(R_{m,r-1},(2^{r}+2r-1)s\log_2(m+1)\big) \Big)\leq \frac{\exp(2sr)}{(m+1)^{2sr}}.
    \end{equation*}
\end{theorem}

The proof is given in Appendix~\appE.

\begin{corollary}\label{cor:randomC}
When $m\geq 3$, there exist generating matrices $\{\mathcal{C}_j,j\in 1{:}s\}$ for which random linear scrambling has marginal order $1$ and satisfies $R_{m,r}\leq (2^r+2r-1) s\log_2(m+1)$ for all $r\in\natu$. 
\end{corollary}

\begin{proof}
    Let $\{\mathcal{C}_j, j\in 1{:}s\}$ be independently sampled from $\dunif(\mathcal{I}_m)$. The marginal order is $1$ because every $\mathcal{C}_j$ is nonsingular. Theorem~\ref{thm:Rmr} along with a union bound over all $r\in\natu$ yields
    \begin{align*}
        \Pr(R_{m,r}\leq (2^r+2r-1) s\log_2(m+1) \text{ for all }r\geq 1)\geq &1-\sum_{r=1}^{\infty}\frac{\exp(2sr)}{(m+1)^{2sr}}.
    \end{align*}
    For $m\geq 3$, we have $(m+1)^{-2s}\exp(2s)<2^{-s}$, so the infinite sum is strictly less than $1$. Hence the probability is positive, guaranteeing the existence of such generating matrices.
\end{proof}

Now we generalize Theorem~\ref{thm:convergetomedian}. Let
\begin{equation}\label{eqn:newNm}
    N_m=\sup \{N\in \natu_0\mid |Q_N|\le (1/2)\log_2(m) 2^m \}.
\end{equation}
By a calculation similar to equations~\eqref{eqn:QNm} and \eqref{eqn:Nm}, $Q_{N_m}\sim (1/2)\log_2(m) 2^m$ and
\begin{equation*}
	N_m\sim \lambda m^2/s +\lambda m\log_2(m)/s \quad\text{for} \quad\lambda=3(\log 2)^2/\pi^2.
\end{equation*}

We first present a generalization of Lemma~\ref{lem:SUM1bound}, whose proof is given in Appendix~\appF.

\begin{lemma}\label{lem:generalSUM1bound}
Let $L_m\subseteq Q_{N_m}$ and $\liminf_{m\to\infty} |L_m|/|Q_{N_m}|>0$.
Under a randomization scheme with marginal order $d\in \natu_0$ and $r$-way rank deficiency $R_{m,r}$ satisfying $\lim_{m\to\infty}R_{m,1}/m=\lim_{m\to\infty}R_{m,2}/m=0$,
 $$\lim_{m\to\infty}\Pr\left(\frac{|L_m|}{2^{m+1}}\leq \sum_{\bsk\in L_m}Z(\bsk)\leq \frac{3|L_m|}{2^{m+1}} \right)=1.$$
\end{lemma}

\begin{theorem}\label{thm:generalver}
    Suppose $f\in C^\infty([0,1]^s)$ satisfies the assumptions of Theorem~\ref{thm:step3}. Then under a randomization scheme with marginal order $d\in \natu_0$ and $r$-way rank deficiency $R_{m,r}$ satisfying $R_{m,r}\leq (2^r+2r-1) s \log_2(m+1)$ for $r\in\natu$,
    $$\lim_{m\to\infty}\left(\Pr(\hat{\mu}_{m,\infty}<\mu)+\frac{1}{2}\Pr(\hat{\mu}_{m,\infty}=\mu)\right)=\frac{1}{2}.$$
\end{theorem}

\begin{proof}
  Let $\mathrm{SUM}_{1,m}, \mathrm{SUM}_{2,m}$ and $\mathrm{SUM}'_{1,m}$ be defined by equations~\eqref{eqn:SUM1} and \eqref{eqn:SUM'1} with  $K_m=Q_{N_m}$ for $N_m$ defined by equation~\eqref{eqn:newNm}. We prove Theorem~\ref{thm:generalver} by following the same three steps outlined at the beginning of Section~\ref{sec:main}.

\textbf{Step 1.} Equation~\eqref{eqn:dTV} implies,  for $\mathcal{V}=\{\bsk\in Q_{N_m}\mid Z(\bsk)=1\}$,
 \begin{align*}
 d_{TV}(\mathrm{SUM}_{1,m},\mathrm{SUM}'_{1,m})
 \leq & \Pr\Big(\mathcal{V}\in I_m, |\mathcal{V}|\le\frac{3}{4}\log_2(m)\Big)+\Pr\Big(|\mathcal{V}|>\frac{3}{4}\log_2(m)\Big).
 \end{align*}
 Lemma~\ref{lem:generalSUM1bound} with $L_m=Q_{N_m}$ gives $\Pr(|\mathcal{V}|>(3/4)\log_2(m))\to 0$ as $m\to\infty$. Next, from equation~\eqref{eqn:VIm}, for sufficiently large $m$,
 \begin{equation*}
    \Pr\Big(\mathcal{V}\in I_m, |\mathcal{V}|\leq \frac{3}{4}\log_2(m) \Big)  \leq \sum_{r=2}^{\lfloor(3/4)\log_2(m)\rfloor} \sum_{W\in I^*_{m,r+1}}\Pr(W\subseteq \mathcal{V}).
    \end{equation*}
    Because each $W\in I^*_{m,r+1}$ has full rank, Lemma~\ref{lem:linearindep} and equation~\eqref{eqn:Rmr} yield
$$\sum_{W\in I^*_{m,r+1}}\Pr(W\subseteq \mathcal{V})\leq |I^*_{m,r+1}|2^{-mr+R_{m,r}}\leq \frac{2^{R_{m,r}}}{(r+1)!}A^r_s N_m^{r/4} r^{-B_s\sqrt{N_m}}.$$
For $r\leq (3/4)\log_2(m)$, $R_{m,r}\leq (m^{3/4}+(3/2)\log_2(m)-1)s\log_2(m+1)$. Consequently, 
\begin{align*}
  & \Pr\Big(\mathcal{V}\in I_m, |\mathcal{V}|\leq \frac{3}{4}\log_2(m) \Big) \\
  \leq &2^{(m^{3/4}+(3/2)\log_2(m)-1)s\log_2(m+1)} \sum_{r=2}^{\lfloor(3/4)\log_2(m)\rfloor}  \frac{(A_s N^{1/4}_m)^r}{(r+1)!} r^{-B_s\sqrt{N_m}}\\
  \leq & 2^{(m^{3/4}+(3/2)\log_2(m)-1)s\log_2(m+1)}\exp(A_s N^{1/4}_m) 2^{-B_s\sqrt{N_m}},
\end{align*}
which converges to $0$ as $m\to\infty$ since $N_m\sim \lambda m^2/s$. This completes Step 1.

\textbf{Step 2.}  The previous argument still applies, except that equation~\eqref{eqn:ZkQtildebound} becomes
\begin{align*}
&\Pr(Z(\bsk)=1 \text{ for some }\bsk\in \Tilde{Q})\leq 2^{-m+R_{m,1}}|\Tilde{Q}|\\ \le &2^{R_{m,1}} m^{K_4} s A'_s(N^*_m)^{1/4}\exp(-B'_s \epsilon_m^2 \sqrt{N^*_m})
\leq  m^{d_{s,\alpha} }\exp(-c'_s \epsilon_m^2 m),    
\end{align*}
using $R_{m,1}\leq 3s\log_2(m+1)$. Corollary~\ref{cor:SUM2bound} therefore remains valid.

\textbf{Step 3.}  Under the assumptions of Theorem~\ref{thm:step3}, equation~\eqref{eqn:Qmam} holds and $|Q_{N_m}\cap \Lambda(T_m)|>c\log_2(m)2^m$ for some $c>0$. Lemma~\ref{lem:generalSUM1bound} with $L_m=Q_{N_m}\cap \Lambda(T_m)$ then yields
$$\lim_{m\to\infty}\Pr(|\mathcal{W}|>(c/2)\log_2(m))=1$$
for $\mathcal{W}=\{\bsk\in Q_{N_m}\cap \Lambda(T_m)\mid Z(\bsk)=1\}$. By an argument analogous to equation~\eqref{eqn:step3},
\begin{align*}
    &\Pr\left(|\mathrm{SUM}'_{1,m}|\le T_m\right)\\
    \leq & \Pr\left(|\mathrm{SUM}'_{1,m}|\le T_m,|\mathcal{W}|> (c/2)\log_2(m)\right) +\Pr\left(|\mathcal{W}|\leq (c/2)\log_2(m)\right)\\
    \leq & \frac{C}{\sqrt{(c/2)\log_2(m)}}+\Pr\left(|\mathcal{W}|\leq (c/2)\log_2(m)\right),
\end{align*}
which converges to $0$ as $m\to\infty$. This completes Step 3.
\end{proof}

\begin{remark}
Under the assumptions of Theorem~\ref{thm:generalver}, analogues of Corollary~\ref{cor:confint}, Theorem~\ref{thm:superpoly} and Corollary~\ref{cor:superpoly} can be established. In particular, if we additionally assume $E_m\geq \lambda m^2/s$ with $\lambda=3(\log 2)^2/\pi^2$, then the quantile interval $[\hat{\mu}^{(\ell)}_{m,E_m},\hat{\mu}^{(u)}_{m,E_m}]$ achieves asymptotic coverage as specified in equation~\eqref{eqn:coverage}.
\end{remark}

\begin{remark}
Corollary~\ref{cor:randomC} shows that generating matrices exist for which random linear scrambling satisfies the assumptions of Theorem~\ref{thm:generalver}. The proof in fact shows that matrices drawn uniformly from $\mathcal{I}_m$ qualify with high probability. If, in addition, $R_{m,r}\leq Cr\log_2(m+1)$ for some $C>0$, one can also establish the stronger $m^{-1/2}$ convergence rate in Theorem~\ref{thm:convergetomedian}. The existence of generating matrices achieving such bounds remains an open question.
\end{remark}

\section{Numerical experiments}\label{sec:experiment}

In this section, we validate our theoretical results on two test integrands. In each experiment, the precision $E$ is chosen according to a small test run to ensure $2^{-E}$ is much smaller than the observed errors.

Below we use CRD as the shorthand for ``complete random designs" and RLS for ``random linear scrambling". The generating matrices for RLS come from \cite{joe:kuo:2008}.

\begin{figure}
    \centering
    \includegraphics[width=0.45\paperwidth]{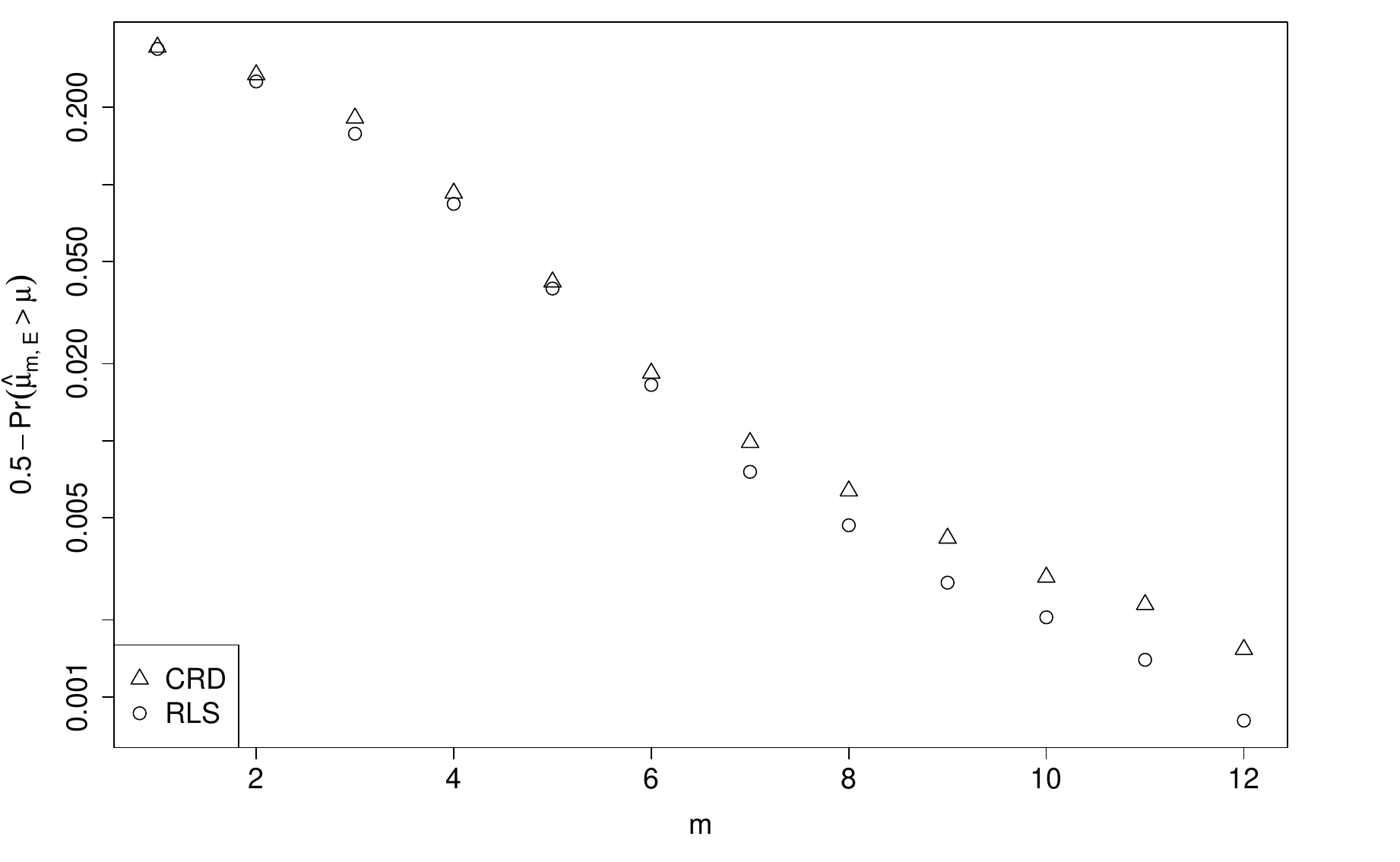}
    \caption{Deviation of $\Pr(\hat{\mu}_{m,E} > \mu)$ from $1/2$ when $f(x)=x^{33}\exp(x)$.}
    \label{fig:onePr}
\end{figure}

\begin{figure}
    \centering
    \includegraphics[width=0.45\paperwidth]{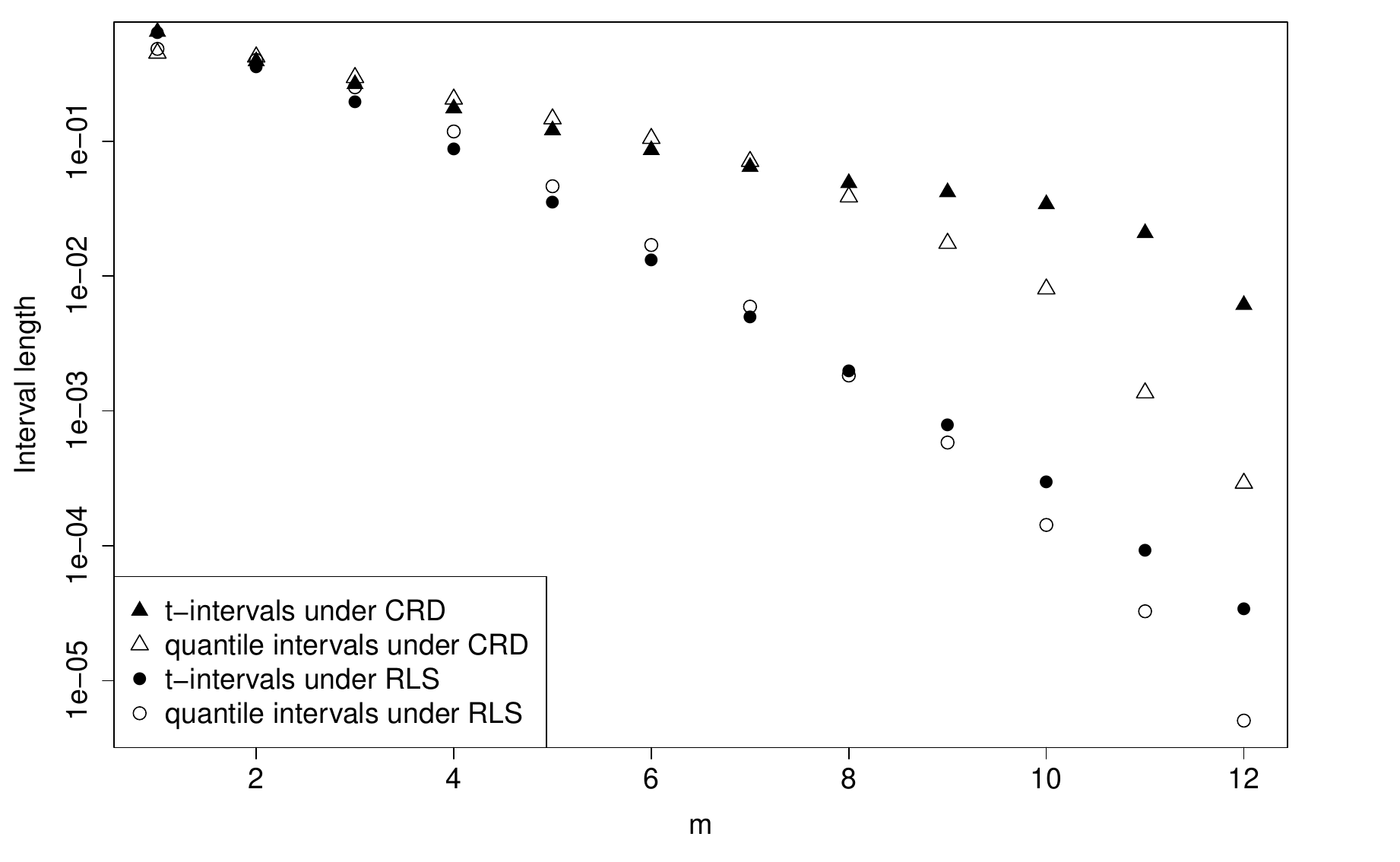}
    \caption{$90$th percentile interval lengths when $f(x)=x^{33}\exp(x)$.}
    \label{fig:onelength}
\end{figure}

\begin{figure}
    \centering
    \includegraphics[width=0.45\paperwidth]{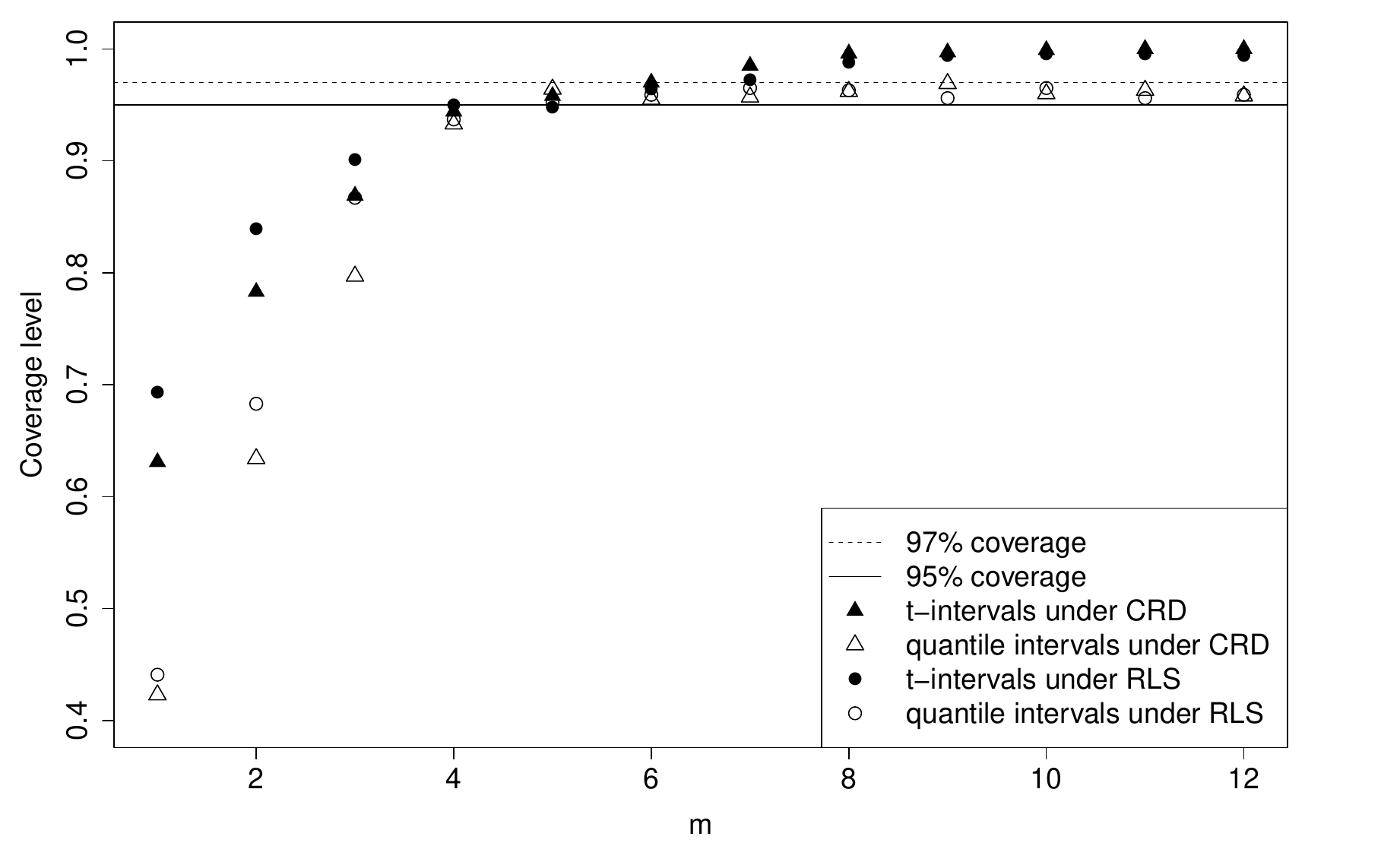}
    \caption{Coverage levels when $f(x)=x^{33}\exp(x)$.}
    \label{fig:onecover}
\end{figure}

\begin{figure}
    \centering
    \includegraphics[width=0.45\paperwidth]{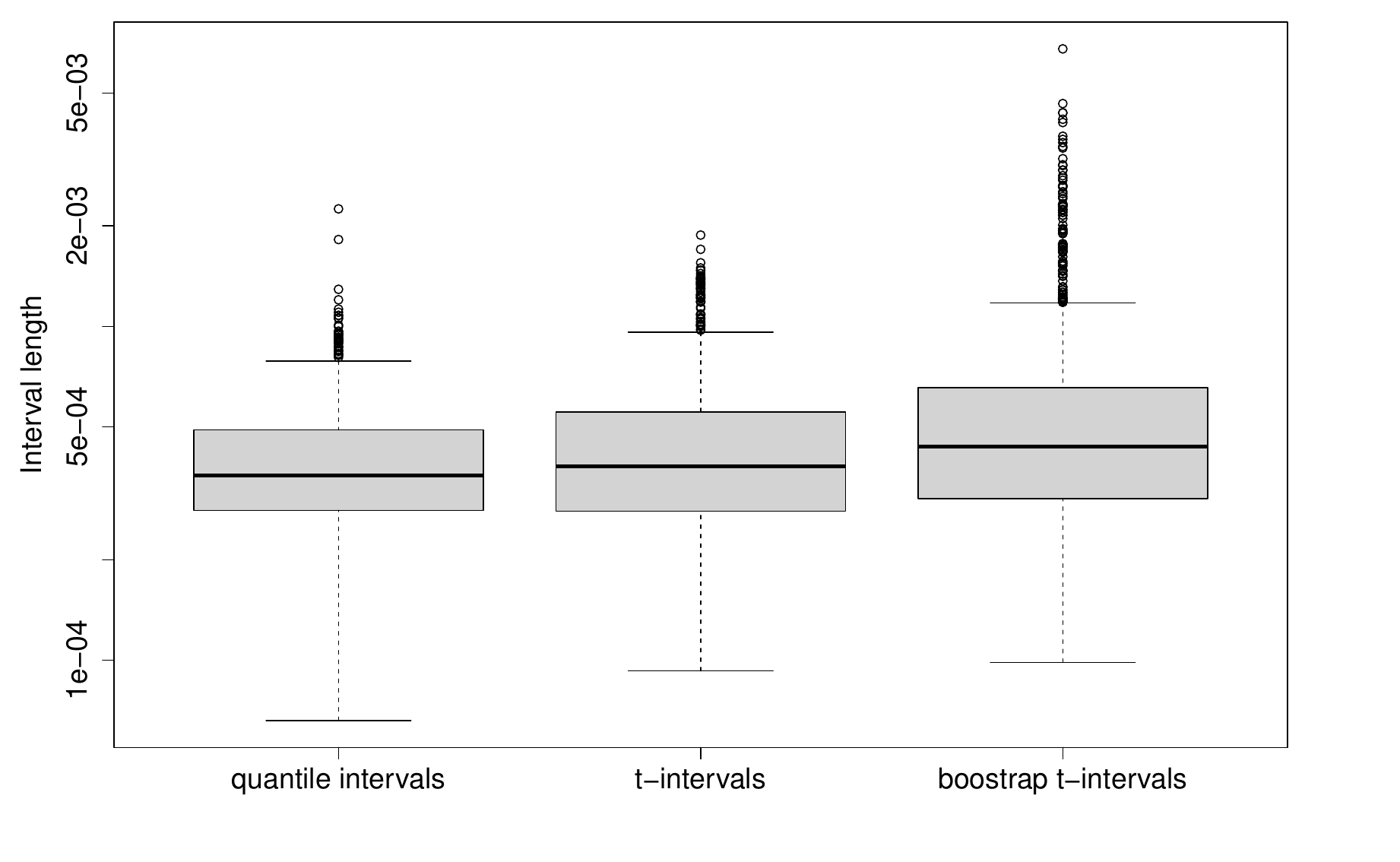}
    \caption{Distribution of interval lengths when $f(\bsx)$ is the Robot Arm function.}
    \label{fig:robotlength}
\end{figure}

\begin{figure}
    \centering
    \includegraphics[width=0.45\paperwidth]{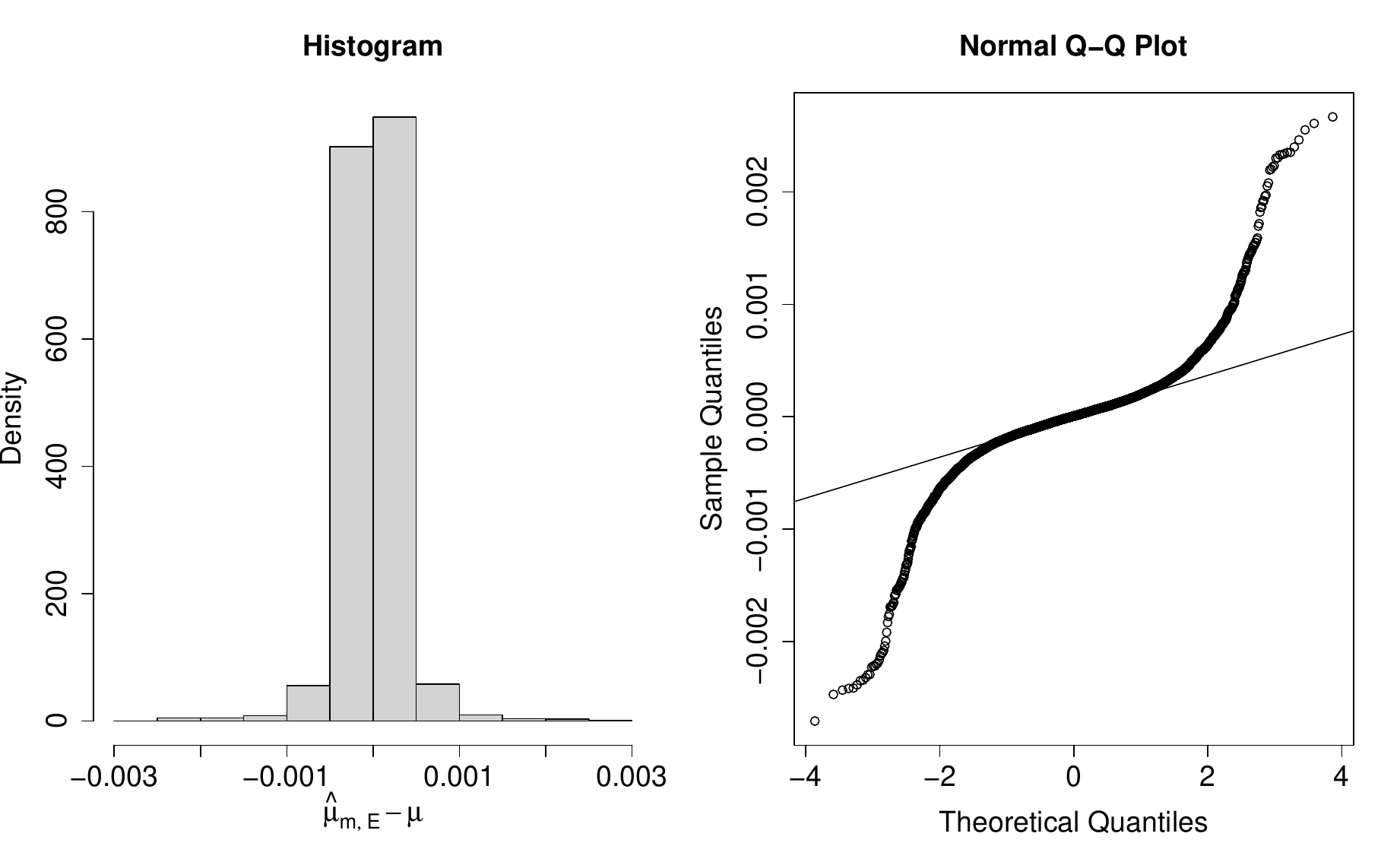}
    \caption{Distribution of errors when $f(\bsx)$ is the Robot Arm function.}
    \label{fig:robotnormality}
\end{figure}

\subsection{One-dimensional example}\label{subsec:expr1}

We first consider a highly skewed one-dimensional integrand $f(x)=x^{33}\exp(x)$. The power $33$ is chosen so that $\Pr(f(U)>\mu)\approx 10\%$ for $U$ following a uniform $[0,1]$ distribution.
In other words, $\Pr(\hat{\mu}_{m,\infty}>\mu)\approx 10\%$ when we set $m=0$ and use one function evaluation to estimate $\mu$.

Figure~\ref{fig:onePr} verifies that $\Pr(\hat{\mu}_{m,E}>\mu)$ converges to $1/2$ under both CRD and RLS.  For $m=1,\dots,12$, each probability is estimated from $8\times10^6$ replicates with $E=64$. Empirically, the convergence appears to be exponential in $m$, which is substantially faster than the $m^{-1/2}$ rate guaranteed by Theorem~\ref{thm:convergetomedian}.

Figures~\ref{fig:onelength} and~\ref{fig:onecover} compare quantile intervals against traditional $t$-intervals. For each $m$, we generate both interval types $1000$ times, each based on $r=9$ independent replicates of $\hat{\mu}_{m,E}$. For the quantile interval we take $\ell=2$ and $u=8$, giving a nominal coverage of approximately $ 96.1\%$ via equation~\eqref{eqn:coverage}. The $t$-interval is $[\bar{\mu}-t\hat{\sigma}/\sqrt{r},\bar{\mu}+t\hat{\sigma}/\sqrt{r}]$ with $\bar{\mu}$ and $\hat{\sigma}$ the sample mean and sample standard deviation of the $9$ replicates. Here, we set $t\approx2.46$ so that the $t$-distribution with $8$ degrees of freedom yields the same nominal coverage.

Figure~\ref{fig:onelength} reports the $90$th percentile of interval lengths. Quantile intervals shrink rapidly as $m$ increases, while $t$-intervals remain substantially wider due to outliers. Figure~\ref{fig:onecover} shows empirical coverage, estimated by the proportion of intervals containing $\mu$. We consider the coverage too low if less than $950$ intervals contain $\mu$, and too high if  more than $970$ do. Quantile intervals maintain the target coverage for $m\ge 5$, whereas $t$-intervals become too wide and their coverage too high for $m\ge 7$. Thus quantile intervals are preferred for constructing confidence intervals from $\hat{\mu}_{m,E}$ in this setting.

We repeated the experiment on the eight-dimensional integrand $f(\bsx)=\prod_{j=1}^8 x_j\exp(x_j)$. The results are provided in Appendix~\appG. The same qualitative pattern emerges: quantile intervals attain nominal coverage as 
$m$ increases, whereas $t$-intervals become overly conservative due to outliers. However, the effect is milder, likely due to the curse of dimensionality.

\subsection{Eight-dimensional example}\label{subsec:expr2}
Next, we consider the eight-dimensional Robot Arm function from \cite{surj:bing:2013}:
$$f(\bsx)=\sqrt{u(\bsx)^2+v(\bsx)^2},$$
where
$$u(\bsx)=\sum_{i=1}^4 L_i(\bsx)\cos\left(\sum_{j=1}^i \theta_j(\bsx)\right) \quad \text{and}\quad v(\bsx)=\sum_{i=1}^4 L_i(\bsx)\sin\left(\sum_{j=1}^i \theta_j(\bsx)\right).$$
Here, $L_i(\bsx)=L_{\min}(1-x_i)+L_{\max}x_i$, so that $L_i(\bsx)$ is uniformly distributed over $[L_{\min},L_{\max}]$, and $\theta_j(\bsx)=2\pi x_{j+4}$, so that $\theta_j(\bsx)$ is uniformly distributed over $[0,2\pi)$. The original function assumes $L_{\min}=0$ and $L_{\max}=1$, but to reduce the influence of the singularity of $\sqrt{x}$ at $x=0$, we instead set $L_{\min}=1$ and $L_{\max}=2$.

We generated $1000$ groups, each consisting of $9$ independent replicates $\hat{\mu}_{m,E}$ under RLS with $m=16$ and $E=32$. For each group, we constructed a quantile interval, a $t$-interval, and a bootstrap $t$-interval \cite{efron1981nonparametric}, using the same parameters $\ell$, $u$, and $t$ as in Subsection~\ref{subsec:expr1}. The true integral value $\mu$ was estimated by the median of $9$ independent replicates $\hat{\mu}_{m,E}$ under RLS with $m=24$ and $E=32$.

Figure~\ref{fig:robotlength} shows the distribution of interval lengths for the three methods. Quantile intervals are the shortest, capturing $\mu$ in $955$ of the $1000$ trials. $T$-intervals are slightly longer and capture $\mu$ $971$ times. Bootstrap $t$-intervals perform the worst, with the longest lengths and only $940$ captures, consistent with the observations in \cite{ci4rqmc}. This is likely due to the heavy tails of $\hat{\mu}_{m,E}$'s distribution (Figure~\ref{fig:robotnormality}), to which bootstrap $t$-intervals are more vulnerable than ordinary $t$-intervals. We refer to \cite{owen2025errorestimationquasimontecarlo,owen2025better} for further discussion.

We also tested the empirical Bernstein intervals \cite{jain2025empirical}. The results are given in Appendix~\appH. These intervals are highly conservative, with lengths ranging from $5\times 10^{-3}$ to $10^{-2}$, and are therefore not recommended in this setting.

\section{Discussion}\label{sec:discussion}

Our analysis has so far focused on infinitely differentiable integrands. A main obstacle in extending our results to finitely differentiable integrands is that the decay of Walsh coefficients is insufficient to decompose $\hat{\mu}_{m,\infty}-\mu$ in a manner compatible with Lemma~\ref{lem:proofoutline}. Consider, for instance, an integrand $f$ with square-integrable dominating mixed derivatives of order $1$ ($f^{|\bskappa|}(\bsx)$ for $|\bskappa|\in \{0,1\}^s$). By \cite[Corollary 3]{pan2024automatic} with $\alpha=0,\lambda=1$, for any $\bsell=(\ell_1,\dots,\ell_s)\in \natu_*^s$,
$$\sum_{\bsk\in B_{\bsell,s}}|\hat{f}(\bsk)|^2\leq C_{f,s} 4^{-\sum_{j=1}^s\ell_j},$$
where $B_{\bsell,s}=\{\bsk\in\natu^s_*\mid \lceil \kappa_j\rceil=\ell_j \text{ for } j\in 1{:}s\}$ and $C_{f,s}$ is a constant depending on $f$ and $s$. Mimicking the proof in Section~\ref{sec:main}, one could set $K_m=Q'_{N_m}$ with
$$Q'_{N}=\Big\{\bsk\in\natu^s_*\Big\vert \bsk\in B_{\bsell,s}\text{ for } \bsell \in \natu_*^s \text{ satisfying } \sum_{j=1}^s\ell_j\leq N\Big\},$$
and attempt to tune $N_m$ to satisfy Lemma~\ref{lem:proofoutline}. Unlike $Q_N$, however, the set $Q'_{N}$ is rich in additive relations (it even forms a $\mathbb{F}_2$-vector space when $s=1$), which restricts the number of summands in $\mathrm{SUM}_{1,m}$ if $d_{TV}(\mathrm{SUM}_{1,m},\mathrm{SUM}'_{1,m})$ is to remain small. Hence, our proof strategy cannot be directly applied to finitely differentiable integrands.

A natural follow-up question concerns the limiting distribution of $\hat{\mu}_{m,\infty}-\mu$. By Theorem~\ref{thm:step1} and Corollary~\ref{cor:SUM2bound}, we can replace $\hat{\mu}_{m,\infty}-\mu$ by $\mathrm{SUM}'_{1,m}$ when studying its limiting distribution. The difficulty lies in the dependencies among $Z(\bsk)$. For large $m$, we conjecture that $\mathrm{SUM}'_{1,m}$ can be approximated by
$$\mathrm{SUM}''_{1,m}=\sum_{\bsk\in Q_{N_m}}Z'(\bsk)S'(\bsk)\hat{f}(\bsk),$$
where each $Z'(\bsk)$ is sampled independently from a Bernoulli distribution with success probability $2^{-m}$. This approximation holds rigorously for polynomial integrands, where the support of non-zero Walsh coefficients is particularly sparse. How to extend this result to general integrands is another challenging question for future research.

A critical limitation of quantile-based confidence intervals lies in their finite-sample coverage guarantees. Assuming $\Pr(\hat{\mu}_{m,E_m}=\mu)=0$, the coverage probability of the interval $[\hat{\mu}^{(\ell)}_{m,E_m},\hat{\mu}^{(u)}_{m,E_m}]$ from Corollary~\ref{cor:confint}, with $r$ odd and a symmetric choice of indices $\ell=r+1-u$, is maximized when $\Pr(\hat{\mu}_{m,E_m}<\mu)=\Pr(\hat{\mu}_{m,E_m}>\mu)=1/2$. Only in this case does the interval reach the nominal level $F(u-1)-F(\ell-1)$. In applications where undercoverage poses significant risks, the slower but conservative $t$-interval may be preferable. If explicit bounds on the integrand are known, the empirical Bernstein interval from \cite{jain2025empirical} yields finite-sample guarantees and faster-than-MC convergence. Designing intervals that simultaneously achieve robust finite-sample coverage and QMC convergence rates without such bounds remains an open problem.

\section*{Acknowledgments}

This work is an expanded version of a chapter from the author's doctoral thesis \cite{mythesis}. The research was completed during the author's postdoctoral fellowship at the Johann Radon Institute for Computational and Applied Mathematics (RICAM). The author is grateful to Professor Art Owen for his mentorship and insightful suggestions, and to RICAM for their fellowship support. 

This work was supported by the U.S. National Science Foundation under grant DMS-2152780 and the Austrian Science Fund (FWF) 
Project DOI 10.55776/P34808. For open access purposes, the author has 
applied a CC BY public copyright license to any author accepted 
manuscript version arising from this submission.

\appendix

\section{Proof of Lemma \ref{lem:proofoutline}}
    First notice that
\begin{align*}
    &\Pr(\hat{\mu}_{m,\infty}<\mu) =  \Pr(\mathrm{SUM}_{1,m}+\mathrm{SUM}_{2,m}<0) \nonumber\\
    \geq & \Pr(\mathrm{SUM}_{1,m}<0 \text{ and } |\mathrm{SUM}_{1,m}|> |\mathrm{SUM}_{2,m}| ) \nonumber\\
    \geq & \Pr(\mathrm{SUM}_{1,m}<0) - \Pr(|\mathrm{SUM}_{1,m}|\leq |\mathrm{SUM}_{2,m}| )\nonumber\\
    \geq & \Pr(\mathrm{SUM}'_{1,m}<0) -d_{TV}(\mathrm{SUM}_{1,m},\mathrm{SUM}'_{1,m})- \Pr(|\mathrm{SUM}_{1,m}|\leq |\mathrm{SUM}_{2,m}| ),
\end{align*}
where we have applied $|\Pr(\mathrm{SUM}'_{1,m}<0)-\Pr(\mathrm{SUM}_{1,m}<0)|\leq d_{TV}(\mathrm{SUM}_{1,m},\mathrm{SUM}'_{1,m})$ in the last inequality.
Similarly, 
\begin{align*}
   &\Pr(\hat{\mu}_{m,\infty}\leq \mu) \geq  \Pr(\mathrm{SUM}_{1,m}\leq 0 \text{ and } |\mathrm{SUM}_{1,m}|\geq |\mathrm{SUM}_{2,m}| ) \\
   \geq & \Pr(\mathrm{SUM}'_{1,m}\leq 0) -d_{TV}(\mathrm{SUM}_{1,m},\mathrm{SUM}'_{1,m})- \Pr(|\mathrm{SUM}_{1,m}|< |\mathrm{SUM}_{2,m}| ).
\end{align*}
Hence,
\begin{align}\label{eqn:muhatlemu}
        &\Pr(\hat{\mu}_{m,\infty}<\mu)+\Pr(\hat{\mu}_{m,\infty}\le\mu)-\left(\Pr(\mathrm{SUM}'_{1,m}<0)+\Pr(\mathrm{SUM}'_{1,m}\leq 0)\right)\\
        \geq &-2d_{TV}(\mathrm{SUM}_{1,m},\mathrm{SUM}'_{1,m})- 2\Pr(|\mathrm{SUM}_{1,m}|\leq |\mathrm{SUM}_{2,m}| ).\nonumber
    \end{align}
  A similar argument shows
\begin{align}\label{eqn:muhatgemu}
        &\Pr(\hat{\mu}_{m,\infty}>\mu)+\Pr(\hat{\mu}_{m,\infty}\ge\mu)-\left(\Pr(\mathrm{SUM}'_{1,m}>0)+\Pr(\mathrm{SUM}'_{1,m}\geq 0)\right)\\
        \geq &-2d_{TV}(\mathrm{SUM}_{1,m},\mathrm{SUM}'_{1,m})- 2\Pr(|\mathrm{SUM}_{1,m}|\leq |\mathrm{SUM}_{2,m}| ). \nonumber
    \end{align}  
    Conditioned on $Z(\bsk)$, $\mathrm{SUM}'_{1,m}$ is a sum of independent symmetric random variables, so
    $$\Pr(\mathrm{SUM}'_{1,m}< 0)+\Pr(\mathrm{SUM}'_{1,m}\le 0)=\Pr(\mathrm{SUM}'_{1,m}>0)+\Pr(\mathrm{SUM}'_{1,m}\geq 0)=1.$$
    Our conclusion then follows from equations~\eqref{eqn:muhatlemu} and \eqref{eqn:muhatgemu}.

\section{Proofs of Lemma \ref{lem:sumk} and Lemma \ref{lem:meankappa}}
This section proves supporting lemmas concerning $Q_N=\{\bsk\in \natu_*^s\mid \Vert\bskappa\Vert_1 \leq N\}$. The proof strategy is inspired by \cite{bert:1993}. To simplify the notation, we write $a_N=O(b_N)$ if $\limsup_{N\to\infty}|a_N|/|b_N|<C$ for some constant $C>0$ and  $a_N=o(b_N)$ if $\lim_{N\to\infty}|a_N|/|b_N|=0$. 

We first construct an importance sampling measure on $\bsk\in \natu_0^s$. Recall that each $k\in\natu_0$ corresponds to $\kappa\subseteq \natu$ through $k=\sum_{\ell\in \kappa}2^{\ell-1}$. Let $L_N(\bsk)$ be the likelihood function of $\bsk$ under the importance sampling measure described by
$$L_N(\bsk)=\prod_{j=1}^s \prod_{\ell=1}^\infty \Bigl(\frac{q_N^{\ell}}{1+q_N^\ell}\Bigr)^{\bsone\{\ell\in \kappa_j\}}\Bigl(\frac{1}{1+q_N^\ell}\Bigr)^{\bsone\{\ell\notin \kappa_j\}}=\frac{q_N^{\Vert\bsk\Vert_1}}{\prod_{\ell=1}^\infty (1+q_N^\ell)^s}$$
with $q_N=\exp(-\pi\sqrt{s/12N})$. The value of $q_N$ is chosen so that $L_N(\bsk)$ closely approximates $\dunif( Q_N)$. Under $L_N(\bsk)$, it is clear that $X_{j\ell}=\bsone\{\ell\in \kappa_j\}$ equals $1$ with probability $q_N^\ell/(1+q_N^\ell)$ and $\{X_{j\ell},j\in 1{:}s, \ell\in \natu\}$ are jointly independent. We use ${\textstyle\Pr},\e, \var$ to denote the probability, expectation and variance when $\bsk$ follows a $\dunif( Q_N)$ distribution and ${\textstyle\Pr}^L,\e^L, \var^L$ to denote those under the importance sampling measure $L_N(\bsk)$.

Suppose we are interested in $\Pr(\bsk\in A)=|A|/|Q_N|$ for a subset $A\subseteq Q_N$.  We can compute it under the importance sampling measure by
\begin{equation}\label{eqn:PrtoEL}
    \Pr(\bsk\in A)=\e^L\Big[\frac{\bsone(\bsk\in A)}{|Q_N|L(\bsk)}\Big]=\e^L\Big[\frac{\bsone(\bsk\in A)}{|Q_N|q_N^{\Vert\bsk\Vert_1}}\prod_{\ell=1}^\infty (1+q_N^\ell)^s\Big].
\end{equation}
Since $\bsk\in A \subseteq Q_N$ implies $\Vert\bsk\Vert_1\leq N$,
\begin{equation}\label{eqn:PrlePrL}
    \Pr(\bsk\in A)\leq \e^L\Big[\frac{\bsone(\bsk\in A)}{|Q_N|q_N^{N}}\prod_{\ell=1}^\infty (1+q_N^\ell)^s\Big]= \frac{{\textstyle\Pr}^L(\bsk\in A)}{|Q_N|q_N^N}\prod_{\ell=1}^\infty (1+q_N^\ell)^s.
\end{equation}
Hence, we can bound $\Pr(\bsk\in A)$ by ${\textstyle\Pr}^L(\bsk\in A)$ times a factor depending only on $N$ and $s$, which is further bounded by the following lemma:
\begin{lemma}\label{lem:likehoodratio} 
When $N\geq 1$,
    $$\frac{1}{|Q_N|q_N^N}\prod_{\ell=1}^\infty (1+q_N^\ell)^s\leq A_s N^{1/4}$$
    with $A_s$ a constant depending on $s$.
\end{lemma}
\begin{proof}
    First we write
    \begin{align*}
        \prod_{\ell=1}^\infty (1+q_N^\ell)^s=\exp\Big(s\sum_{\ell=1}^\infty \log(1+q^\ell_N)\Big).
    \end{align*}
    Because $\log(1+q^\ell_N)$ is monotonically decreasing in $\ell$, 
    \begin{align*}
        \sum_{\ell=1}^\infty \log(1+q^\ell_N)\leq&\int_0^\infty \log\Big(1+\exp(-\pi\ell\sqrt{s/12N})\Big)\rd \ell\\
        =&\frac{1}{\pi}\sqrt{\frac{12N}{s}}\int_0^\infty \log\Big(1+\exp(-\ell)\Big)\rd \ell\\
        =&\pi\sqrt{\frac{N}{12s}}.
    \end{align*}
    Hence
    \begin{equation*}
        \prod_{\ell=1}^\infty (1+q_N^\ell)^s\leq \exp\Big(\pi\sqrt{\frac{sN}{12}}\Big).
    \end{equation*}
    Our conclusion then follows from $q^N_N=\exp(-\pi\sqrt{sN/12})$ and equation~\eqref{eqn:QNasymp}.
\end{proof}

Now we are ready to prove Lemma \ref{lem:meankappa}.
\begin{proof}[Proof of Lemma \ref{lem:meankappa}]
    Equation~\eqref{eqn:PrlePrL} and Lemma~\ref{lem:likehoodratio} imply
    \begin{equation}\label{eqn:PrlePrLforkappa}
        \Pr\Bigl(\Big|\frac{|\kappa_j|}{\sqrt{\lambda N/s}}-2\Big|>\epsilon\Bigr)\leq A_sN^{1/4}{\textstyle\Pr}^L \Bigl(\Big|\frac{|\kappa_j|}{\sqrt{\lambda N/s}}-2\Big|>\epsilon\Bigr).
    \end{equation}
    Because 
        $$|\kappa_j|=\sum_{\ell\in \natu}\bsone\{\ell\in \kappa_j\}=\sum_{\ell\in \natu}X_{j\ell},$$
    we have
    $$\e^L[|\kappa_j|]=\sum_{\ell\in \natu}\frac{q_N^\ell}{1+q_N^\ell}=\sum_{\ell\in \natu}\frac{\exp(-\pi\ell\sqrt{s/12N})}{1+\exp(-\pi\ell\sqrt{s/12N})},$$
     $$\var^L(|\kappa_j|)=\sum_{\ell\in \natu}\frac{q_N^\ell}{(1+q_N^\ell)^2}=\sum_{\ell\in \natu}\frac{ \exp(-\pi\ell\sqrt{s/12N})}{(1+\exp(-\pi\ell\sqrt{s/12N}))^2}.$$       
     Since $q_N^\ell/(1+q_N^\ell)$ is monotonically decreasing in $\ell$,
     $$\int_1^\infty \frac{\exp(-\pi\ell\sqrt{s/12N})}{1+\exp(-\pi\ell\sqrt{s/12N})}\rd \ell\leq \e^L[|\kappa_j|]\leq \int_0^\infty \frac{\exp(-\pi\ell\sqrt{s/12N})}{1+\exp(-\pi\ell\sqrt{s/12N})}\rd \ell.$$
     Recall that $\lambda=3\log(2)^2/\pi^2$. The difference between the above two integral is $O(1)$ and 
     $$\int_0^\infty \frac{\exp(-\pi\ell\sqrt{s/12N})}{1+\exp(-\pi\ell\sqrt{s/12N})}\rd \ell=\frac{\log(2)}{\pi}\sqrt{\frac{12N}{s}}=2\sqrt{\frac{\lambda N}{s}},$$
     so 
     \begin{equation}\label{eqn:ELkappaasym}
         \e^L[|\kappa_j|]\sim 2\sqrt{\frac{\lambda N}{s}}+O(1).
     \end{equation}
         Similarly, because $q^\ell_N<1$ and $x/(1+x^2)$ is monotonically increasing in $x$ over $[0,1]$, we know $q_N^\ell/(1+q_N^\ell)^2$ is monotonically decreasing in $\ell$ over $\ell\geq 0$ and 
         \begin{equation*}
             \int_1^\infty \frac{ \exp(-\pi\ell\sqrt{s/12N})}{(1+\exp(-\pi\ell\sqrt{s/12N}))^2}\rd\ell\leq \var^L(|\kappa_j|)\leq \int_0^\infty \frac{ \exp(-\pi\ell\sqrt{s/12N})}{(1+\exp(-\pi\ell\sqrt{s/12N}))^2}\rd\ell.
         \end{equation*} 
    The difference is again $O(1)$ and 
    $$\int_0^\infty \frac{ \exp(-\pi\ell\sqrt{s/12N})}{(1+\exp(-\pi\ell\sqrt{s/12N}))^2}\rd\ell=\frac{1}{\pi}\sqrt{\frac{3 N}{s}},$$
    so
    \begin{equation}\label{eqn:varLasym}
        \var^L(|\kappa_j|)\sim \frac{1}{\pi}\sqrt{\frac{3 N}{s}}+O(1).
    \end{equation}
    By Bernstein's inequality,
    $$\textstyle{\Pr^L}\Big(\Big||\kappa_j|-\e^L[|\kappa_j|]\Big|>t\Big)\leq 2\exp\Bigl(-\frac{t^2/2}{\var^L(|\kappa_j|)+t/3}\Bigr)$$
    for any $t>0$. Setting $t=\epsilon\sqrt{\lambda N/4s}$, we get
    $$
    \textstyle{\Pr^L}\Big(\Big||\kappa_j|-\e^L[|\kappa_j|]\Big|>\epsilon\sqrt{ \frac{\lambda N}{4s}}\Big)\leq 2\exp\Bigl(-\frac{\epsilon^2\lambda N/8s}{\var^L(|\kappa_j|)+\epsilon\sqrt{\lambda N/4s}/3}\Bigr)
    $$
    or equivalently
        $${\textstyle\Pr}^L\Big(\Big|\frac{|\kappa_j|}{\sqrt{\lambda N/s}}-\frac{\e^L[|\kappa_j|]}{\sqrt{\lambda N/s}}\Big|>\frac{\epsilon}{2}\Big)\leq 2\exp\Bigl(-\sqrt{\frac{\lambda N}{s}}\frac{\epsilon^2/8}{\var^L(|\kappa_j|)/\sqrt{\lambda N/s}+\epsilon/6}\Bigr).$$ 
    Because $\epsilon < 1$ and $\var^L(|\kappa_j|)\sim \pi^{-1}\sqrt{3 N/s} $, the right hand side can be bounded by $2\exp(-B_s \epsilon^2 \sqrt{N})$ for some $B_s>0$. If further $|\e^L[|\kappa_j|]/\sqrt{\lambda N/s}-2|< \epsilon/2$, 
$${\textstyle\Pr}^L \Bigl(\Big|\frac{|\kappa_j|}{\sqrt{\lambda N/s}}-2\Big|>\epsilon\Bigr)\leq {\textstyle\Pr}^L\Big(\Big|\frac{|\kappa_j|}{\sqrt{\lambda N/s}}-\frac{\e^L[|\kappa_j|]}{\sqrt{
\lambda N/s}}\Big|>\frac{\epsilon}{2}\Big)\leq 2\exp(-B_s \epsilon^2 \sqrt{N})$$
and we have proven the conclusion in view of equation~\eqref{eqn:PrlePrLforkappa}.
On the other hand, if $|\e^L[|\kappa_j|]/\sqrt{\lambda N/s}-2|\geq \epsilon/2$, equation~\eqref{eqn:ELkappaasym} implies
$$\epsilon^2\sqrt{N}\leq \frac{4}{\sqrt{\lambda/s}}\Big|\e^L[|\kappa_j|]-2\sqrt{\frac{\lambda N}{s}}\Big|=O(\sqrt{s}).$$
So by decreasing $B_s$ if necessary, we can assume $B_s\epsilon^2\sqrt{N}\leq 1$ for any $\epsilon$ satisfying $|\e^L[|\kappa_j|]/\sqrt{\lambda N/s}-2|\geq \epsilon/2$. After increasing $A_s$ if necessary so that $A_s\geq \exp(1)$,
$$A_sN^{1/4}\exp(-B_s \epsilon^2 \sqrt{N})\geq A_s\exp(-1)\geq 1$$
and the conclusion is trivially true.
\end{proof}

The proof of Lemma \ref{lem:sumk} is similar. By an abuse of notation, we let $\Pr$ and ${\textstyle\Pr}^L$ be the probability when $\bsk_1,...\bsk_r$ are sampled independently from $\dunif( Q_N)$ and $L(\bsk)$, respectively. An analogous argument using the importance sampling trick shows for any subset $A \subseteq (Q_N)^r$
\begin{equation}\label{eqn:rversionPrlePrL}
    \Pr((\bsk_1,\dots,\bsk_r)\in A)\leq {\textstyle\Pr}^L((\bsk_1,\dots,\bsk_r)\in A)\Bigl(\frac{1}{|Q_N|q_N^N}\prod_{\ell=1}^\infty (1+q_N^\ell)^s\Bigr)^r.
\end{equation}

\begin{proof}[Proof of Lemma \ref{lem:sumk}]
To simplify our notation, we write $\bsk^\oplus=\oplus_{i=1}^r\bsk_i$ with components $(k^\oplus_1,\dots, k^\oplus_s)$. By equation~\eqref{eqn:rversionPrlePrL} and Lemma~\ref{lem:likehoodratio},
    \begin{equation}\label{eqn:PrlePrLforsumk}
        \Pr\big(\bsk^\oplus\in Q_N\big)\leq A^r_s N^{r/4}{\textstyle\Pr}^L\big(\bsk^\oplus\in Q_N\big)\leq A^r_s N^{r/4}{\textstyle\Pr}^L\big(\Vert\bsk^\oplus\Vert_1\leq N\big).
    \end{equation}
    By the definition of $\bsk^\oplus$, $X^\oplus_{j\ell}=\bsone\{\ell\in \kappa^\oplus_j\}$ equals $1$ if and only if $\ell\in \kappa_j$ for an odd number of $\bsk$ among $\bsk_1,...,\bsk_r$.
    By a binomial distribution with success probability $q_N^\ell/(1+q_N^\ell)$,
    \begin{align}\label{eqn:oddbinomial}
        {\textstyle\Pr}^L(X^\oplus_{j\ell}=1)&=\sum_{j=1}^{\lceil r/2\rceil}{r\choose 2j-1}\Big(\frac{q_N^\ell}{1+q_N^\ell}\Big)^{2j-1}\Big(\frac{1}{1+q_N^\ell}\Big)^{r-2j+1}\\
        &=\frac{1}{2}-\frac{1}{2}\Bigl(\frac{1-q_N^\ell}{1+q_N^\ell}\Bigr)^r. \nonumber
    \end{align}
    Also notice that $\{X^\oplus_{j\ell}, j\in1{:}s, \ell\in \natu\}$ are jointly independent under $L(\bsk)$ and $$\Vert\bsk^\oplus\Vert_1=\sum_{j=1}^s\sum_{\ell\in\natu}\ell X^\oplus_{j\ell}.$$
 By Markov's inequality, for any $t>0$
    \begin{align}\label{eqn:chernoffbound}
        {\textstyle\Pr}^L\Big(\Vert\bsk^\oplus\Vert_1\leq N\big)&={\textstyle\Pr}^L\big(\exp\Big(-t\sum_{j=1}^s\sum_{\ell\in\natu}\ell X^\oplus_{j\ell}\Big)\geq e^{-tN}\Big)\\
        & \leq e^{tN} \e^L\Big[\exp\Big(-t\sum_{j=1}^s\sum_{\ell\in\natu}\ell X^\oplus_{j\ell}\Big)\Big]\nonumber\\
        & =e^{tN}\prod_{j=1}^s\prod_{\ell\in N}\Bigl(1-{\textstyle\Pr}^L(X^\oplus_{j\ell}=1)(1-e^{-t\ell })\Bigr)\nonumber\\
        &\leq \exp\Bigl(tN-s\sum_{\ell\in\natu}{\textstyle\Pr}^L(X^\oplus_{j\ell}=1)(1-e^{-t\ell })\Bigr).\nonumber
    \end{align}
    Because ${\textstyle\Pr}^L(X^\oplus_{j\ell}=1)$ is monotonically increasing in $r$ and $r\geq 2$,
    \begin{align*}
        \sum_{\ell\in\natu}{\textstyle\Pr}^L(X^\oplus_{j\ell}=1)(1-e^{-t\ell })
        &\geq \sum_{\ell\in\natu}\Bigl(\frac{1}{2}-\frac{1}{2}\Bigl(\frac{1-q_N^\ell}{1+q_N^\ell}\Bigr)^2\Bigr)(1-e^{-t\ell })\\
        &=\sum_{\ell\in\natu}\frac{1}{2}(1-e^{-t\ell })\frac{(1+q_N^\ell)^2-(1-q_N^\ell)^2}{(1+q_N^\ell)^2}\\
        &=\sum_{\ell\in\natu}2(1-e^{-t\ell })\frac{q_N^\ell}{(1+q_N^\ell)^2}.
    \end{align*}
    Setting $t=-\alpha\log(q_N)$ for $\alpha> 0$, which we will tune later, we have
    \begin{align*}
    \sum_{\ell\in\natu}2(1-e^{-t\ell })\frac{q_N^\ell}{(1+q_N^\ell)^2}
        &=2\sum_{\ell\in\natu}\frac{q_N^\ell}{(1+q_N^\ell)^2}-2\sum_{\ell\in\natu}\frac{q_N^{\alpha\ell}q_N^\ell}{(1+q_N^\ell)^2}.
    \end{align*}
    Similar to equation~\eqref{eqn:varLasym}, because both $q_N^\ell/(1+q_N^\ell)^2$ and $q_N^{\alpha\ell}q_N^\ell/(1+q_N^\ell)^2$ are monotonically decreasing in $\ell$ over $\ell\geq 0$,
    \begin{align*}
       \sum_{\ell\in\natu}\frac{q_N^\ell}{(1+q_N^\ell)^2}\sim &\int_0^\infty \frac{ \exp(-\pi\ell\sqrt{s/12N})}{(1+\exp(-\pi\ell\sqrt{s/12N}))^2}\rd\ell+O(1)\\
       =&\frac{1}{\pi}\sqrt{\frac{12 N}{s}}\int_0^\infty \frac{ \exp(-\ell)}{(1+\exp(-\ell))^2}\rd\ell+O(1)
    \end{align*}
    and
    \begin{align*}
        \sum_{\ell\in\natu}\frac{q_N^{\alpha\ell}q_N^\ell}{(1+q_N^\ell)^2}\sim &\int_0^\infty \frac{ \exp(-\pi(\alpha+1)\ell\sqrt{s/12N})}{(1+\exp(-\pi\ell\sqrt{s/12N}))^2}\rd\ell+O(1)\\
        =&\frac{1}{\pi}\sqrt{\frac{12 N}{s}}\int_0^\infty \frac{ \exp(-(\alpha+1)\ell)}{(1+\exp(-\ell))^2}\rd\ell+O(1). 
    \end{align*}
    Combining $t=-\alpha\log(q_N)=\alpha \pi\sqrt{s/12N}$ with the above equations, we get
    $$tN-s\sum_{\ell\in\natu}{\textstyle\Pr}^L(X^\oplus_{j\ell}=1)(1-e^{-t\ell })\leq tN-2s\sum_{\ell\in\natu}\frac{q_N^\ell-q_N^{\alpha\ell}q_N^\ell}{(1+q_N^\ell)^2}\sim c(\alpha)\sqrt{sN}+O(s)$$
    if $c(\alpha)\neq 0$ with 
    $$c(\alpha)=\alpha\frac{\pi}{\sqrt{12}}-\frac{4\sqrt{3 }}{\pi}\int_0^\infty \frac{\exp(-\ell)- \exp(-(\alpha+1)\ell)}{(1+\exp(-\ell))^2}\rd\ell.$$
    Because $c(\alpha)\to\infty$ as $\alpha\to \infty$ and
    $$c'(\alpha)=\frac{\pi}{\sqrt{12}}-\frac{4\sqrt{3}}{\pi}\int_0^\infty \frac{ \ell\exp(-(\alpha+1)\ell)}{(1+\exp(-\ell))^2}\rd\ell$$
    is strictly increasing in $\alpha$, we see $c(\alpha)$ has a unique minimum $\alpha^*$ over $\alpha\geq 0$. Furthermore,
     $$c'(0)=\frac{\pi}{\sqrt{12}}-\frac{4\sqrt{3}}{\pi}\int_0^\infty \frac{ \ell\exp(-\ell)}{(1+\exp(-\ell))^2}\rd\ell=\frac{\pi}{\sqrt{12}}-\frac{4\sqrt{3}\log(2)}{\pi}<0,$$
     so $\alpha^*>0$ and $c(\alpha^*)<0$. A numerical approximation using Mathematica shows $\alpha^*\approx 0.24$ and $c(\alpha^*)< -0.066$.
By choosing $t=-\alpha^*\log(q)$, we have shown
    $$\exp\Bigl(tN-s\sum_{\ell\in\natu}{\textstyle\Pr}^L(X^\oplus_{j\ell}=1)(1-e^{-t\ell })\Bigr)\leq \exp\Big(c(\alpha^*)\sqrt{sN}+O(s)\Big).$$
    Putting together equation~\eqref{eqn:PrlePrLforsumk} and equation~\eqref{eqn:chernoffbound}, we get
    $$\Pr\big(\bsk^\oplus\in Q_N\big)\leq A^r_sN^{r/4}\exp\Big(c(\alpha^*)\sqrt{sN}+O(s)\Big).$$
    For a threshold $R_s\geq 2$ that we will determine later, we can choose $B_s$ small enough so that $B_s \log(R_s)\leq -c(\alpha^*)\sqrt{s}$. By increasing $A_s$ if necessary to account for the $\exp(O(s))$ term, the conclusion holds for all $N\geq 1$ and $r\leq R_s$.

        It remains to show the conclusion holds for some $A_s, B_s>0$ when $r> R_s$ for some threshold $R_s$. Let $\ell^*$ be the largest $\ell\in \natu$ for which ${\textstyle\Pr}^L(X^\oplus_{j\ell}=1)\geq 1/4$. We conventionally set $\ell^*=0$ if ${\textstyle\Pr}^L(X^\oplus_{j\ell}=1)< 1/4$ for all $\ell\in\natu$. By equation~\eqref{eqn:oddbinomial},
    $${\textstyle\Pr}^L(X^\oplus_{j\ell}=1)=\frac{1}{2}-\frac{1}{2}\Bigl(\frac{1-q_N^\ell}{1+q_N^\ell}\Bigr)^r=\frac{1}{2}-\frac{1}{2}\Bigl(\frac{1-\exp(-\pi\ell\sqrt{s/12N})}{1+\exp(-\pi\ell\sqrt{s/12N})}\Bigr)^r.$$
    Because ${\textstyle\Pr}^L(X^\oplus_{j\ell}=1)$ is monotonically decreasing in $\ell$ over $\ell\geq 0$, $\ell^*$ equals the floor of the solution of ${\textstyle\Pr}^L(X^\oplus_{j\ell}=1)=1/4$. A straightforward calculation gives
    $$\ell^*=\lfloor\log\Bigl(\frac{1+2^{-1/r}}{1-2^{-1/r}}\Bigr)\sqrt{\frac{12N}{\pi^2 s}}\rfloor.$$
    By convexity of the function $f(x)=x^{-1/r}$,
    $$2^{-1-1/r}r^{-1}\leq 1-2^{-1/r}\leq r^{-1}.$$
    Hence
    $$r\leq\frac{1+2^{-1/r}}{1-2^{-1/r}}\leq (1+2^{-1/r})2^{1+1/r}r\leq 8r$$
    and
    \begin{equation}\label{eqn:lstarbound}
      \log(r)\sqrt{\frac{12N}{\pi^2 s}}-1\leq \ell^*\leq \log(8r)\sqrt{\frac{12N}{\pi^2 s}}.  
    \end{equation}
    Using the inequality $1-\exp(-x)\geq x/(1+x)$ when $x\geq 0$, equation~\eqref{eqn:chernoffbound} becomes
    \begin{align}\label{eqn:chernoffbound2}
        {\textstyle\Pr}^L\big(\Vert\bsk^\oplus\Vert_1\leq N\big)&\leq \exp\Bigl(tN-s\sum_{\ell\in\natu}{\textstyle\Pr}^L(X^\oplus_{j\ell}=1)\frac{t\ell}{1+t\ell}\Bigr)\\
    &\leq \exp\Bigl(tN-s\sum_{\ell\in \natu}\bsone\{\ell\leq\ell^*\}\frac{t\ell}{4(1+t\ell)}\Bigr)\nonumber\\
    &= \exp\Bigl(tN-\frac{st\ell^*(\ell^*+1)}{8(1+t\ell^*)}\Bigr).\nonumber
    \end{align}
    Setting $t=\sqrt{s/N}$, we derive from equation~\eqref{eqn:lstarbound} that there exists a large enough $R_s$ so that for all $r\geq R_s$,
    $$1+t\ell^*\leq 1+\log(8r)\sqrt{\frac{12}{\pi^2 }}<2\log(r)\sqrt{\frac{12}{\pi^2 }}$$
    and
    $$st\ell^*(\ell^*+1)\geq  \sqrt{sN}\log(r)\sqrt{\frac{12}{\pi^2 }}\Bigl(\log(r)\sqrt{\frac{12}{\pi^2 }}-\sqrt{\frac{s}{N}}\Bigr)> \sqrt{sN}\log(r)^2\frac{6}{\pi^2}.$$
    By increasing $R_s$ if necessary, we further have for all $r\geq R_s$
    $$tN-\frac{st\ell^*(\ell^*+1)}{8(1+t\ell^*)}\leq \sqrt{sN}-\sqrt{sN}\log(r)\frac{\sqrt{3}}{16\pi}<-\sqrt{sN}\log(r)\frac{\sqrt{3}}{32\pi}.$$
    Putting together equation~\eqref{eqn:PrlePrLforsumk} and equation~\eqref{eqn:chernoffbound2}, we get for $r\geq R_s$
    $$\Pr\big(\bsk^\oplus\in Q_N\big)\leq A^r_sN^{r/4}\exp(-\sqrt{sN}\log(r)\frac{\sqrt{3}}{32\pi}),$$
    which completes the proof.
\end{proof}

\section{Proof of Theorem 6}
By our assumption, we assume without loss of generality that
\begin{equation}\label{eqn:thm6}
\liminf_{N\to\infty}\frac{|Q_{N,c,\beta,\cj}(f)|}{|Q_{N}|} >0     
\end{equation}
for some $c>0$ and $\beta\in (0,1]$.

 We first define $N_{m,\beta}=N_m-D_\beta m$ for a large enough $D_\beta\in \natu$, which we will specify later.  By $N_m\sim \lambda m^2/s$ and the mean value theorem,
\begin{equation*}
    \sqrt{\frac{\lambda N_m}{s}}-\sqrt{\frac{\lambda N_{m,\beta}}{s}}\leq \sqrt{\frac{\lambda}{s}}  \frac{1}{2\sqrt{ N_{m,\beta} }}D_\beta m\sim \frac{1}{2}D_\beta. 
\end{equation*}
Hence for large $m$, we have $\sqrt{\lambda N_m/s}-\sqrt{\lambda N_{m,\beta}/s} \leq D_\beta$. Furthermore,  equation~\eqref{eqn:QNasymp} implies 
\begin{equation}\label{eqn:QNbetacount}
    \lim_{m\to\infty}\frac{|Q_{N_{m,\beta} }|}{|Q_{N_m}|}= c_{\beta}>0,
\end{equation}
where $c_{\beta}$ is a constant depending on $D_\beta$ and $s$.

Next, we define for $m \geq 1$
\begin{equation*}
    \Tilde{Q}_{m,\beta}=\Bigl\{\bsk\in Q_{N_{m,\beta}}\Big \vert \hspace{1em}\Bigl|\frac{|\kappa_j|}{\sqrt{\lambda N_{m,\beta}/s}}-2\Bigr|>m^{-1/4}\text{ for some } j \in 1{:}s\Bigr\}.
\end{equation*}
When $\bsk \in Q_{N_{m,\beta}}\setminus \Tilde{Q}_{m,\beta}$, Stirling's formula implies for large $m$
\begin{align*}
    \max_{J\in \cj} \prod_{j\in J} (|\kappa_j|)!\geq &\Bigl(\big(\lceil (2-m^{-1/4})\sqrt{\lambda N_{m,\beta}/s}\rceil\big)!\Bigr)^{\cj_{\max}}\\
    \geq & \Bigl(\big(\lceil (2-m^{-1/4})\sqrt{\lambda N_m/s}\rceil-2D_\beta\big)!\Bigr)^{\cj_{\max}}\\
    \geq & \Bigl(\big(\lceil (2-m^{-1/4})\sqrt{\lambda N_m/s}\rceil\big)! \big(2\sqrt{\lambda N_m/s}\big)^{-2D_\beta}\Bigr)^{\cj_{\max}}\\
     \geq &2^{2\cj_{\max}\log_2(m)\sqrt{\lambda N_m/s}-\cj_{\max}K_sm} \Big(2\sqrt{\lambda N_m/s}\Big)^{-2\cj_{\max}D_\beta}
\end{align*}
for sufficiently large $K_s$ depending on $s$. By our assumption and the above lower bound, for large $m$, $\bsk \in Q_{ N_{m,\beta},c,\beta,\cj}(f)\setminus \Tilde{Q}_{m,\beta}$ satisfy
\begin{align*}
    &|\hat{f}(\bsk)|\\
    \geq & c2^{-\Vert\bskappa\Vert_1}\beta^{\Vert\bskappa\Vert_0}\max_{J\in \cj} \prod_{j\in J} (|\kappa_j|)!\\
    \geq & c2^{-N_{m,\beta}} \beta^{s (2+m^{-1/4})\sqrt{\lambda N_{m,\beta}/s} } 2^{2\cj_{\max}\log_2(m)\sqrt{\lambda N_m/s}-\cj_{\max}K_s m  }\Big(2\sqrt{\lambda N_m/s}\Big)^{-2\cj_{\max}D_\beta} \\
     \geq &c 2^{-N_m+2\cj_{\max}\log_2(m)\sqrt{\lambda N_m/s}+D_\beta [m-2\cj_{\max}\log_2(2\sqrt{\lambda N_m/s})]+3s\log_2(\beta)\sqrt{\lambda N_m/s}-\cj_{\max}K_s m}.
\end{align*}
Comparing the above bound with
$$T_m=K2^{-N_m+2 \cj_{\max}\log_2(m)\sqrt{\lambda N_m/s}+D_{s,\alpha,\cj}m},$$
we can lower bound $|Q_{N_m}\cap \Lambda(T_m)|$ by
$$|Q_{N_m}\cap \Lambda(T_m)|\geq |Q_{ N_{m,\beta},c,\beta,\cj}(f)\setminus \Tilde{Q}_{m,\beta}| \geq |Q_{ N_{m,\beta},c,\beta,\cj}(f)|-|\Tilde{Q}_{m,\beta} |$$
for large $m$ whenever we choose a sufficiently large $D_\beta\in \natu$ such that
$$D_\beta \left[m-2\cj_{\max}\log_2(2\sqrt{\lambda N_m/s})\right]+3s\log_2(\beta)\sqrt{\lambda N_m/s}-\cj_{\max}K_s m> D_{s,\alpha,J}m$$
for large $m$. By Lemma~\ref{lem:meankappa}, 
$$\limsup_{m\to\infty}\frac{|\Tilde{Q}_{m,\beta} |}{|Q_{N_{m,\beta}}|}\leq \limsup_{m\to\infty} sA'_sN_{m,\beta}^{1/4}\exp(-B'_s m^{-1/2} \sqrt{N_{m,\beta}})=0.$$
Hence, we conclude from equations~\eqref{eqn:thm6} and \eqref{eqn:QNbetacount} that
$$ \liminf_{m\to\infty}\frac{|Q_{N_m}\cap \Lambda(T_m)|}{|Q_{N_m}|}\geq \liminf_{m\to\infty}\frac{|Q_{ N_{m,\beta},c,\beta,\cj}(f)|-|\Tilde{Q}_{m,\beta} |}{|Q_{N_{m,\beta}}|} \frac{|Q_{N_{m,\beta}}|}{|Q_{N_m}|}>0.$$

\section{Proof of Theorem~\ref{thm:superpoly}}
Given $\gamma>0$, let
        $$N_{m,\gamma}=\sup \{N\in \natu\mid |Q_N|\le m^{-\gamma} 2^m \}.$$
        By equation~\eqref{eqn:QNasymp}, 
        \begin{equation}\label{eqn:Nmgamma}
    N_{m,\gamma}\sim \lambda m^2/s +\lambda(1-2\gamma) m\log_2(m)/s \quad\text{for} \quad\lambda=3(\log 2)^2/\pi^2. 
\end{equation}
Next, let $\hat{\mu}_{m,\infty}-\mu=\mathrm{SUM}_{1,\gamma}+\mathrm{SUM}_{2,\gamma}$ with
\begin{equation*}
    \mathrm{SUM}_{1,\gamma}=\sum_{\bsk\in Q_{N_{m,\gamma}} }Z(\bsk)S(\bsk)\hat{f}(\bsk),
\end{equation*}
\begin{equation*}
    \mathrm{SUM}_{2,\gamma}=\sum_{\bsk\in  \natu_*^s\setminus  Q_{N_{m,\gamma}}}Z(\bsk)S(\bsk)\hat{f}(\bsk).
\end{equation*}
Because $|Q_{N_{m,\gamma}}|\leq m^{-\gamma}2^m$ and $\Pr(Z(\bsk)=1)=2^{-m}$ for all $\bsk\in Q_{N_{m,\gamma}}$,
\begin{align*}
    \Pr\Big(Z(\bsk)=1 \text{ for any } \bsk\in Q_{N_{m,\gamma}}\Big)=&\Pr\Big(\sum_{\bsk\in Q_{N_{m,\gamma}} }Z(\bsk)\geq 1\Big)\\
    \leq & \e \Big[\sum_{\bsk\in Q_{N_{m,\gamma}} }Z(\bsk)\Big]
    \leq  m^{-\gamma}.
\end{align*}
Therefore, $\mathrm{SUM}_{1,\gamma}=0$ with probability at least $1-m^{-\gamma}$. 

By Remark \ref{rmk:generalNm}, we can apply Corollary \ref{cor:SUM2bound} to $\mathrm{SUM}_{2,\gamma}$ with $N_{m,\gamma}$ replacing $N_m$ and obtain
$$\lim_{m\to\infty} m^{\gamma} \Pr\Big(|\mathrm{SUM}_{2,\gamma}|\geq T_{m,\gamma}\Big)=0$$
for 
$$T_{m,\gamma}=K 2^{-N_{m,\gamma}+2 \cj_{\max}\log_2(m)\sqrt{\lambda N_{m,\gamma}/s}+D_{s,\alpha,\gamma,\cj}m}$$
with a sufficiently large $D_{s,\alpha,\gamma,\cj}$ depending on $s,\alpha,\gamma,\cj_{\max}$. In view of equation~\eqref{eqn:Nmgamma}, we can further find $\Gamma$ depending on $s,\gamma,\cj_{\max},D_{s,\alpha,\gamma,\cj}$ such that 
$${-N_{m,\gamma}+2 \cj_{\max}\log_2(m)\sqrt{\lambda N_{m,\gamma}/s}+D_{s,\alpha,\gamma,\cj}m} \leq -\lambda m^2/s +\Gamma m\log_2(m) $$
for large enough $m$. Our conclusion then follows by taking the union bound over the probability of $\mathrm{SUM}_{1,\gamma}\neq 0$ and $|\mathrm{SUM}_{2,\gamma}|\geq T_{m,\gamma}$.

\section{Proof of Theorem~\ref{thm:Rmr}}
Recall that $\lceil\kappa\rceil$ is the largest element of $\kappa\subseteq \natu$ and $\lceil\bskappa\rceil=\max_{j\in 1{:}s}\lceil\kappa_j\rceil$. For a matrix $C$, $C(\ell, :)$ denotes its $\ell$-th row.

    We first prove the result for $R_{m,1}$. For any nonzero $\bsk$, if $\lceil\kappa_{j^*}\rceil >m$ for a $j^*\in 1{:}s$, we can find $\ell\in \kappa_{j^*}$ such that $\ell>m$ and $M_{j^*}(\ell,:)$ follows a $\dunif(\{0,1\}^m)$ distribution. Because $\mathcal{C}_{j^*}$ is nonsingular, $C_{j^*}(\ell,:)=M_{j^*}(\ell,:)\mathcal{C}_{j^*}$ also follows a $\dunif(\{0,1\}^m)$ distribution and $\sum_{j=1}^s \sum_{\ell\in \kappa_j}  C_j(\ell,:)=\bszero \tmod 2$ occurs with probability $2^{-m}$. Hence, it suffices to consider the maximum of $\Pr(Z(\bsk)=1\mid \mathcal{C}_j, j\in 1{:}s)$ over all nonzero $\bsk$ with $\lceil\bskappa\rceil\leq m$.

Instead of sampling $\mathcal{C}_j$ directly  from $\dunif(\mathcal{I}_m)$, we first sample $\mathcal{C}^*_j$ uniformly from all $m\times m$ matrices over $\mathbb{F}_2$ and then view $\mathcal{C}_j$ as $\mathcal{C}^*_j$ conditioned on $\mathcal{C}^*_j\in \mathcal{I}_m$. For each $j\in 1{:}s$, the probability $\mathcal{C}^*_j\in\mathcal{I}_m$ is given by $\prod_{\ell=1}^m (1-2^{-m+\ell-1})$, because there are $2^m-2^{\ell-1}$ choices for the $\ell$-th row of $\mathcal{C}^*_j$ that are linearly independent of previous rows. We notice this probability is monotonically decreasing in $m$ and 
$$\lim_{m\to\infty}\prod_{\ell=1}^m (1-2^{-m+\ell-1})=\prod_{\ell=1}^\infty (1-2^{-\ell})\geq \exp(-\sum_{\ell=1}^\infty \frac{2^{-\ell}}{1-2^{-\ell}})\geq \exp(-2), $$
where we have used $\log(1-x)\geq - x/(1-x)$ for $x\in (0,1)$.

Now set  $C^*_j=M_j \mathcal{C}^*_j$ and 
\begin{equation}\label{eqn:Zkstar}
    Z^*(\bsk)=\bsone\Big\{ \sum_{j=1}^s \sum_{\ell\in \kappa_j}  C^*_j(\ell,:)=\bszero \tmod 2\Big\}=\bsone\Big\{\sum_{j=1}^s \sum_{\ell\in \kappa_j} M_j(\ell,:) \mathcal{C}^*_j=\bszero \tmod 2\Big\}.
\end{equation}
When $\kappa_j\neq\emptyset$ and $\lceil\kappa_j\rceil\leq m$, the sum $\sum_{\ell\in \kappa_j} M_j(\ell,:)$ is nonzero, and consequently $ \sum_{\ell\in \kappa_j} M_j(\ell,:) \mathcal{C}^*_j$ follows a $\dunif(\{0,1\}^m)$ distribution. Hence when $\bsk\neq \bszero$ and $\lceil\bskappa\rceil\leq m$,
$$2^{-m}=\Pr(Z^*(\bsk)=1)\geq \Pr(\mathcal{C}^*_j\in \mathcal{I}_m,j\in 1{:}s)\Pr(Z^*(\bsk)=1\mid \mathcal{C}^*_j\in \mathcal{I}_m,j\in 1{:}s).$$
Conditioned on $\mathcal{C}^*_j\in \mathcal{I}_m$ for all $j\in 1{:}s$, $Z(\bsk)$ has the same distribution as $Z^*(\bsk)$. Therefore
\begin{equation}\label{eqn:Zkbound}
    \Pr(Z(\bsk)=1)\leq \frac{1}{\Pr(\mathcal{C}^*_j\in \mathcal{I}_m,j\in 1{:}s)} 2^{-m} \leq \exp(2s)2^{-m}.
\end{equation}
Because $\Pr(Z(\bsk)=1)=\e[\Pr(Z(\bsk)=1\mid \mathcal{C}_j, j\in 1{:}s )]$,  the Markov's inequality shows for each nonzero $\bsk$
\begin{equation}\label{eqn:largeZkbound}
    \Pr\Big(\Pr(Z(\bsk)=1\mid \mathcal{C}_j, j\in 1{:}s)> 2^{-m+R}\Big)\leq \exp(2s)2^{-R}
\end{equation}
for $R\geq 0$. 

Next, observe that $\Pr(Z(\bsk)=1\mid \mathcal{C}_j, j\in 1{:}s)$ depends on $\bsk$ only through the values $\{\lceil\kappa_j\rceil, j\in 1{:}s\}$. Indeed,
$$\sum_{j=1}^s \sum_{\ell\in \kappa_j}  C_j(\ell,:)=\sum_{j=1}^s \Big(\sum_{\ell\in\kappa_j}M_j(\ell,:)\Big) \mathcal{C}_j,$$
and because $M_j$ is lower-triangular,  $\sum_{\ell\in\kappa_j}M_j(\ell,:)\overset{d}{=}M_j(\lceil\kappa_j\rceil,:)$. When $\lceil\bskappa\rceil\leq m$, each $\lceil\kappa_j\rceil$ can take any value in $\{0,1,\dots,m\}$giving at most $(m+1)^s$ distinct combinations. Applying a union bound over these combinations with $R=3s\log_2(m+1)$ yields
\begin{align*}
 \Pr\Big(\sup_{\bsk\neq\bszero}\Pr(Z(\bsk)=1\mid \mathcal{C}_j, j\in 1{:}s)> 2^{-m+R}\Big)\leq & (m+1)^s \exp(2s)2^{-R}\\
 =& \frac{\exp(2s)}{(m+1)^{2s}}.   
\end{align*}
By the definition of $1$-way rank deficiency, $R_{m,1}\leq R$ whenever $\sup_{\bsk\neq\bszero}\Pr(Z(\bsk)=1\mid \mathcal{C}_j, j\in 1{:}s)\le 2^{-m+R}$. This completes the proof.

The result for $R_{m,r}$ when $r\geq 2$ can be established similarly. Let $V=\{\bsk_1,\dots,\bsk_r\}$ have rank $r$ with $\bsk_i=(k_{1,i},\dots,k_{s,i})$.

\textbf{Case 1.} Suppose $\lceil\kappa_{j^*,i^*}\rceil>m$ for some $j^*\in 1{:}s$ and $i^*\in 1{:}r$. After an invertible linear transformation on $V$ if necessary, we can find $\ell\in \kappa_{j^*,1}$ such that $\ell>m$ and $\ell\notin \kappa_{j^*,i}$ for all $i\in 2{:}r$. Condition on all random entries of $\mathcal{C}_j,j\in 1{:}s$ and $M_j, j\in 1{:}s$ other than $M_{j^*}(\ell,:)$. Then $Z(\bsk_{i})$ becomes deterministic for $i\in 2{:}r$ and $Z(\bsk_1)=1$ with $2^{-m}$ probability because $\mathcal{C}_{j^*}$ is nonsingular and $C_{j^*}(\ell,:)=M_{j^*}(\ell,:)\mathcal{C}_{j^*}$ follows a $\dunif(\{0,1\}^m)$ distribution. Hence 
\begin{align}\label{eqn:ZkVbound}
    \Pr(Z(\bsk)=1 \text{ for all }\bsk\in V)=2^{-m}\Pr(Z(\bsk_i)=1, i\in 2{:}r )\leq 2^{-mr+R_{m,r-1}}.
\end{align}

\textbf{Case 2.} Suppose $\max_{i\in 1{:}r}\lceil\bskappa_{i}\rceil\leq m$. After an invertible linear transformation on $V$ if necessary, we can find $j^*\in 1{:}s$ such that $\lceil\kappa_{j^*,1}\rceil>\lceil\kappa_{j^*,i}\rceil$ for all $i\in 2{:}r$. Denote $\ell^*=\lceil\kappa_{j^*,1}\rceil$. 

As before, we sample $\{\mathcal{C}^*_j, j\in 1{:}s\}$ independently from the uniform distribution on $m\times m$ matrices over $\mathbb{F}^2$ and define $Z^*(\bsk)$ by equation~\eqref{eqn:Zkstar}. Condition  on all random entries of $\{M_j, j\in 1{:}s\}$ and $\{\mathcal{C}^*_j,j\in 1{:}s\}$ other than $\mathcal{C}^*_{j^*}(\ell^*,:)$. Then $Z^*(\bsk_{i})$ is deterministic for $i\in 2{:}r$,
and because $M_{j^*}$ has ones one the diagonal,
$$M_{j^*}(\ell^*,:)\mathcal{C}^*_{j^*}=\mathcal{C}^*_{j^*}(\ell^*,:)+\sum_{\ell<\ell^*}M_{j^*}(\ell^*,\ell)\mathcal{C}^*_{j^*}(\ell,:)$$
follows a $\dunif(\{0,1\}^m)$ distribution. Hence $Z^*(\bsk_1)=1$ with $2^{-m}$ probability and
\begin{align*}
    \Pr(Z^*(\bsk)=1 \text{ for all }\bsk\in V)=2^{-m}\Pr(Z^*(\bsk_i)=1, i\in 2{:}r ).
\end{align*}
Applying this argument inductively to $V'=\{\bsk_2,\dots,\bsk_r\}$ yields 
$$\Pr(Z^*(\bsk)=1 \text{ for all }\bsk\in V)=2^{-mr}.$$
Following the same reasoning that led to equations~\eqref{eqn:Zkbound} and \eqref{eqn:largeZkbound}, we obtain
$$\Pr(Z(\bsk)=1 \text{ for all }\bsk\in V) \leq \exp(2sr) 2^{-mr}$$
and, for $R\geq 0$,
\begin{equation}\label{eqn:largeZkVbound}
   \Pr\Big(\Pr(Z(\bsk)=1 \text{ for all }\bsk\in V\mid \mathcal{C}_j, j\in 1{:}s)> 2^{-mr+R}\Big)\leq \exp(2sr)2^{-R}. 
\end{equation}
Finally, for each $j\in 1{:}s$ and $u\subseteq 1{:}r$, we define 
$$\kappa_{j,u}=\{\ell\in \natu\mid \ell\in \kappa_{j,i} \text{ for } i\in u, \ell\notin \kappa_{j,i} \text{ for } i\in 1{:}r \setminus u\}.$$
These sets are disjoint for different $u$. By the lower-triangular structure of $M_j$,
$$\sum_{\ell\in \kappa_{j,i}} M_j(\ell,:) \mathcal{C}_j=\sum_{\substack{u\subseteq 1{:}r\\i\in u}}\Big(\sum_{\ell\in \kappa_{j,u}} M_j(\ell,:) \Big)\mathcal{C}_j\overset{d}{=}\sum_{\substack{u\subseteq 1{:}r\\i\in u}} M_j(\lceil\kappa_{j,u}\rceil,:)\mathcal{C}_j.$$
Therefore $\Pr(Z(\bsk)=1 \text{ for all }\bsk\in V\mid \mathcal{C}_j, j\in 1{:}s)$ equals the probability that
$$\sum_{j=1}^s\sum_{\substack{u\subseteq 1{:}r\\i\in u}} M_j(\lceil\kappa_{j,u}\rceil,:)\mathcal{C}_j= \bszero\tmod 2 \text{ for all }i\in 1{:}r.$$
There are $2^r-1$ nonempty subsets $u\subseteq 1{:}r$. When $\max_{i\in 1{:}r}\lceil\bskappa_{i}\rceil\leq m$, each $\lceil\kappa_{j,u}\rceil$ can take any value in $\{0,1,\dots,m\}$, giving at most $(m+1)^{(2^r-1)s}$ combinations of $\{\lceil\kappa_{j,u}\rceil,j\in 1{:}s, u\subseteq 1{:}r\}$. Applying equation~\eqref{eqn:largeZkVbound} together with a union bound over these combinations, we obtain that the probability of
$$\sup_{\substack{V=(\bsk_1,\dots,\bsk_r)\in \mathbb{V}_r\\ \max_{i\in 1{:}r}\lceil\bskappa_{i}\rceil\leq m}}\Pr(Z(\bsk)=1 \text{ for all }\bsk\in V\mid \mathcal{C}_j, j\in 1{:}s)> 2^{-mr+R} $$
is bounded by $(m+1)^{s(2^r-1)}\exp(2sr)2^{-R}$, which equals $\exp(2sr)(m+1)^{-2sr}$ when $R=(2^{r}+2r-1) s\log_2(m+1)$. 

Finally, combining this estimate with the bound from Case 1 (equation~\eqref{eqn:ZkVbound}) completes the proof.

\section{Proof of Lemma \ref{lem:generalSUM1bound}}

We first present the following technical lemma, which we will use to show most $\bsk\in Q_{N_m}$ satisfy $\Pr(Z(\bsk)=1)=2^{-m}$ even when the marginal order of the randomization scheme is nonzero.

\begin{lemma}\label{lem:maxkappa}
For $L\geq 0$ and $\kappa\subseteq\natu$, define $\kappa^{>L}=\{\ell\in \kappa\mid \ell>L \}.$
Let $N\geq 1$ and $\bsk_1,\bsk_2$ be sampled independently from $\dunif (Q_N)$. Then for any $\rho>0$, there exist positive constants $A_{\rho,s}, B_{\rho,s}$ depending on $\rho,s$ such that for each $j\in 1{:}s$,
$$\Pr\Big(\kappa^{>\rho\sqrt{N}}_{j,1}=\emptyset\Big)\leq A_{\rho,s}N^{1/4}\exp(- B_{\rho,s}\sqrt{N})$$
and 
$$\Pr\Big(\kappa^{>\rho\sqrt{N}}_{j,1}=\kappa^{>\rho\sqrt{N}}_{j,2}\Big)\leq A^2_{\rho,s}N^{1/2}\exp(- 2B_{\rho,s}\sqrt{N}).$$
\end{lemma}

\begin{proof}
    By equation~\eqref{eqn:PrlePrL} and Lemma~\ref{lem:likehoodratio},
$$\Pr\Big(\kappa^{>\rho \sqrt{N}}_{j,1}=\emptyset\Big)\leq A_s N^{1/4}{\textstyle\Pr}^L\Big(\kappa^{>\rho \sqrt{N}}_{j,1}=\emptyset\Big).$$
Because $\kappa^{>\rho \sqrt{N}}_{j,1}=\emptyset$ if and only if $\ell\notin \kappa_{j,1}$ for all $\ell>\rho\sqrt{N}$, 
$${\textstyle\Pr}^L\Big(\kappa^{>\rho \sqrt{N}}_{j,1}=\emptyset\Big)=\prod_{\ell=\lceil\rho\sqrt{N}\rceil}^\infty\frac{1}{1+q^\ell_N}=\exp\Big(-\sum_{\ell=\lceil\rho\sqrt{N}\rceil}^\infty\log(1+q^\ell_N)\Big).$$
Because $\log(1+q^\ell_N)$ is monotonically decreasing in $\ell$, 
\begin{align}\label{eqn:crhobound}
        \sum_{\ell=\lceil\rho\sqrt{N}\rceil}^\infty \log(1+q^\ell_N)
        \geq&\int_{\lceil\rho\sqrt{N}\rceil}^\infty \log\Big(1+\exp(-\pi\ell\sqrt{s/12N})\Big)\rd \ell \\
        \geq &c_{\rho,s}\sqrt{N}-\log(2) \nonumber
    \end{align}
for 
\begin{equation*}
c_{\rho,s}=\frac{1}{\pi}\sqrt{\frac{12}{s}}\int_{\pi\rho\sqrt{s/12}}^\infty \log\Big(1+\exp(-\ell)\Big)\rd \ell.
\end{equation*}
Hence
$$\Pr\Big(\kappa^{>\rho \sqrt{N}}_{j.1}=\emptyset\Big)\leq 2A_s N^{1/4}\exp\Big(-c_{\rho,s}\sqrt{N}\Big).$$

Similarly, by equation~\eqref{eqn:rversionPrlePrL} and Lemma~\ref{lem:likehoodratio},
$$\Pr\Big(\kappa^{>\rho \sqrt{N}}_{j,1}=\kappa^{>\rho \sqrt{N}}_{j,2}\Big)\leq A^2_s N^{1/2}{\textstyle\Pr}^L\Big(\kappa^{>\rho \sqrt{N}}_{j,1}=\kappa^{>\rho \sqrt{N}}_{j,2}\Big).$$
Because $\kappa^{>\rho \sqrt{N}}_{j,1}=\kappa^{>\rho \sqrt{N}}_{j,2}$ if and only if each $\ell>\rho \sqrt{N}$ either appears in both of or neither of $\kappa_{j,1},\kappa_{j,2}$,
\begin{align*}
{\textstyle\Pr}^L(\kappa^{>\rho \sqrt{N}}_{j,1}=\kappa^{>\rho \sqrt{N}}_{j,2})
=& \prod_{\ell=\lceil\rho\sqrt{N}\rceil}^\infty\frac{1+q^{2\ell}_N}{(1+q^\ell_N)^2}\\
=&\exp\Big(\sum_{\ell=\lceil\rho\sqrt{N}\rceil}^\infty\log(1+q^{2\ell}_N)-\sum_{\ell=\lceil\rho\sqrt{N}\rceil}^\infty 2\log(1+q^\ell_N)\Big).
\end{align*}
By the monotonicity of $\log(1+q^{2\ell}_N)$,
\begin{align*}
        \sum_{\ell=\lceil\rho\sqrt{N}\rceil}^\infty \log(1+q^{2\ell}_N)
        \leq&\int_{\lceil\rho\sqrt{N}\rceil-1}^\infty \log\Big(1+\exp(-2\pi\ell\sqrt{s/12N})\Big)\rd \ell \\
        \leq &c'_{\rho,s}\sqrt{N} +\log(2)
    \end{align*}
    for 
    \begin{equation*}
c'_{\rho,s}=\frac{1}{2\pi}\sqrt{\frac{12}{s}}\int_{2\pi\rho\sqrt{s/12}}^\infty \log\Big(1+\exp(-\ell)\Big)\rd \ell.
\end{equation*}
Notice that $c'_{\rho,s}<c_{\rho,s}/2$. Along with equation~\eqref{eqn:crhobound}, we get the bound
$$\Pr\Big(\kappa^{>\rho \sqrt{N}}_{j,1}=\kappa^{>\rho \sqrt{N}}_{j,2}\Big)\leq 8 A^2_s N^{1/2} \exp\Big(-2(c_{\rho,s}-\frac{1}{2}c'_{\rho,s})\sqrt{N}\Big). $$
Our conclusion follows by taking $A_{\rho,s}=2\sqrt{2}A_s$ and $B_{\rho,s}=c_{\rho,s}-c'_{\rho,s}/2>(3/4)c_{\rho,s}.$
\end{proof}

\begin{proof}[Proof of Lemma \ref{lem:generalSUM1bound}]
    Let $L_{m,d}=\{\bsk\in L_m\mid \lceil\bskappa\rceil\le dm\}$. Because $N_m\sim\lambda m^2/s$, we can find a constant $\rho$ depending on $d$ and $s$ such that $dm\leq \rho \sqrt{N_m}$ for all sufficiently large $m$. Lemma~\ref{lem:maxkappa} then yields 
$$|L_{m,d}|=|Q_{N_m}|\frac{|L_{m,d}|}{|Q_{N_m}|}\leq \frac{1}{2}\log_2(m) 2^{m}A_{\rho,s}N_m^{1/4}\exp\Big(-B_{\rho,s}\sqrt{N_m}\Big).$$
Let $\mathcal{A}=\{Z(\bsk)=1\text{ for some } \bsk\in L_{m,d}\}$. By a union bound argument,
\begin{align*}
\Pr(\mathcal{A})\leq 2^{-m+R_{m,1}}|L_{m,d}| 
\leq \frac{1}{2}\log_2(m)2^{R_{m,1}}A_{\rho,s}N_m^{1/4}\exp\Big(-B_{\rho,s}\sqrt{N_m}\Big),
\end{align*}
which converges to $0$ since $\lim_{m\to\infty}R_{m,1}/m=0$. Similarly, we define 
$$L'_{m,d}=\{(\bsk_1,\bsk_2)\in L_{m}\times L_m\mid \kappa^{>dm}_{j,1}=\kappa^{>dm}_{j,2} \text{ for all } j\in 1{:}s\}.$$
An analogous application of Lemma~\ref{lem:maxkappa} shows that $2^{-2m+R_{m,2}}|L'_{m,d}|$ also converges to $0$.

For $\bsk_1\in L_m\setminus L_{m,d}$, we can find $\ell_1,j_1$ such that $\ell_1>dm,\ell_1\in\kappa_{j_1}$. By the definition of marginal order, $C_{j_1}(\ell_1,:)$ is independently drawn from $\dunif(\{0,1\}^m)$, so $Z(\bsk)$ is independent of $\mathcal{A}$ and $$\Pr(Z(\bsk)\mid \mathcal{A}^c)=2^{-m}.$$
Now suppose  $\bsk_1,\bsk_2\in L_{m}\setminus L_{m,d}$ and $(\bsk_1,\bsk_2)\notin L'_{N_m,d}$. After replacing $\bsk_2$ by $\bsk_1\oplus\bsk_2$ if necessary, we can assume there exist $\ell_1,\ell_2,j_1,j_2$ such that 
$$\ell_1>dm,\ell_1\in \kappa_{j_1,1},\ell_1\notin \kappa_{j_1,2}, \quad \ell_2>dm, \ell_2\in \kappa_{j_2,2},\ell_2\notin \kappa_{j_2,1}.$$ Because $C_{j_1}(\ell_1,:)$ and $C_{j_2}(\ell_2,:)$ are independently drawn from $\dunif(\{0,1\}^m)$, $\{Z(\bsk_1),Z(\bsk_2),\mathcal{A}\}$ are jointly independent and the conditional covariance satisfies $$\cov(Z(\bsk_1),Z(\bsk_2)\mid \mathcal{A}^c)=0.$$
Therefore, 
$$\e\left[\sum_{\bsk\in L_{m}}Z(\bsk)\Big\vert \mathcal{A}^c\right]=\sum_{\bsk \in L_{m}\setminus L_{m,d}}\Pr(Z(\bsk)\mid \mathcal{A}^c)=2^{-m}(|K_{m}|-|L_{m,d}|),$$
and
\begin{align*}
&\var\left(\sum_{\bsk\in L_{m}}Z(\bsk)\Big\vert \mathcal{A}^c\right) \\
\le & \sum_{\bsk \in L_{m}\setminus L_{m,d}} \var(Z(\bsk)\mid \mathcal{A}^c)+\sum_{(\bsk_1,\bsk_2)\in L'_{m,d}} \e[Z(\bsk_1)Z(\bsk_2)\mid \mathcal{A}^c ]\\
\leq & 2^{-m}(|L_{m}|-|L_{m,d}|)+\frac{1}{\Pr(\mathcal{A}^c)} 2^{-2m+R_{m,2}}|L'_{m,d}|.
\end{align*}
Since $2^{-m}|L_m|\to\infty$, $2^{-m}|L_{m,d}|\to 0$, $\Pr(\mathcal{A}^c)\to 1$ and $2^{-2m+R_{m,2}}|L'_{m,d}| \to 0$ as $m\to\infty$, Chebyshev's inequality gives the desired result.
\end{proof}

\section{Additional simulation results for Subsection \ref{subsec:expr1}}

We repeated the experiment in Subsection \ref{subsec:expr1} on the eight-dimensional integrand $f(\bsx)=\prod_{j=1}^8 x_j\exp(x_j)$. 

Figure~\ref{fig:eightPr} verifies the convergence of $\Pr(\hat{\mu}_{m,E}>\mu)$ to $1/2$ for $m=1,\dots,18$. Each probability is estimated from $8\times10^4$ replicates with precision $E=32$. Convergence is markedly slower than in the one-dimensional case, with complete random designs (CRD) and random linear scrambling (RLS) exhibiting comparable rates.

Figures~\ref{fig:eightlength} and~\ref{fig:eightcover} compare the $90$th percentile interval lengths and empirical coverage levels, respectively. Under CRD, quantile intervals outperform $t$-intervals; under RLS, the opposite holds. This is because $\hat{\mu}_{m,E}-\mu$ under RLS remains approximately normal for the range of $m$ tested. Although our theory predicts that the distribution of $\hat{\mu}_{m,E}-\mu$ becomes concentrated and heavy-tailed asymptotically, the curse of dimensionality delays these effects. At $m=18$, RLS errors exhibit only marginally heavier tails than a normal distribution (Figure~\ref{fig:normality}).

\begin{figure}
    \centering
    \includegraphics[width=0.45\paperwidth]{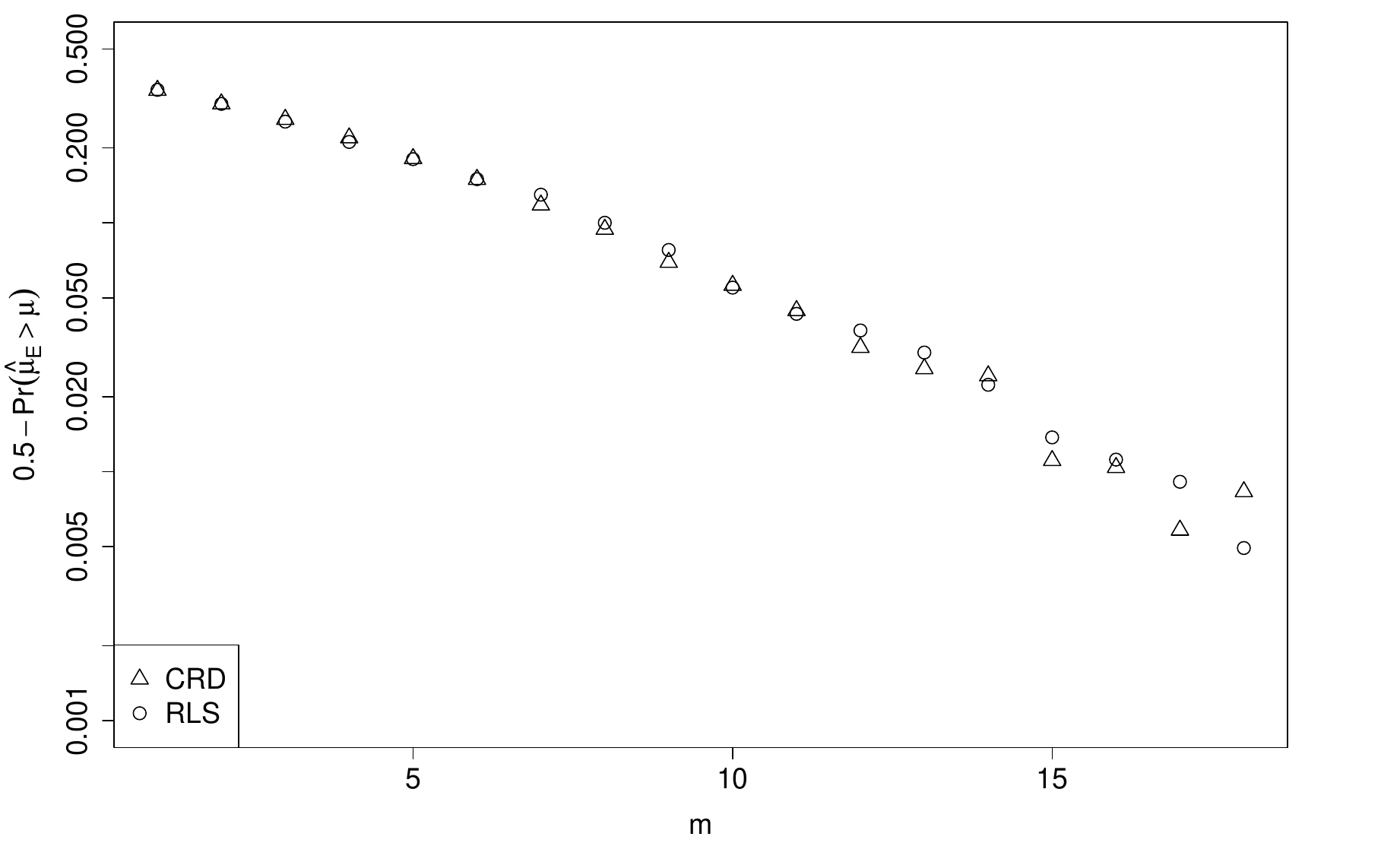}
    \caption{Deviation of $\Pr(\hat{\mu}_{m,E} > \mu)$ from $1/2$ when $f(\bsx)=\prod_{j=1}^8 x_j\exp(x_j)$. }
    \label{fig:eightPr}
\end{figure}

\begin{figure}
    \centering
    \includegraphics[width=0.45\paperwidth]{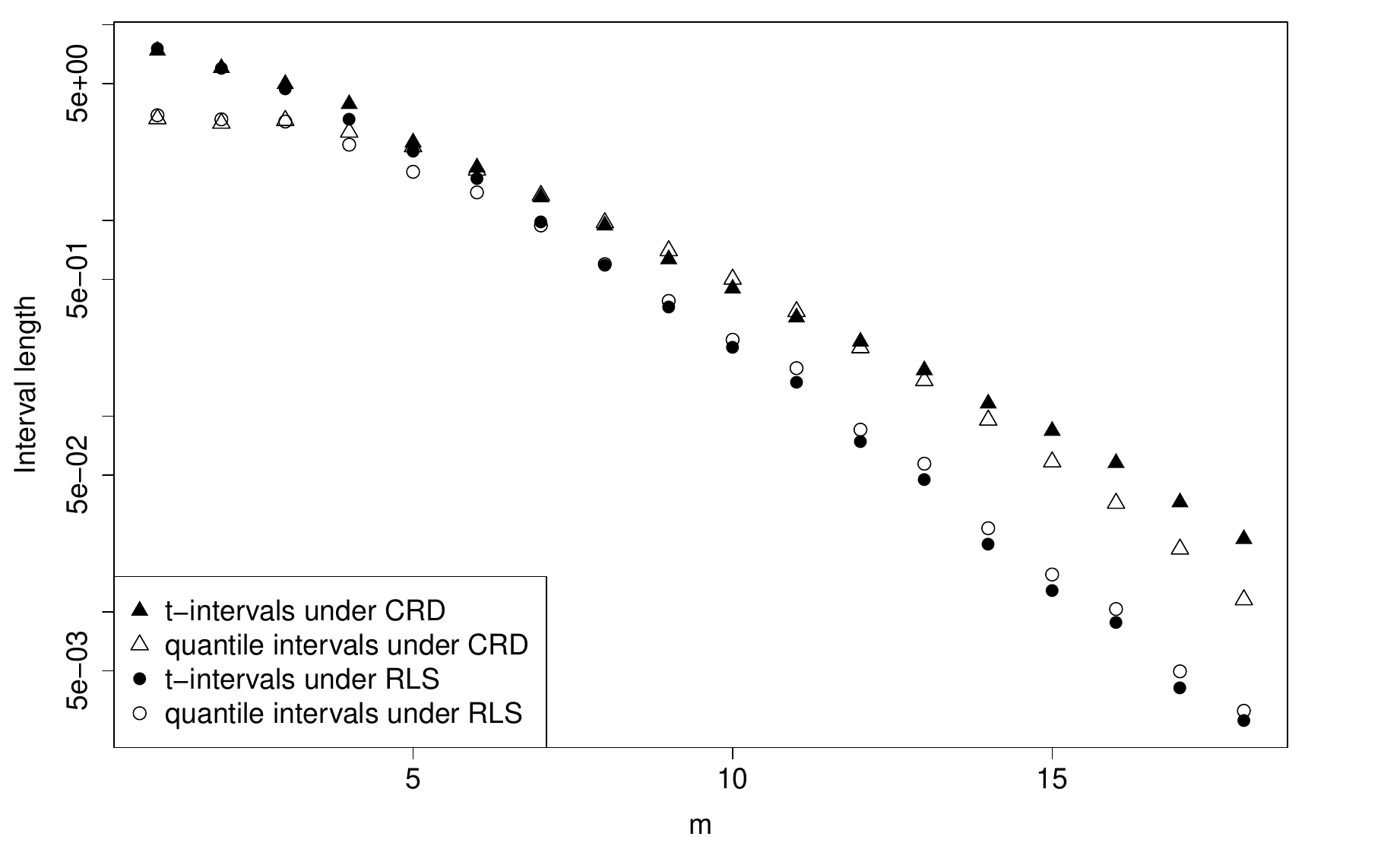}
    \caption{$90$th percentile interval lengths when $f(\bsx)=\prod_{j=1}^8 x_j\exp(x_j)$.}
    \label{fig:eightlength}
\end{figure}

\begin{figure}
    \centering
    \includegraphics[width=0.45\paperwidth]{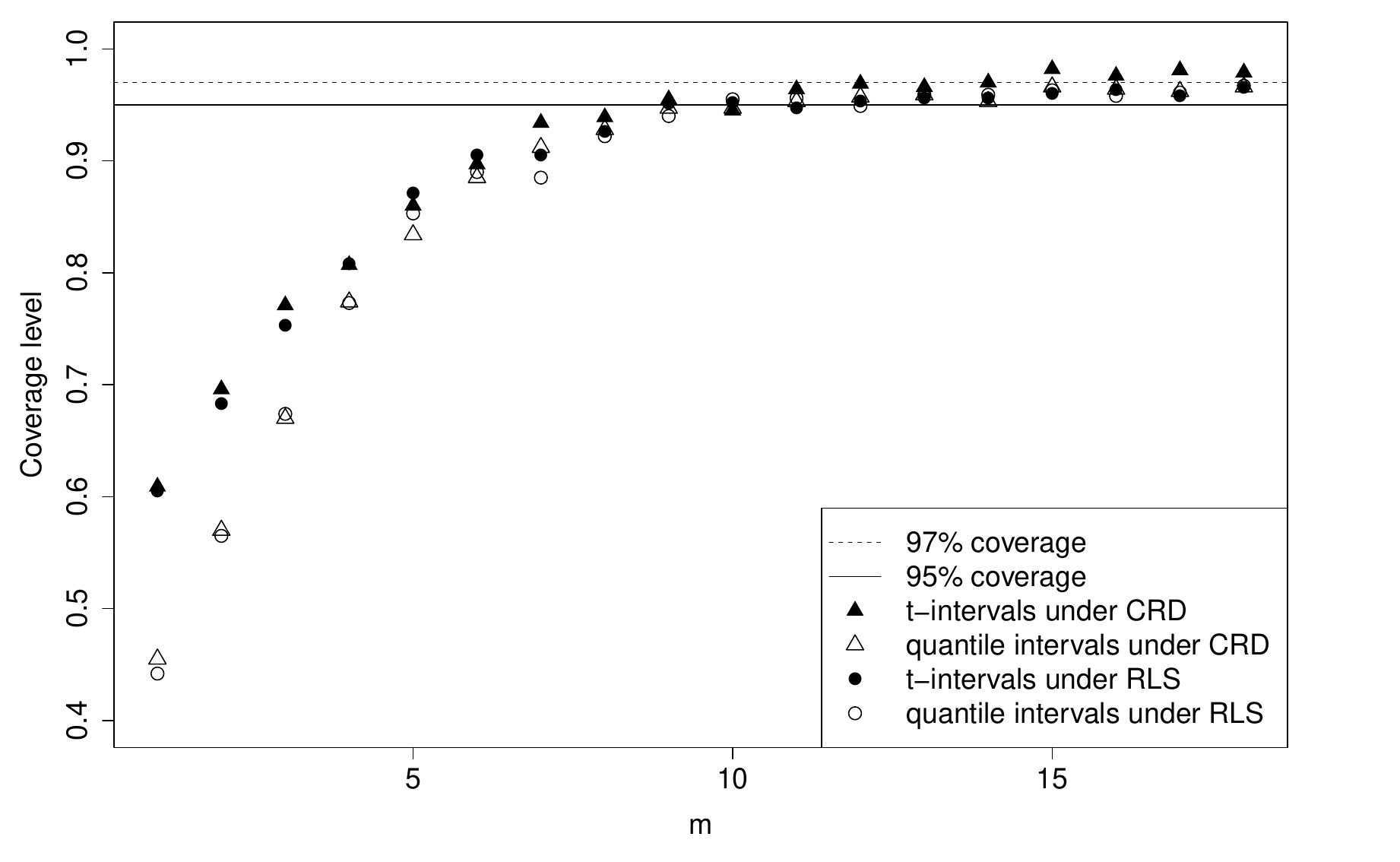}
    \caption{Coverage levels when $f(\bsx)=\prod_{j=1}^8 x_j\exp(x_j)$.}
    \label{fig:eightcover}
\end{figure}

\begin{figure}
    \centering
    \includegraphics[width=0.45\paperwidth]{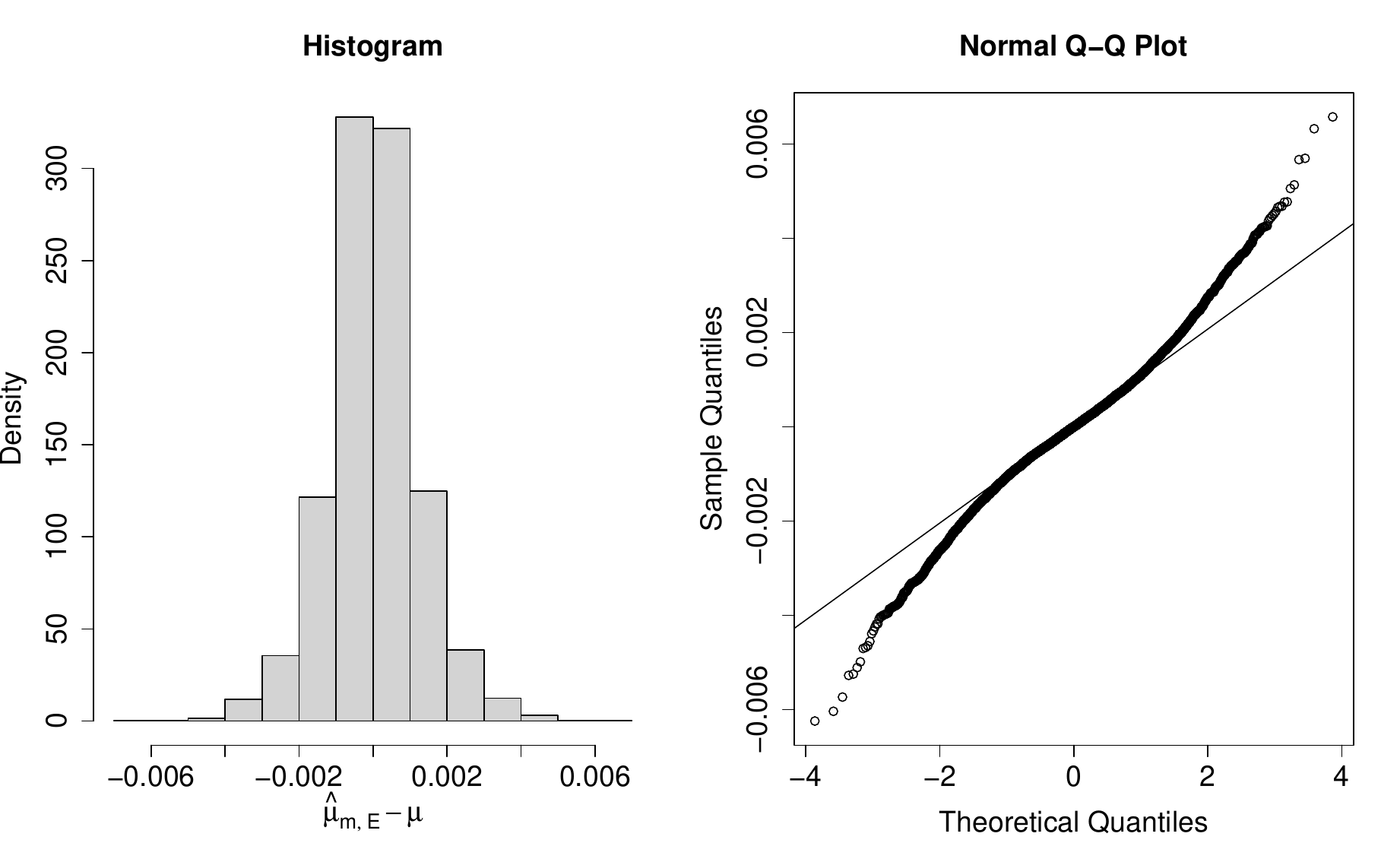}
    \caption{Distribution of errors under RLS when $m=18$ and $f(\bsx)=\prod_{j=1}^8 x_j\exp(x_j)$.}
    \label{fig:normality}
\end{figure}

\section{Additional simulation results for Subsection \ref{subsec:expr2}}

We tested the empirical Bernstein confidence intervals (EBCI) introduced in \cite{jain2025empirical} on the Robot Arm function $f(\bsx)$ from Subsection \ref{subsec:expr2}. Since $f(\bsx)$ measures the Euclidean distance from the origin after four displacements in $\real^2$, it takes values in $[0,4L_{\max}]=[0,8]$. To set the parameters $\lambda_i$ in \cite[equation (7)]{jain2025empirical}, we followed \cite[equation (8)]{jain2025empirical} with $c=0.5$.

For a fair comparison with the methods in Subsection \ref{subsec:expr2}, we fixed the total number of function evaluations at $N=9\times 2^{16}$ and matched the nominal coverage level. We tested $\hat{\mu}_{m,E}$ with $m=2,3,4,5$ and $E=32$, corresponding to per‑replicate sample sizes $n=4,8,16,32$. The number of replicates was then $N/n$, and for each $n$ we constructed $100$ intervals.

Figure~\ref{fig:EBCI} reports the results. The configuration with $n=8$ performs best overall, yielding interval lengths between $5\times10^{-3}$ and $7\times10^{-5}$ and covering the true integral in all $100$ trials. For comparison, the Monte Carlo variance $\sigma^2$ is approximately $1.8$ (estimated from $N$ independent uniform draws from $[0,1]^8$). A standard normal interval with the same nominal coverage $1-\alpha\approx 96.1\%$ would therefore have length $2\Phi^{-1}(1-\alpha/2)(\sigma/\sqrt{N})\approx 7.2\times10^{-3}$, where $\Phi^{-1}$ denotes the normal quantile function. Thus, the best EBCI marginally outperforms the plain Monte Carlo interval.

\begin{figure}
    \centering
    \includegraphics[width=0.45\paperwidth]{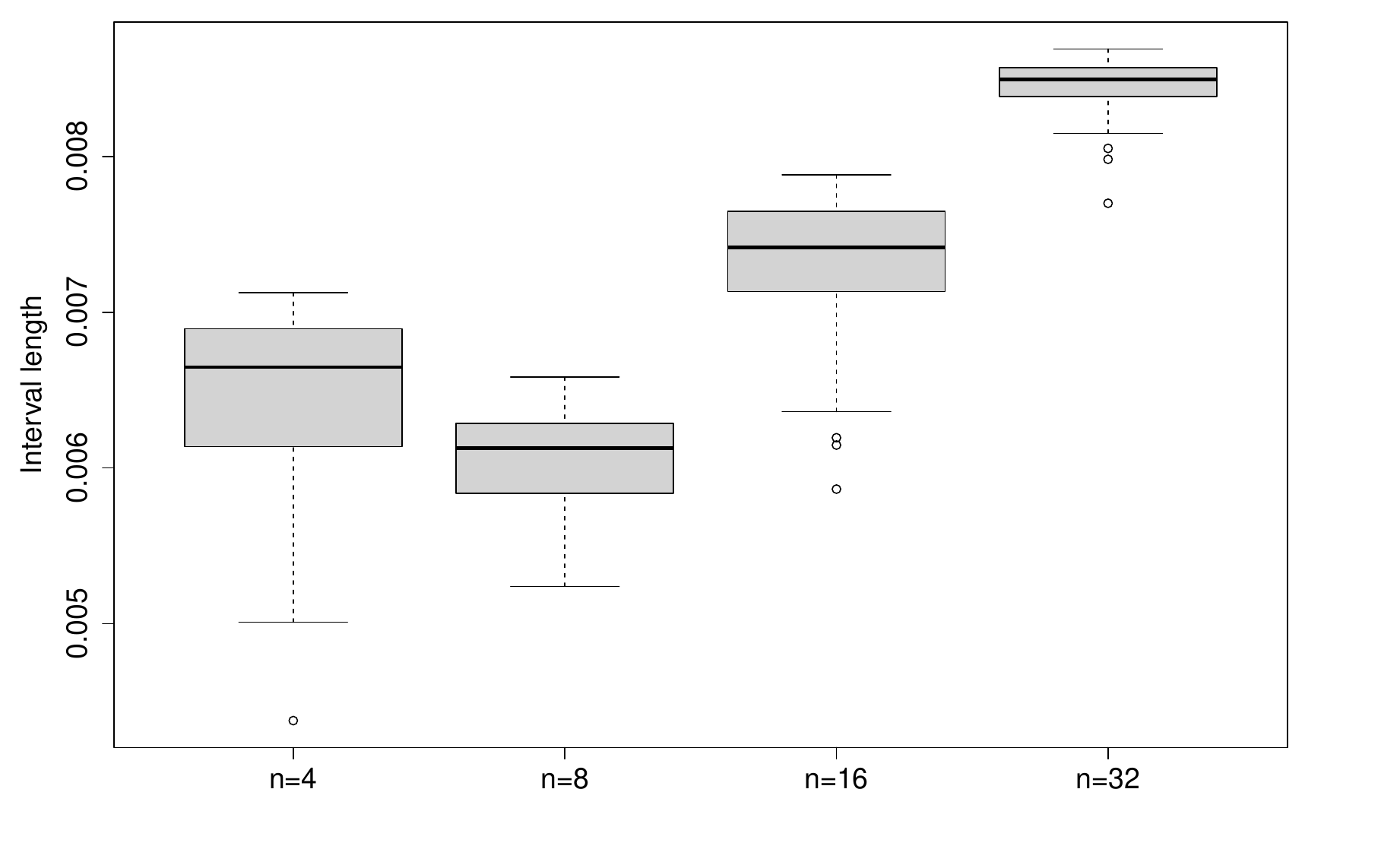}
    \caption{Distribution of interval lengths of EBCI when  $f(\bsx)$ is the Robot Arm function.}
    \label{fig:EBCI}
\end{figure}

\bibliographystyle{abbrv} 
\bibliography{qmc}

\end{document}